\documentclass[acmsmall,authorversion,nonacm,screen]{acmart}
\pdfoutput=1
\usepackage{enumitem}
\usepackage[noend,boxed]{algorithm2e} % algorithms
\usepackage{mathrsfs} 
\usepackage{tikz}  
\usepackage{pgf} 
\usetikzlibrary{arrows,positioning,shapes.geometric,calc, positioning, hobby,fit}    
\newcommand{\boldf}[1]{\bar{#1}}
\newcommand{\struct}[1]{\mathbf{#1}}  

\usepackage{relsize}
\newcommand{\of}[1]{\raisebox{0.5pt}{\textup{{\relsize{-2}[}\kern-0.025em{\relsize{-1}\raisebox{-0.5pt}{$#1$}}\kern-0.025em{\relsize{-2}]}\kern0.025em}}}
 
\newcommand{\ar}{\mathrm{ar}}
\DeclareMathOperator{\Min}{\mathrm{min}}  
\DeclareMathOperator{\Max}{\mathrm{max}}   

\newcommand{\Ideal}[2]{{\downarrow}_{#1}\llbracket #2 \rrbracket} 
\newcommand{\Filter}[2]{{\uparrow}_{#1}\llbracket #2 \rrbracket} 
 
 \newcommand{\supp}{\mathrm{supp}} 
   
\newcommand{\MS}{\ensuremath{\chi_{\boldf{0}}}} 
\newcommand{\Minset}{\mathrm{argmin}} 
\newcommand{\Minsystem}{\mathrm{SMS}} 
  
\DeclareMathOperator{\elel}{\mathrm{ll}}
\DeclareMathOperator{\dual}{\mathrm{dual}}
\DeclareMathOperator{\med}{\mathrm{twin}} 
\DeclareMathOperator{\mx}{\mathrm{mx}}
\DeclareMathOperator{\mi}{\mathrm{mi}}
\DeclareMathOperator{\pp}{\mathrm{pp}}
\DeclareMathOperator{\Aut}{Aut}
\DeclareMathOperator{\End}{End}   

\newcommand{\ifp}{\mathsf{ifp}}
\newcommand{\rk}{\mathsf{rk}}
\newcommand{\dfp}{\mathsf{dfp}}

\DeclareMathOperator{\pr}{proj}    

\DeclareMathOperator{\con}{ctrn}    
 
\DeclareMathOperator{\lex}{\mathrm{lex}}
  
\DeclareMathOperator{\Pol}{Pol}
\DeclareMathOperator{\CSP}{CSP}
\DeclareMathOperator{\LFP}{LFP}
\DeclareMathOperator{\FP}{FP}
\DeclareMathOperator{\FPC}{FPC}
\DeclareMathOperator{\FPR}{FPR}
\DeclareMathOperator{\Datalog}{Datalog} 
\DeclareMathOperator{\FO}{FO}  
\DeclareMathOperator{\FOC}{FOC}  
\DeclareMathOperator{\IFP}{IFP}

\DeclareMathOperator{\Ifp}{\mathsf{Ifp}}
\DeclareMathOperator{\Dfp}{\mathsf{Dfp}}
\DeclareMathOperator{\Op}{\mathsf{Op}}
\DeclareMathOperator{\Mat}{\mathsf{Mat}}

\newcommand{\blue}[1]{{\color{black}#1}}
\newcommand{\red}[1]{{\color{black}#1}}
\newcommand{\teal}[1]{{\color{black}#1}}
  
\theoremstyle{plain}
\newtheorem{theorem}{Theorem}[section]
\newtheorem{conjecture}[theorem]{Conjecture}
\newtheorem{proposition}[theorem]{Proposition}
\newtheorem{lemma}[theorem]{Lemma}
\newtheorem{corollary}[theorem]{Corollary}
\newtheorem{claim}[theorem]{Claim}  
 
\theoremstyle{definition}  
\newtheorem{definition}[theorem]{Definition}
\theoremstyle{remark}
\newtheorem{remark}[theorem]{Remark}
\newtheorem{example}[theorem]{Example}
 \usepackage{titletoc}

 \titlecontents{section}
 [6mm]
 {}
 {\small\sffamily\contentslabel{6mm}\small\sffamily\uppercase}
 {\hspace*{-6mm}}
 {\small\sffamily\titlerule*[1pc]{.}\contentspage}[%\medskip
 ] 
 
  \titlecontents{subsection}
 [12mm]
 {\vspace{-0.5mm}}
 {\small\sffamily\contentslabel{12mm}\small\sffamily}
 {\hspace*{-6mm}}
 {\small\sffamily\titlerule*[1pc]{.}\contentspage}[%\medskip
 ]

\begin{document} 
	
 \let\oldaddcontentsline\addcontentsline% Store \addcontentsline
 \renewcommand{\addcontentsline}[3]{}% Make \addcontentsline a no-op
	
	\title[On the Descriptive Complexity of Temporal Constraint Satisfaction Problems]{On the Descriptive Complexity of Temporal Constraint Satisfaction Problems}       
	\titlenote{An extended version of an article  which appeared in LICS 2020 \cite{bodirsky2020temporal}}

	\author{Manuel Bodirsky}
	\affiliation{% 
		\institution{TU Dresden}
    	\country{Germany}
    	\city{Dresden}
    }
	\authornote{Manuel Bodirsky was supported by the European Research Council (ERC) under the European Union's Horizon 2020 research and innovation programme (Grant Agreement No 681988, CSP-Infinity).} 
	\email{manuel.bodirsky@tu-dresden.de} 
	
	\author{Jakub Rydval}
	\affiliation{% 
		\institution{TU Dresden}
       \country{Germany}	
           	\city{Dresden}
}
		\authornote{Jakub Rydval was supported by the DFG Graduiertenkolleg 1763 (QuantLA)}
	\email{jakub.rydval@tu-dresden.de}

	%%\renewcommand{\shortauthors}{Trovato and Tobin, et al.}
	
	\begin{abstract}  
		Finite-domain constraint satisfaction problems are either
		solvable by Datalog, or not even expressible
		in fixed-point logic with counting.
		The border between the two regimes can be described by a  universal-algebraic minor condition.
		For infinite-domain CSPs, the situation is more complicated even
		if the template structure of the CSP is model-theoretically tame.
		We prove that there is no Maltsev condition that characterizes
		Datalog already for the CSPs of
		first-order reducts of
		$({\mathbb Q};<)$; such CSPs are
		called \emph{temporal CSPs} and
		are of fundamental importance in
		infinite-domain constraint satisfaction.
		Our main result is a complete classification of temporal CSPs
		that can be expressed in one of the following logical formalisms: Datalog, fixed-point logic (with or without counting), or fixed-point logic with the \teal{mod-$2$} rank operator. The classification shows that many of
		the equivalent conditions in the finite fail to capture expressibility in Datalog or fixed-point logic already for temporal CSPs. 
	\end{abstract}	

 	  \begin{teaserfigure}    
    \includegraphics[width=\textwidth]{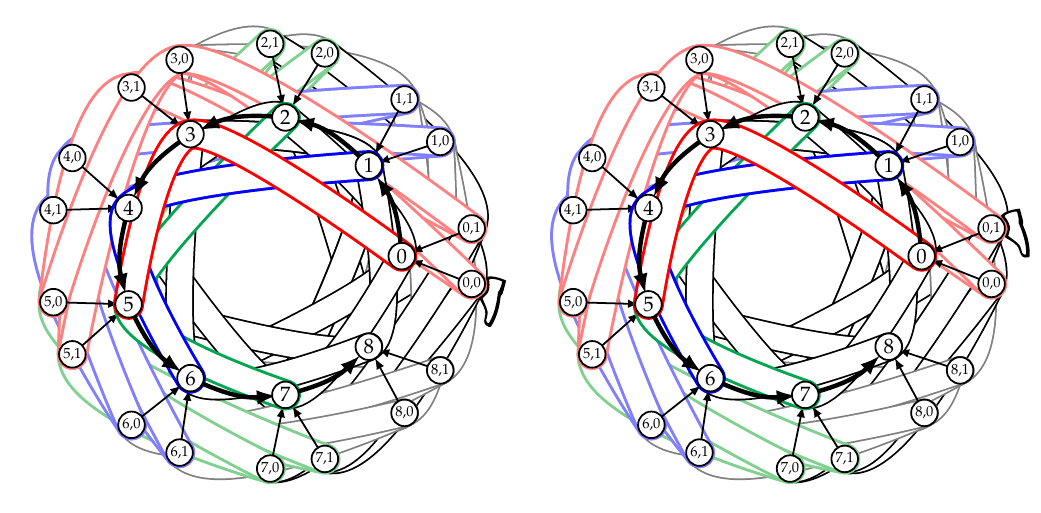}
 	\caption{\label{multipede}A pair of non-isomorphic expansions of an odd $2$-meager  multipede by a constant (shoe).} 
 \end{teaserfigure}
	 \maketitle
	 
 \renewcommand{\contentsname}{\raggedright\bfseries\textsf{CONTENTS}}
	   
\tableofcontents

\DontPrintSemicolon
\let\addcontentsline\oldaddcontentsline

\vspace{-3em}

\section{Introduction}   

The quest for finding a logic capturing Ptime is an ongoing challenge in the field of finite model theory originally motivated by questions from database theory \cite{grohe2008quest}. 
Ever since its proposal, most candidates are based on various extensions of \emph{fixed-point logic} ($\FP$), for example by \emph{counting} or by \emph{rank operators}. 
Though not a candidate for capturing Ptime, \emph{Datalog} is perhaps the most studied fragment of $\FP$. 
Datalog is particularly well-suited for formulating various algorithms
for solving \emph{constraint satisfaction problems} (CSPs); examples of famous
algorithms that can be formulated in Datalog are the \emph{arc consistency}
procedure and the \emph{path consistency} procedure.
In general, the expressive power of  $\FP$ is limited as it fails to express counting properties of finite structures such as even cardinality.
However, the combination
of a mechanism for iteration and a mechanism for counting
provided by \emph{fixed-point logic with counting} (FPC) is strong
enough to express most known algorithmic techniques leading
to polynomial-time procedures \cite{dawar2015nature,gradel2019rank}. In fact, all known
decision problems for finite structures that provably separate
FPC from Ptime are at least as hard as deciding solvability of
systems of equations over a \blue{fixed} non-trivial finite Abelian
group \cite{pakusa2015linear}. If we extend $\FPC$ further by the \emph{\blue{mod-2} rank operator}~\cite{gradel2019rank}, we
obtain the logic $\FPR_2$ which is known to capture Ptime for CSPs \blue{of two-element structures}~\cite{wang2018descriptive}.
\blue{Extending $\FPC$ by a single rank operator modulo a fixed prime number is not sufficient for expressing the solvability of equations modulo a different prime number~\cite{gradel2019rank,GradelGPP19}. 
	Instead, one typically considers the extension $\FPR$ by rank operators modulo every prime number.
	This logic is currently one of the two leading candidates for capturing Ptime for finite-domain CSPs, the other being \emph{choiceless polynomial time} (CPT).
	Outside of the scope of CSPs, the logic $\FPR$ has already been eliminated as a candidate and replaced with the more expressive extension $\FPR^{*}$ by the \emph{uniform rank operator}~\cite{gradel2019rank}. 
	It has recently been announced~\cite{lichter2021separating} that the satisfiability of mod-$2^i$  equations where $i$ is a part of the input, a problem that is clearly in Ptime, is not even expressible in $\FPR^{*}$.
	However, the results in~\cite{lichter2021separating} have no consequences for finite-domain CSPs \red{in the standard setting of structures with a \emph{finite} relational signature}.}
%, which is the standard setting. 
%because, by definition, CSPs only allow finitely many different constraint relations.}

%Proving inexpressibility results for $\FPR_2$ seems to be very difficult.
%
The first inexpressibility result for $\FPC$ is due to Cai, F\"urer, and Immerman for systems of equations over $\mathbb{Z}_{2}$ \cite{cai1992optimal}. In 2009, this result was extended to arbitrary non-trivial finite Abelian groups by Atserias, Bulatov, and Dawar \cite{atserias2009affine}; their work was formulated purely in the framework of CSPs. 
At around the same time, Barto and Kozik \cite{barto2014constraint} settled the closely related bounded width conjecture of Larose and Z\'adori \cite{larose2007bounded}.
%
% A-B-D: omitting types 1,2 = variety contains idempotent reduct of a module. 
% Maroti: omitting types mentioned, meet semi-distributive iff WNU(n) for all n
A combination of both works together with results from \cite{maroti2008existence,kozik2015characterizations} yields the following theorem.
\begin{theorem}\label{finitedomainsituation} For a finite structure $\struct{B}$, the following seven statements are equivalent.  
\begin{enumerate} 
	\item $\CSP(\struct{B})$ is expressible in $\Datalog$~\textup{\cite{barto2014constraint}}.
	\item $\CSP(\struct{B})$ is expressible in $\FP$~\textup{\cite{atserias2009affine}}.
	\item $\CSP(\struct{B})$ is expressible in $\FPC$~\textup{\cite{atserias2009affine}}. 
	\item $\struct{B}$ does not pp-construct equations over any non-trivial finite Abelian group~\textup{\cite{larose2007bounded,AbilityToCount}}.
	\item $\struct{B}$ does not pp-construct equations over $\mathbb{Z}_p$ for any prime $p\geq 2$ \textup{\cite{barto2017polymorphisms,zhuk2021strong}}.
	\item $\struct{B}$ has weak near-unanimity polymorphisms for all but finitely many arities~\textup{\cite{maroti2008existence}}.
	\item  $\struct{B}$ has weak near-unanimity polymorphisms $f, g$ that satisfy  $	g(x, x, y) \approx f(x, x, x, y)$~\textup{\cite{kozik2015characterizations}}.
	\item \blue{$\struct{B}$ has $(n+3)$-polymorphisms for some $n$~\textup{\cite{olvsak2019generalizing}}.}
\end{enumerate} 
\end{theorem} 
\noindent In particular, $\Datalog$, $\FP$, and $\FPC$  are equally expressive when it comes to finite-domain CSPs.
This observation raises the question whether the above-mentioned fragments and extensions of FP might collapse on CSPs in general.
In fact, this question was already answered negatively in 2007 by Bodirsky and K\'ara in their investigation of the CSPs of first-order reducts of ${(\mathbb{Q};<)}$, also known as (infinite-domain) \emph{temporal CSPs} \cite{bodirsky2010complexity}; the decision problem $\CSP(\mathbb{Q};\mathrm{R}_{\Min})$, where  
\[ \mathrm{R}_{\Min}\coloneqq\{(x,y,z)\in \mathbb{Q}^{3} \mid y<x\vee z<x \}, \]
is provably not solvable by any $\Datalog$ program \cite{bodirsky2010fast} but it is expressible in FP, as we will see later. 
Since every CSP represents a class of finite structures whose complement is closed under homomorphisms, this simultaneously yields an alternative proof of a result from \cite{dawar2008datalog} stating that the homomorphism preservation theorem fails for FP.

Several famous NP-hard problems such as the \emph{Betweenness} problem
or the \emph{Cyclic Ordering} problem are temporal CSPs.
Temporal CSPs have been studied for example in artificial intelligence  \cite{nebel1995reasoning}, scheduling \cite{bodirsky2010fast}, and approximation \cite{guruswami2011beating}.
Random instances of temporal CSPs have been studied in \cite{goerdt2009random}.
Temporal CSPs fall into the larger class of CSPs of \emph{reducts of finitely bounded homogeneous structures}. It is an open problem whether all CSPs of reducts of finitely bounded homogeneous structures have a complexity dichotomy in the sense that they are in P or NP-complete (\blue{Conjecture~\ref{conjecture}}).
In this class, temporal CSPs play a particular role since they are among the few known cases where the important technique of reducing infinite-domain CSPs to finite-domain CSPs from \cite{bodirsky2016reducts} fails to provide any polynomial-time tractability results.

\subsection{\blue{Contributions}}
We present a complete classification of temporal CSPs that can be solved in $\Datalog$, $\FP$, $\FPC$, or $\FPR_{2}$.
The classification leads to the following sequence of inclusions for temporal CSPs:
$$\Datalog \subsetneq \FP = \FPC   \subsetneq \FPR_{2}.$$
Our results show that the expressibility of temporal CSPs in these logics can be characterised in terms of avoiding pp-constructibility of certain structures, namely $(\mathbb{Q};\mathrm{R}_{\Min})$, $(\mathbb{Q};\mathrm{X})$ where \[\mathrm{X}\coloneqq   \{(x,y,z)\in \mathbb{Q}^{3}\mid  x=y<z \vee y=z<x\vee z=x<y\},\]
and $(\{0,1\};1\textup{IN}3)$ where \[
1\textup{IN}3\coloneqq \{(1,0,0),(0,1,0),(0,0,1)\}. \]

\begin{theorem} \label{datalogmainresult} Let $\struct{B}$ be a temporal structure. The following are equivalent: 
\begin{enumerate} 
	\item $\CSP(\struct{B})$ is expressible in $\Datalog$.
	\item  $\struct{B}$ does not pp-construct $(\{0,1\};1\textup{IN}3)$ and  $(\mathbb{Q};\mathrm{R}_{\Min})$. 
	\item $\struct{B}$ is preserved by $\elel$ and $\dual\elel$, or by a constant operation.    
\end{enumerate} 
\end{theorem} 
\begin{samepage} 
\begin{theorem} \label{mainresult} Let $\struct{B}$ be a temporal structure. The following are equivalent:
	\begin{enumerate}
		\item $\CSP(\struct{B})$ is expressible in $\FP$.
		\item $\CSP(\struct{B})$ is expressible in $\FPC$.
		\item $\struct{B}$ does not pp-construct $(\{0,1\};1\textup{IN}3)$ and $(\mathbb{Q};\mathrm{X})$.  
		\item $\struct{B}$ is preserved by $\Min$, $\mi$, $\elel$, the dual of one of these operations, or by a constant operation.
	\end{enumerate} 
\end{theorem} 
\end{samepage}
\begin{theorem} \label{FPRclassification} 
Let $\struct{B}$ be a temporal structure. The following are equivalent:
\begin{enumerate} 
	\item $\CSP(\struct{B})$ is expressible in $\FPR_2$.
	\item $\struct{B}$ does not pp-construct $(\{0,1\};1\textup{IN}3)$.
	\item $\struct{B}$ is preserved by $\mx,$ $\Min$, $\mi$, $\elel$, the dual of one of these operations, or by a constant operation.
\end{enumerate} 
\end{theorem} 
As a byproduct of our classification we get that all polynomial-time algorithms for temporal CSPs from \cite{bodirsky2010complexity} can be implemented in $\FPR_{2}$. 
Our results also show that every temporal CSP of a structure that pp-constructs $(\mathbb{Q};\mathrm{X})$ but not $(\{0,1\};1\textup{IN}3)$
is solvable in polynomial time, is not expressible in $\FPC$, and cannot encode systems of equations over any non-trivial finite Abelian group. 
Such temporal CSPs are \blue{equivalent to the following decision problem \teal{under} Datalog-reductions}:

\vbox{\medskip
\noindent \textbf{\textup{3-Ord-Xor-Sat}}  

\smallskip
\noindent  INPUT:  A finite homogeneous system of \blue{mod-2} equations of length $3$.

\noindent QUESTION:  Does every non-empty subset $E$ of the equations have a solution where at least one variable  in an equation from $E$ denotes the value $1$?
\medskip} 

\teal{As we will see, there exists a straightforward FP-reduction from 3-Ord-Xor-Sat to satisfiability of mod-$2$ equations.
However, it is unclear whether there exists an FP-reduction in the opposite direction.
In our inexpressibility result for 3-Ord-Xor-Sat, we use a Datalog-reduction from satisfiability of mod-$2$ equations restricted to those systems which have at most one solution.}
We have eliminated the following candidates for general algebraic criteria for expressibility of CSPs in FP  motivated by the articles \cite{atserias2009affine}, \cite{bodirsky2016reducts}, and \cite{barto2014constraint}, respectively.
\blue{\begin{theorem}   $\CSP(\mathbb{Q};\mathrm{X})$ is inexpressible in FPC, but
	\begin{enumerate} 
		\item $(\mathbb{Q};\mathrm{X})$ does not pp-construct equations over any non-trivial finite Abelian group,
		\item $(\mathbb{Q};\mathrm{X})$ has pseudo-WNU polymorphisms $f,g$ that satisfy  $ g(x, x, y) \approx f(x, x, x, y)$,
		\item $(\mathbb{Q};\mathrm{X})$ has a $k$-ary pseudo-WNU polymorphism for all but finitely many $k\in \mathbb{N}$.
	\end{enumerate}
\end{theorem}} 
% %\begin{itemize}
% 	 the inability to pp-construct equations over a non-trivial finite Abelian group  \cite{atserias2009affine},
% 	 the 3-4 equation for weak near-unanimity polymorphisms modulo outer endomorphisms \cite{bodirsky2016reducts},
% 	 and the existence of weak near-unanimity polymorphisms modulo outer endomorphisms for all but finitely many arities \cite{barto2014constraint}.
%% \end{itemize} 
% % 
We have good news and bad news regarding the existence of general algebraic criteria for expressibility of CSPs in fragments and/or extensions of FP. 
The bad news is that there is no Maltsev condition that would capture expressibility of temporal CSPs in Datalog (see Theorem~\ref{topology})
which carries over to CSPs of reducts of finitely
bounded homogeneous structures and more generally to CSPs of \emph{$\omega$-categorical} templates.
\begin{theorem} \label{topology} There is no condition preserved by uniformly continuous clone homomorphisms that would capture the expressibility of temporal CSPs in $\Datalog$.
\end{theorem}

This is particularly striking because $\omega$-categorical CSPs are otherwise well-behaved when it comes to expressibility in $\Datalog$\textemdash every $\omega$-categorical $\CSP$ expressible in $\Datalog$ admits a \emph{canonical Datalog program} \cite{bodirsky2013datalog}.
\blue{The good news is that the expressibility in FP for finite-domain and temporal CSPs can be characterised by universal-algebraic minor conditions.
We introduce a family $\mathcal{E}_{k,n}$ of minor conditions that are similar to the \emph{dissected weak near-unanimity} identities from \cite{barto2019equations,gillibert2022symmetries}.}
\begin{theorem} \label{alternativefp} Let $\struct{B}$ be a finite structure or a temporal structure. The following are equivalent.
\begin{enumerate} 
\item $\CSP(\struct{B})$ is expressible in $\FP$ / $\FPC$. 
\item $\Pol(\struct{B})$ satisfies $\mathcal{E}_{k,k+1}$ for all but finitely many $k\in \mathbb{N}$.  
\end{enumerate}
\end{theorem}  
The polymorphism clone of
every first-order reduct of a finitely bounded homogeneous structure
known to the authors satisfies $\mathcal{E}_{k,k+1}$
for all but finitely many $k$ if and only if its CSP is in $\FP$ / $\FPC$.
This includes in particular all CSPs that are in the complexity class AC$_0$:
all of these CSPs can be expressed as CSPs of reducts of finitely bounded homogeneous structures, by a combination of results
of Rossman~\cite{Rossman08}, Cherlin, Shelah, and Shi~\cite{cherlin1999universal}, and \blue{Hubi\v{c}ka} and Ne\v{s}et\v{r}il~\cite{hubivcka2016homomorphism} (see~\cite{bodirsky2021complexity}, Section~5.6.1), and their polymorphism clones satisfy $\mathcal{E}_{k,k+1}$
for all but finitely many $k$. 
To prove that the polymorphism clone of a given temporal structure does or does not satisfy $\mathcal{E}_{k,k+1}$ we apply a new general characterisation of the satisfaction of minor conditions in polymorphism clones of $\omega$-categorical structures (Theorem~\ref{testing}). 

\subsection{\blue{Outline of the article}}
In Section~\ref{section_preliminaries}, we introduce various basic concepts from algebra and logic as well as some specific ones for temporal CSPs.
In Section~\ref{section_tcsps_in_fp}, we start discussing the descriptive complexity of temporal CSPs by expressing some particularly chosen tractable temporal CSPs in FP.
In Section~\ref{section_inexpressibility},  we continue the discussion by showing that $\CSP(\mathbb{Q};\mathrm{X})$ is inexpressible
in $\FPC$ but
expressible in $\FPR_2$. 
At this point we have enough information so that in Section~\ref{section_classification} we can 
classify the temporal CSPs which are expressible in $\FP$ / $\FPC$ and the temporal CSP which are expressible in $\FPR_2$.
% using an algebraic argument. 
%
In Section~\ref{section_datalog} we 
classify the temporal CSPs which are expressible in $\Datalog$.
%, again, using a purely algebraic argument. 
%
In Section~\ref{section_algebraic_conditions} we provide  results regarding  algebraic criteria for expressibility of finitely bounded homogeneous CSPs in $\Datalog$ and in $\FP$ based on our investigation of temporal CSPs.

%%%%%%%%%%%%%%%%%%%% Section 2

\section{Preliminaries}\label{section_preliminaries}
%
%We need various notions from algebra and logic.
% 
%E.g., our main inexpressibility result, Theorem~\ref{FPxorsat}, relies on the method of model-theoretic games as well as the framework of logical interpretations and reductions from \cite{atserias2009affine}.
%  
The set $\{1,\dots,n\}$ is denoted by $[n]$.
\red{The set of rational numbers is denoted by ${\mathbb Q}$, and the set of positive rational numbers by ${\mathbb Q}_{> 0}$.}

We use the bar notation for tuples; for a tuple ${\boldf{t}}$ indexed by a set $I$, the value of ${\boldf{t}}$ at the position $i\in I$ is denoted by  $\boldf{t}\of{i}$.
%  
%For a $k$-ary tuple $\boldf{t}$ and $I\subseteq [k]$, we use the notation $\pr_{I}(\boldf{t})$ for the tuple $(\boldf{t}\of{i_1},\dots, \boldf{t}\of{i_{\ell}})$ where $I=\{i_1,\dots, i_{\ell}\}$ with $i_1<\dots< i_{\ell}$.
%
For a function $f\colon A^n \rightarrow B$ ($n\geq 1$) and $k$-tuples $\boldf{t}_1,\ldots,\boldf{t}_n \in A^{k}$,   
we sometimes use $f(\boldf{t}_1,\ldots,\boldf{t}_n)$ as a shortcut for the $k$-tuple 
$(f(\boldf{t}_1\of{1},\ldots\boldf{t}_n\of{1}),\dots, f(\boldf{t}_1\of{k},\ldots,\boldf{t}_n\of{k}) )$.
\blue{This is usually called the \emph{component-wise action} of $f$ on $A^k$~\cite{barto2021algebraic}.}

\subsection{\blue{Structures and first-order logic}}

A (\emph{relational}) \emph{signature} $\tau$ is a set of \emph{relation symbols}, each $R\in\tau$ with an associated natural number $\ar(R)$ called \emph{arity}.
A (\emph{relational}) \emph{$\tau$-structure} $\struct{A}$ consists of a set $A$ (the \emph{domain}) together with the relations $R^{\struct{A}}\subseteq A^{k}$ for each relation symbol $R\in \tau$ with arity $k$.
We often describe structures by listing their domain and relations, that is, we write $\struct{A}=(A;R^{\struct{A}}_{1},\dots)$.
\blue{We sometimes identify relation symbols with the relations interpreting them, but only when it improves readability of the text.}
An \emph{expansion} of $\struct{A}$ is a $\sigma$-structure $ \struct{B}$ with $A=B$ such that $ \tau\subseteq \sigma$, $R^{\struct{B}}=R^{\struct{A}}$ for each relation symbol $R\in \tau$. Conversely, we call $\struct{A}$ a \emph{reduct} of $\struct{B}$.  
\blue{We write $(\struct{A},R)$ for the expansion of $\struct{A}$ by the relation $R$ over $A$.}
In the context of relational structures, we reserve the notion of a \emph{constant} for singleton unary relations. A \emph{constant symbol} is then a symbol of such a relation.

A \emph{homomorphism} $h\colon \struct{A} \rightarrow \struct{B}$ for $\tau$-structures $\struct{A},\struct{B}$ is a mapping $h\colon  A\rightarrow B$ that \emph{preserves} each relation of $\struct{A}$, that is, if $ \boldf{t} \in R^{\struct{A}}$ for some $k$-ary relation symbol $R\in \tau$, then $h(\boldf{t})\in R^{\struct{B}}$.
We write $\struct{A} \rightarrow \struct{B}$ if $\struct{A}$ maps homomorphically to $\struct{B}$ and $\struct{A} \centernot{\rightarrow} \struct{B}$ otherwise. 
We say that $\struct{A}$ and $\struct{B}$ are \emph{homomorphically equivalent} if $\struct{A} \rightarrow \struct{B}$ and $\struct{B} \rightarrow \struct{A}$.
An \emph{endomorphism} is a homomorphism from $\struct{A}$ to $\struct{A}$. 
\blue{The set of all endomorphisms of $\struct{A}$ is denoted by $\End(\struct{A})$.}
We call a homomorphism $h\colon \struct{A} \rightarrow \struct{B}$ \emph{strong} if it
additionally satisfies the following condition: for every $k$-ary relation symbol $R\in \tau$ and $\boldf{t}\in A^{k}$ we have $h(\boldf{t})\in R^{\struct{B}}$ only if $\boldf{t}\in R^{\struct{A}}.$
An \emph{embedding} is an injective strong homomorphism.
We write $\struct{A}\hookrightarrow \struct{B}$ if $\struct{A}$ embeds to $\struct{B}$. 
A \emph{substructure} of $\struct{A}$ is a structure $\struct{B}$ over $B\subseteq A$ such that the inclusion map $i\colon B\rightarrow A$ is an embedding.   
An \emph{isomorphism} is a surjective embedding. Two structures $\struct{A}$ and $\struct{B}$ are \emph{isomorphic} if there exists an isomorphism from $\struct{A} $ to $\struct{B}$.  An \emph{automorphism} is an isomorphism from $\struct{A}$ to $\struct{A}$.  
The set of all automorphisms of $\struct{A}$, denoted by $\Aut(\struct{A})$, forms a \emph{permutation group} w.r.t.\ the map composition \cite{hodges1997shorter}.  
The \emph{orbit} of a tuple $\boldf{t}\in A^{k}$ under the \blue{component-wise} action of $\Aut(\struct{A})$ on $A^{k}$ is the set $\{g(\boldf{t}) \mid g\in \Aut(\struct{A})\}.$

An $n$-ary \emph{polymorphism} of a relational structure $\struct{A}$ is a mapping $f\colon A^{n}\rightarrow A$ such that, for every $k$-ary relation symbol $R\in \tau$ and tuples $\boldf{t}_{1},\dots,\boldf{t}_{n}\in R^{\struct{A}}$,  we have 
$  f(\boldf{t}_1,\dots,\boldf{t}_n)\in R^{\struct{A}}$.
We say that $f$ \emph{preserves} $\struct{A}$ to indicate that $f$ is a polymorphism of $\struct{A}$. We might also say that an operation \emph{preserves} a relation $R$ over $A$ if it is a polymorphism of $(A;R)$.  
%  

%\subsection{Model theory}  
%  
%We understand the notion of a \emph{logic} as defined in \cite{grohe2008quest}.
%
We assume that the reader is familiar with classical \emph{first-order} logic (FO); we allow the first-order formulas $x=y$ and $\bot$.  
The \emph{positive quantifier-free} fragment of FO is abbreviated by pqf. 
A first-order $\tau$-formula $\phi$ is \emph{primitive positive} (pp) if it is of the form $\exists x_{1},\dots,x_{m}  (\phi_{1}\wedge \dots \wedge \phi_{n})$, where each $\phi_{i}$ is \emph{atomic}, that is, of the form $\bot$, $x_{i}=x_{j}$, or $R(x_{i_{1}},\dots,x_{i_{\ell}})$ for some $R\in \tau$. 
Note that if $\psi_{1},\dots,\psi_{n}$ are primitive positive formulas, then $\exists  x_{1},\dots,x_{m}  (\psi_{1} \wedge \dots \wedge \psi_{n})$ can be re-written into an equivalent primitive positive formula, so we sometimes treat such formulas as primitive positive formulas as well. 
If $\struct A$ is a $\tau$-structure and $\phi(x_1,\dots,x_n)$ is a $\tau$-formula with free variables $x_1,\dots,x_n$, then the relation 
$\{{\boldf{t}}\in A^{n} \mid \struct{A}\models \phi({\boldf{t}}) \}$ is called the \emph{relation defined by $\phi$ in $\struct A$}, and denoted by $\phi^{\struct{A}}$. 
If $\Theta$ is a set of $\tau$-formulas, we say that an $n$-ary relation has a \emph{$\Theta$-definition} in $\struct{A}$ if it is of the form  $\phi^{\struct{A}}$ for some   $\phi \in \Theta$.
%
%If $\phi(x_1,\dots,x_n)$ defines 
When we work with tuples $\boldf{t}$ in a relation defined
by a formula $\phi(x_1,\dots,x_n)$, then we sometimes refer to the entries of ${\boldf{t}}$
through the free variables of $\phi$, and write 
$\boldf{t}\of{x_i}$ instead of $\boldf{t}\of{i}$. 
\begin{proposition}[e.g.\ \cite{bodirsky2021complexity}]  \label{InvAutPol} Let $\struct{A}$ be a relational structure and $R$ a relation over $A$.  
\begin{enumerate}
\item If $R$ has a first-order definition in $\struct{A}$, then it is preserved by all  automorphisms of $\struct{A}$.
\item If $R$ has a primitive positive definition in $\struct{A}$, then it is preserved by all  polymorphisms of $\struct{A}$.
\end{enumerate} 
\end{proposition}

The main tool for complexity analysis of CSPs is the concept of \emph{primitive positive constructions} (see Theorem~\ref{REDUCTION}).
\begin{definition}[\cite{barto2018wonderland}] Let $\struct{A}$ and $\struct{B}$ be relational structures with signatures $\tau$ and $\sigma$, respectively. 
We say that $\struct{B}$ is a \emph{($d$-dimensional)  pp-power} of $\struct{A}$ if $B = A^d$ for some $d\geq 1$ and, for every $R\in \tau$, the relation 
$
\{({\boldf{t}}_{1}\of{1},\dots,{\boldf{t}}_{1}\of{d}, \dots,{\boldf{t}}_{n}\of{1},\dots,{\boldf{t}}_{n}\of{d}) \in A^{n\cdot d} \mid  (\boldf{t}_{1},\dots,\boldf{t}_{n}) \in R^{\struct{B}} \} 
$
has a pp-definition in $\struct{A}$. 
We say that $\struct{B}$ is \emph{pp-constructible} from $\struct{A}$ if  $\struct{B}$ is  homomorphically equivalent to a pp-power of $\struct{A}$.
If $\struct{B}$ is a $1$-dimensional pp-power of $\struct{A}$, then we say that $\struct{B}$ is \emph{pp-definable} in $\struct{A}$.
\end{definition}
Primitive positive constructibility, seen as a binary relation, is transitive~\cite{barto2018wonderland}.

A structure is $\omega$-\emph{categorical} if its first-order theory has exactly one countable model up to isomorphism.  
The theorem of Engeler, Ryll-Nardzewski, and Svenonius (Theorem~6.3.1 in \cite{hodges1997shorter}) asserts that the following statements are equivalent for a countably infinite
structure $\struct{A}$ with countable signature:
\begin{itemize}
\item $\struct{A}$ is $\omega$-categorical. 
\item Every relation over $A$ preserved by all automorphisms of $\struct{A}$ has a first-order definition in $\struct{A}$.
\item For every $k\geq 1$,  there are finitely many orbits of $k$-tuples under the component-wise action of $\Aut(\struct{A})$ on $A^k$.
\end{itemize} 
% 
%An $\omega$-categorical structure $\struct{A}$ is called a \emph{model-complete core} if every relation with a FO-definition in $\struct{A}$ is preserved by all endomorphisms of $\struct{A}$. 

A structure $\struct{A}$ is \emph{homogeneous} if every isomorphism between finite substructures of $\struct{A}$ extends to an automorphism of $\struct{A}$.
If the signature of $\struct{A}$ is finite, then $\struct{A}$ is homogeneous if and only if $\struct{A}$ is $\omega$-categorical and admits quantifier-elimination \cite{hodges1997shorter}.
A structure $\struct{A}$ is \emph{finitely bounded} if there is a universal first-order sentence $\phi$ such that a finite structure embeds into $\struct{A}$ if and only if it satisfies $\phi$. 
The standard example of a finitely bounded homogeneous structure is $(\mathbb{Q};<)$ \cite{bodirsky2016reducts}. 

\subsection{\blue{Finite variable logics and counting}}\label{fragments}
%
%\begin{definition}[Counting quantifiers, finite variable logics \cite{atserias2019definable,otto2000bounded}]   
% 
%

We denote the fragment of $\FO$ in which every formula uses only the variables $x_1,\dots,x_k$ by $\mathcal{L}^{k}$, and its existential positive fragment by $\exists^{+}\!\!\mathcal{L}^{k}$.
By $\FOC$ we denote the extension of $\FO$ by the counting quantifiers $\exists^{i}$. If  $\struct{A}$ is a $\tau$-structure and $\phi$ a $\tau$-formula with a free variable $x$, then  $\struct{A}\models \exists^{i}x.\,  \phi(x)  $ if and only if there exist $i$ distinct $a\in A$ such that $\struct{A} \models \phi(a)$. 
While $\FOC$ is not more expressive than $\FO$, the presence of counting quantifiers might affect the number of variables that are necessary to define a particular relation.   
The  $k$-variable fragment of $\FOC$ is denoted by $\mathcal{C}^{k}$. 
The infinitary logic $\mathcal{L}_{\infty\omega}^{k}$ extends $\mathcal{L}^{k}$ with infinite disjunctions and conjunctions.
The extension of $\mathcal{L}_{\infty\omega}^{k}$ by the counting quantifiers $\exists^{i}$ is denoted by $\mathcal{C}_{\infty\omega}^{k}$, and    $\mathcal{C}_{\infty\omega}^{\omega}$ stands for $  \bigcup_{k\in \mathbb{N}} \mathcal{C}_{\infty\omega}^{k}$.

We understand the notion of a \emph{logic} as defined in \cite{grohe2008quest}. Given two $\tau$-structures $\struct{A}$ and $\struct{B}$ and a logic $\mathcal{L}$, we write 
$\struct{A} \equiv_{\mathcal{L}} \struct{B}$ to indicate that a $\tau$-sentence from $\mathcal{L}$ holds in $\struct{A}$ if and only if it holds in $\struct{B}$,
and we write $\struct{A} \Rightarrow_{\mathcal{L}} \struct{B}$ to indicate that every $\tau$-sentence from $\mathcal{L}$ which is true in $\struct{A}$ is also true in $\struct{B}$.
\blue{By definition, the relation $\struct{A} \Rightarrow_{\mathcal{L}} \struct{B}$ is reflexive and transitive, and $\struct{A} \equiv_{\mathcal{L}} \struct{B}$ is an equivalence relation.}
\blue{The relations  $\Rightarrow_{\exists^{+}\!\!\mathcal{L}^{k}}$ and $\equiv_{\mathcal{C}^{k}}$ have well-known characterizations
in terms of two-player pebble games; $\Rightarrow_{\exists^{+}\!\!\mathcal{L}^{k}}$ is characterized by the \emph{existential $k$-pebble game}, and $\equiv_{\mathcal{C}^{k}}$ is characterized by the \emph{bijective $k$-pebble game}. 
See, e.g., \cite{atserias2019definable} for details about the approach to these relations via model-theoretic games. 
Here we only give a brief definition.

For $\tau$-structures $\struct{A}$ and $\struct{B}$ and  $Y\subseteq A$,  a map $f\colon Y \rightarrow B$ is \red{called a} \emph{partial homomorphism} (\emph{isomorphism})  if it is a homomorphism (an embedding) from the substructure of $\struct{A}$ on $Y$ to $\struct{B}$. 
% \emph{partial isomorphism} w.r.t.\ $X\subseteq A$ and $Y\subseteq B$ if  $f|_X$ is an isomorphism between the substructures of $\struct{A}$ and $\struct{B}$ on $X$ and $Y$, respectively.
\red{Both the existential and the bijective game}  are played on an ordered pair of $\tau$-structures \red{$(\struct{A},\struct{B})$} by two players, Spoiler and Duplicator, using $k$ pairs of pebbles $(a_1,b_1),\dots, (a_k,b_k)$.
In the existential $k$-pebble game, in each move, Spoiler chooses $i\in [k]$ and places the pebble $a_i$ (which might or might not already be on an element of $\struct{A}$) on any element of $\struct{A}$.
Duplicator has to respond by placing the pebble $b_i$ on an element of $\struct{B}$.
If at any point, the partial map specified by the pairs of pebbles placed on the board is not a partial homomorphism from $\struct{A}$ to $\struct{B}$, then \red{the game is over and} Spoiler \red{wins} the game.
In the bijective $k$-pebble game, in each move, Spoiler chooses $i\in [k]$.
Duplicator has to respond by selecting a bijection $f\colon A \rightarrow B$ with $f(a_j)=b_j$ for all $j\in [k]\setminus\{i\}$ such that the pair $(a_j,b_j)$ is already placed on the board. Then Spoiler places the pebble $a_i$ on any element of $\struct{A}$ and the pebble $b_i$ on its image under $f$.
If at any point, the partial map specified by the pairs of pebbles placed on the board is not a partial isomorphism from $\struct{A}$ to $\struct{B}$, then \red{the game is over and} Spoiler \red{wins}  the game.
In both games, Duplicator wins if the game continues forever.}

\subsection{\blue{Fixed-point logics}}

\emph{Inflationary fixed-point logic}  ($\IFP$)  is defined by adding formation rules to $\FO$ whose semantics is defined with inflationary fixed-points of arbitrary operators, and \emph{least fixed-point logic} ($\LFP$) is defined by adding formation rules to $\FO$ whose semantics is defined using least  fixed-points of monotone operators. 
The logics $\LFP$ and $\IFP$ are \emph{equivalent} in the sense that they define the same relations over the class of all structures \cite{kreutzer2004expressive}. For this reason, they are both commonly referred to as FP (see, e.g., \cite{atserias2019definable}).  
\emph{Datalog} is usually understood as the existential positive fragment of LFP (see \cite{dawar2008datalog}). The existential positive fragments of LFP and IFP are equivalent, because the fixed-point operator induced by a formula from either of the fragments is monotone, which implies that its least and inflationary fixed-point coincide (see Proposition~10.3 in \cite{libkin2013elements}). This means that we can informally identify $\Datalog$ with the  existential positive fragment of FP. 
For the definitions of the counting extensions IFPC and LFPC  we refer the reader to \cite{gradel2007finite}.
One important detail is that the equivalence LFP $\equiv$ IFP extends to LFPC $\equiv$ IFPC (see p.~189 in \cite{gradel2007finite}).  
Again, we refer to both counting extensions simply as FPC. It is worth mentioning that the extension of Datalog with counting is also equivalent to
$\FPC$ \cite{gradel1992inductive}.
All we need to know about FPC in the present article is Theorem~\ref{separation}.
\begin{theorem}[Immerman and Lander \cite{dawar2015nature}] \label{separation} 
\blue{For every $\FPC$ $\tau$-sentence $\phi$, there exists $k\in \mathbb{N}$ such that, for all finite $\tau$-structures $\struct{A}$ and $\struct{B}$, if $\struct{A} \equiv_{\mathcal{C}^{k}}\! \struct{B}$, then $\struct{A}\models \phi $ if and only if $\struct{B} \models \phi$.}
\end{theorem}  
\noindent This result follows from the fact that for every FPC formula $\phi$ there exists $k$ such that, on structures with at most $n$ elements, $\phi$ is equivalent to a formula of $\mathcal{C}^{k}$ whose quantifier depth is bounded by a polynomial function of $n$ \cite{dawar2015nature}. 
%
%Clearly $\smash{\struct{A}\equiv_{\mathcal{C}_{\infty\omega}^{k}}\! \struct{B}}$  implies $\struct{A}\equiv_{\mathcal{C}^{k}}\! \struct{B}$.
%
Moreover, every formula of FPC is  equivalent to a formula of $\mathcal{C}_{\infty\omega}^{k}$ for some $k$, that is, FPC forms a fragment of the infinitary logic $\smash{\mathcal{C}_{\infty\omega}^{\omega}}$ (Corollary~4.20 in \cite{otto2000bounded}).

The logic $\FPR_{2}$ extends $\FPC$ by the \blue{mod-2} rank operator making it the most expressive logic explicitly treated in this article.
It adds an additional logical constructor that can be used to form a rank term $[\rk_{x,y}\phi(x, y)\bmod 2]$ from a given formula $\phi(x, y)$. The value of $[\rk_{x,y}\phi(x, y)\bmod 2]$ in an input structure $\struct{A}$ is the rank of a $\{0,1\}$-matrix specified by $\phi(x, y)$ through its evaluation in $\struct{A}$. 
For instance, $[\rk_{x,y} (x=y) \wedge \psi(x)\bmod 2]$ computes in an input structure $\struct{A}$  the number of elements $a\in A$ such that $\struct{A}\models \psi(a)$ for a given formula $\psi(x)$ \cite{dawar2009logics}.
The  satisfiability of a suitably encoded system of \blue{mod-2} equations ${M}\boldf{x}=\boldf{v}$ can be tested in $\FPR_{2}$ by comparing the rank of ${M}$ with the rank of the extension of ${M}$ by $\boldf{v}$ as a last column.
A thorough definition of $\FPR_{2}$ can be found in \cite{dawar2009logics,gradel2019rank}; our version below is rather simplified, e.g., we disallow the use of $\leq$ for comparison of numeric terms, and also the use of free variables over the numerical sort.

Let $S$ be a finite set. A \emph{fixed-point} of an operator $F\colon \mathrm{Pow}(S) \rightarrow \mathrm{Pow}(S)$ is an element $U\in \mathrm{Pow}(S)$ with $U = F(U)$. A fixed-point $U$ of $F$ is called \emph{inflationary}  if it is the limit of the sequence $U_{i+1}\coloneqq U_{i} \cup F(U_{i})$ with $U_{0}=\emptyset$ in which case we write  $U=\Ifp(F)$, and \emph{deflationary} if it is the limit of the sequence $U_{i+1}\coloneqq U_{i}\cap F(U_{i})$ with $U_{0}=S$ in which case we write  $U=\Dfp(F)$. 
The members of either of the sequences are called the \emph{stages} of the induction. Clearly, $\Ifp(F)$ and $\Dfp(F)$ exist and are unique  for every such operator $F$.

Let $\tau$ be a relational signature. The set of \emph{inflationary fixed-point (IFP) formulas} over $\tau$ is defined inductively as follows.  
Every atomic $\tau$-formula is an IFP $\tau$-formula and formulas built from IFP $\tau$-formulas using the usual first-order constructors are again IFP $\tau$-formulas.
If $\phi(\boldf{x},\boldf{y})$ is an IFP $(\tau\cup \{R\})$-formula for some relation symbol $R\notin \tau$ of arity $k$, $\boldf{x}$ is $k$-ary, and $\boldf{y}$ is $\ell$-ary, then $[\ifp_{R,\boldf{x}}\phi]$ is an IFP $\tau$-formula with the same set of free variables.
The semantics of inflationary fixed-point logic is defined similarly as for first-order logic; we only discuss how to interpret the inflationary fixed point constructor. 
Let $\struct{A}$ be a finite relational $\tau$-structure.
For every  $\boldf{c}\in A^{\ell}$, we consider the induced operator 
$  \Op^\struct{A} \llbracket \phi(\cdot ,\boldf{c})  \rrbracket   \colon  \mathrm{Pow}(A^{k})  \rightarrow  \mathrm{Pow}(A^{k}),  R\mapsto  \{ {\boldf{t}} \in A^{k} \mid  (\struct{A},R) \models \phi(\boldf{t}, \boldf{c})   \}.  $
Then $\struct{A} \models [\ifp_{R,\boldf{x}} \phi]({\boldf{t}}, \boldf{c})$ if and only if 
${\boldf{t}}\in \Ifp  \Op^\struct{A} \llbracket  \phi(\cdot,\boldf{c})  \rrbracket$.
To make our IFP formulas more readable, we introduce the expression
$[\dfp_{R,\boldf{x}}\phi]$ as a shortcut for the IFP formula $\neg [\ifp_{R,\boldf{x}} \neg  \phi_{R/ \neg R}] $ where $\phi_{R/ \neg R}$ is obtained from $\phi$ by replacing every occurrence of $R$ in $\phi$ with $\neg R$. 
Note that  ${\boldf{t}}\in \Dfp \Op^\struct{A} \llbracket  \phi(\cdot,\boldf{c})  \rrbracket$ if and only if $\struct{A} \models [\dfp_{R,\boldf{x}} \phi]({\boldf{t}},\boldf{c})$.

Finally, we present a simplified version of the \blue{mod-2} rank operator which is, nevertheless, expressive enough for the purpose of capturing those temporal CSPs that are expressible in $\FPR_2$.
We define the set of \emph{numeric terms} over $\tau$ inductively as follows. 
Every $\IFP$ $\tau$-formula is a numeric term taking values in $\{0,1\}$ corresponding to its truth values when evaluated in $\struct{A}$.
Composing numeric terms with the nullary function symbols $0$, $1$ and the binary function symbols $+,\cdot $, which have the usual interpretation over $\mathbb{N}$, yields numeric terms taking values in $\mathbb{N}$ when evaluated in $\struct{A}$.
Finally, if $f(\boldf{x},\boldf{y},\boldf{z})$ is a numeric term where $\boldf{x}$ is $k$-ary, $\boldf{y}$ is $\ell$-ary, and $\boldf{z}$ is $m$-ary, then $[\rk_{\boldf{x},\boldf{y}}f\bmod 2]$ is a numeric term with free variables consisting of the entries of $\boldf{z}$.
We use the notation $f^{\struct{A}}$ for the evaluation of a numeric term $f$ in $\struct{A}$.
For $\boldf{c}\in A^{m}$,  we write $\Mat^\struct{A}_2\llbracket  f(\cdot,\cdot,\boldf{c}) \rrbracket$ for the $\{0,1\}$-matrix 
whose entry at the coordinate 	$(\boldf{t},\boldf{s}) \in A^{k}\times A^{\ell}$ is   $  f^{ \struct{A}}(\boldf{t},\boldf{s}, \boldf{c}) \bmod 2.$
Then $[\rk_{\boldf{x},\boldf{y}}f \bmod 2]^{ \struct{A}}(\boldf{c})$ denotes the rank of $\Mat^\struct{A}_2\llbracket  f(\cdot,\cdot,\boldf{c}) \rrbracket$.  
The value for $[\rk_{\boldf{x},\boldf{y}}f\bmod 2]$ is well defined because the rank of $\Mat^\struct{A}_2\llbracket  f(\cdot,\cdot,\boldf{c}) \rrbracket$ does not depend on the ordering of the rows and the columns.
Now we can define the set of $\FPR_2$ $\tau$-formulas.
Every $\IFP$ $\tau$-formula is an $\FPR_2$ $\tau$-formula.   
If $f(\boldf{x})$ and $g(\boldf{y})$ are numeric terms, then $f = g$ is an $\FPR_2$ $\tau$-formula whose free variables are the entries of $\boldf{x}$ and $\boldf{y}$.
The latter carries the obvious semantics $\struct{A}\models (f = g)(\boldf{t},\boldf{s})$ if and only if $f^{\struct{A}}(\boldf{t}) =  g^{\struct{A}}(\boldf{s})$.

\begin{example} \red{The  $\FPR_2$ formula}  
$ \bigwedge_{j=0}^{i-1} \neg ([\rk_{x,y}(x=y)\wedge \phi(x)\bmod 2]  = j)$  is equivalent to  $\exists^{i}x.\, \phi(x)$.
\end{example} 

%
%Let $\FP^{+}$ be the fragment of $\FP$ which allows the following constructors: $\exists, \wedge, \vee$, $\ifp$, and $\dfp$.
%% 
%Note that satisfiability of $\FP^{+}$ formulas is preserved under homomorphisms.
%%
%Indeed, this is known for $\Datalog$, the fragment of $\FP^{+}$ without the $\dfp$ constructor,
%%
%and if $\phi$ is a Datalog formula and $f\colon\struct{A} \rightarrow \struct{B}$ a homomorphism, then  $f(\Dfp \Op \llbracket  \struct{A}, \phi(\,\cdot \, ,\boldf{c})  \rrbracket) \subseteq \Dfp \Op \llbracket  \struct{B}, \phi(\,\cdot \, ,f(\boldf{c}))  \rrbracket$
%%
%which means that $\struct{A} \models [\dfp_{R,\boldf{x}} \phi]({\boldf{t}},\boldf{c})$ implies $\struct{B} \models [\dfp_{R,\boldf{x}} \phi](f({\boldf{t}}),f(\boldf{c}))$.
%%
%Also note that $\FP^{+}$ forms a fragment of $\exists^{+} \smash{L_{\infty\omega}^{\omega}}$, the existential positive fragment of 
%%
%$\smash{L_{\infty\omega}^{\omega}\coloneqq \bigcup_{k\in \mathbb{N}} L_{\infty\omega}^{k}}$.
%%
%Here we would translate the dfp constructor into an infinite conjunction in a similar way how the ifp constructor can be translated into an infinite disjunction, see, e.g., Theorem~11.2 in \cite{libkin2013elements}.
% 

\subsection{\blue{Logical expressibility of constraint satisfaction problems\label{section_CSPs}}}
The \emph{constraint satisfaction problem} $\CSP(\struct{B})$ for a structure $\struct{B}$ with a finite relational signature $\tau$ is the computational problem of deciding whether a given finite $\tau$-structure $\struct{A}$ maps homomorphically to $\struct{B}$.  
By a standard result from database theory, $\struct{A}$ maps homomorphically to $\struct{B}$ if and only if the \emph{canonical conjunctive query} $Q_{\struct{A}} $ is true in $\struct{B}$ \cite{chandra1977optimal}; $Q_{\struct{A}}$ is the pp-sentence whose variables are the domain elements of $\struct{A}$ and whose quantifier-free part is the conjunction of all atomic formulas that hold in $\struct{A}$. 
We might occasionally refer to the atomic subformulas of $Q_{\struct{A}}$ as \emph{constraints}.
We call $\struct{B}$ a \emph{template} of $\CSP(\struct{B})$. A \emph{solution} for an instance $\struct{A}$ of $\CSP(\struct{B})$ is a homomorphism $\struct{A}\rightarrow \struct{B}$.

Formally, $\CSP(\struct{B})$ stands for the class of all finite $\tau$-structures that homomorphically map to $\struct{B}$.  
Following Feder and Vardi~\cite{feder1998computational}, we say that the CSP of a $\tau$-structure $\struct{B}$ is \emph{expressible} in a logic $\mathcal{L}$ if there exists a sentence in $\mathcal{L}$ that defines the complementary class $\text{co-CSP}(\struct{B})$ of all finite $\tau$-structures which do not homomorphically map to $\struct{B}$.  
\begin{example} 	$ \exists z   [ \ifp_{T,(x,y)}\,  x<y \vee \exists h  (x<h \wedge T(h,y))] (z,z)$ defines $\textup{co-CSP}(\mathbb{Q};<)$. 
\end{example}
\blue{Naturally, showing logical inexpressibility of CSPs becomes more difficult the further we get in the search for a logic capturing Ptime.
Fortunately, inexpressibility in fixed-point logics can often be proved by showing inexpressibility in a much stronger infinitary logic with finitely many variables, e.g., in $\smash{\mathcal{C}_{\infty\omega}^{\omega}}$ for FPC.
In the case of FPC, we adapt the terminology from \cite{dawar2017definability} and call this proof method the unbounded counting width argument.
Formally, the \emph{counting width} of $\CSP(\struct{B})$ for a $\tau$-structure $\struct{B}$ is the function that assigns to each $n\in \mathbb{N}$ the minimum value $k$ for which there is a $\tau$-sentence $\phi$ in $\mathcal{C}^{k}$ such that, for every $\tau$-structure $\struct{A}$ with $|A|\leq n$, we have $\struct{A} \models \phi$ if and only if $\struct{A} \rightarrow \struct{B}$ \cite{dawar2017definability}. 
By Theorem~\ref{separation}, if $\CSP(\struct{B})$ has unbounded counting width, then it is inexpressible in FPC.
The main tool for transferring logical (in)expressibility results for CSPs are logical reductions.}
\begin{definition}[\cite{atserias2009affine}] 
Let $\sigma$, $\tau$ be finite relational signatures. Moreover, let $\Theta$ be a set of  $\FPR_{2}$ $\sigma$-formulas.
A \emph{$\Theta$-interpretation of $\tau$ in $\sigma$  with $p$ parameters} is a tuple $\mathcal{I}$ of $\sigma$-formulas from $\Theta$ consisting of a distinguished
$(d+p)$-ary \emph{domain} formula $\delta_{\mathcal{I}}(\boldf{x},\boldf{y})$ and,
for each $R\in \tau$,  an $(n\cdot d+p)$-ary formula $\phi_{\mathcal{I},R}(\boldf{x}_{1},\dots,\boldf{x}_{n},\boldf{y})$ where $n =\ar(R)$.
The \emph{image of $\struct{A}$ under $\mathcal{I}$ with parameters $\boldf{c}\in A^p$} is the $\tau$-structure $\mathcal{I}(\struct{A},\boldf{c})$ with domain $\{\boldf{t} \in A^{d} \mid \struct{A} \models \delta_{\mathcal{I}}(\boldf{t},\boldf{c})\}$ and  relations  
$$
R^{\mathcal{I}(\struct{A},\boldf{c})} = \{(\boldf{t}_{1},\dots,\boldf{t}_{n}) \in (A^{d})^n  \mid  \struct{A} \models \phi_{\mathcal{I},R}(\boldf{t}_1,\dots, \boldf{t}_n,\boldf{c})   \}.
$$ 
Let $\struct{B}$ be a $\sigma$-structure and $\struct{C}$ a $\tau$-structure. We write $\CSP(\struct{B}) \leq_{\Theta} \CSP(\struct{C})$ and say that $\CSP(\struct{B})$ \emph{reduces to $\CSP(\struct{C})$ under $\Theta$-reducibility} if  there exists a $\Theta$-interpretation $\mathcal{I}$ of $\tau$ in $\sigma$ with $p$ parameters such that, for every finite $\sigma$-structure $\struct{A}$ with $|A|\geq p$,  the following are  equivalent:   
\begin{itemize}
\item $\struct{A} \rightarrow \struct{B}$,
\item $\mathcal{I}(\struct{A},\boldf{c}) \rightarrow \struct{C}$ for some injective tuple $\boldf{c}\in A^p$,
\item $\mathcal{I}(\struct{A},\boldf{c}) \rightarrow \struct{C}$ for every injective tuple $\boldf{c}\in A^p$.
\end{itemize} 
\end{definition} 

Seen as a binary relation, $\Theta$-reducibility is transitive if $\Theta$ is any of the standard logical fragments or extensions of FO we have mentioned so far.
The following reducibility result was obtained in \cite{atserias2009affine} for finite-domain CSPs. A close inspection of the  original proof reveals that the statement holds for infinite-domain CSPs as well.  
\begin{theorem}[Atserias, Bulatov, and Dawar \cite{atserias2009affine}]  \label{REDUCTION}  
Let  $\struct{B}$ and $\struct{C}$ be structures with finite relational signatures such that  $\struct{B}$ is pp-constructible from $\struct{C}$. Then $\CSP(\struct{B})\leq_{\textup{Datalog}} \CSP(\struct{C}).$     
\end{theorem}   
%
%A full proof of Theorem~\ref{REDUCTION}, which mostly follows the original one, can be found in the appendix.
%  
%
\red{It is important to note that $\leq_{\textup{Datalog}}$ preserves the expressibility of CSPs in $\mathcal{L}$ for every $\mathcal{L} \in \{\Datalog, \FP, \FPC , \FPR_{2}\}$. 
This 
%(not entirely trivial) 
fact is mentioned in \cite{atserias2009affine} for $\mathcal{L}=\mathcal{C}_{\infty\omega}^{\omega}$
and in~\cite{atserias2008digraph} for Datalog (referring to the techniques in~\cite{kolaitis2000conjunctive}); we include a short proof which uses a result from~\cite{feder2003homomorphism}.} 
\begin{proposition}  
\label{prop:reduction_pp_constructible}  Let $\struct{B}$, $\struct{C}$ be structures with finite relational signatures.   If  $\CSP(\struct{B})\leq_{\textup{Datalog}} \CSP(\struct{C})$ and $\CSP(\struct{C})$ is expressible in $\mathcal{L} \in \{\Datalog, \FP, \FPC, \FPR_{2}\}$, then $\CSP(\struct{B})$ is expressible in $\mathcal{L}$. 
\end{proposition} 
\begin{proof} We only prove the statement in the case of $\Datalog$. The remaining cases  are analogous and in fact even simpler, because $\FP$, $\FPC$ and $\FPR_{2}$ allow inequalities. 

Let $\sigma$ be the signature of $\struct{B}$, let $\tau$ be the signature of $\struct{C}$, and let $\phi_{\struct{C}}$ be a $\Datalog$ $\tau$-sentence that defines  $\text{co-CSP}(\struct{C})$. 
Let $\mathcal{I}$ be an interpretation of $\tau$ in $\sigma$ with $p$ parameters witnessing that $\CSP(\struct{B})\leq_{\textup{Datalog}} \CSP(\struct{C}).$  
Consider the sentence $\phi'_{\struct{B}}$ obtained from $\phi_{\struct{C}}$ by the following sequence of syntactic replacements. 
First, we introduce a fresh $p$-tuple $\boldf{y}$ of existentially quantified variables.
Second, we replace each existentially quantified variable $x_{i}$ in $\phi_{\struct{C}}$ by a $d$-tuple $ \boldf{x}_i$ of fresh existentially quantified variables    and conjoin $\phi_{\struct{C}}$ with the formula $\bigwedge_{i}\delta_{\mathcal{I}}(\boldf{x}_i,\boldf{y})$. 
Then, we replace each atomic formula in $\phi_{\struct{C}}$ of the form $R(x_{i_{1}},\dots,x_{i_{n}})$ for $R\in \tau$ by the formula $\phi_{\mathcal{I},R}(\boldf{x}_{i_{1}},\dots,\boldf{x}_{i_{n}},\boldf{y}) $;
we also readjust the arities of the auxiliary relation symbols and the amount of the first-order free variables in each IFP subformula of $\phi_{\struct{C}}$. 
Finally, we conjoin the resulting formula with  $\bigwedge_{i\neq j} \boldf{y}\of{i} \neq \boldf{y}\of{j}$.
Now, for all $\sigma$-structures $\struct{A}$ with $|A|\geq p$, we have that $\struct{A} \models \phi'_{\struct{B}}$ if and only if $\mathcal{I}(\struct{A},\boldf{c}) \models \phi_{\struct{C}}$ for some injective tuple $\boldf{c} \in A^p$. 
Since $\phi_{\struct{C}}$ defines the class of all instances of $\CSP(\struct{C})$ which have no solution,  $\phi'_{\struct{B}}$ defines the class of all instances of $\CSP(\struct{B})$ with at least $p$ elements which have no solution.
Let $\phi''_{\struct{B}}$ be the disjunction of the canonical conjunctive queries for all the finitely many instances $\struct{A}_{1},\dots,\struct{A}_{\ell}$ of $\CSP(\struct{B})$ with less than $p$ elements which have no solution.
Then $\phi''_{\struct{B}}$ defines the class of all instances of $\CSP(\struct{B})$ with less than $p$ elements which have no solution.
Let $\struct{A}$ be a $\sigma$-structure with $|A|<p$.
If $\struct{A}\centernot{\rightarrow} \struct{B}$, then $\struct{A} \models Q_{\struct{A}_{i}}$ for some $i\in [\ell]$, which implies $\struct{A} \models \phi''_{\struct{B}}$. 
If $\struct{A}\rightarrow \struct{B}$, then $\struct{A} \centernot{\models} Q_{\struct{A}_{i}}$ for every $i\in [\ell]$, otherwise $\struct{A}_{i} \rightarrow \struct{A}$ for some $i\in [\ell]$ which yields a contradiction to $\struct{A}_i \centernot{\rightarrow} \struct{B}$. 
Thus $ \phi'_{\struct{B}}\vee \phi''_{\struct{B}}$ defines $\text{co-CSP}(\struct{B})$.
We are not finished yet, because $\phi'_{\struct{B}}\vee \phi''_{\struct{B}}$ is not a valid $\Datalog$ sentence. It is,  however, a valid sentence in $\Datalog(\neq)$, the expansion of $\Datalog$ by inequalities between variables.
Note that, if $\struct{A} \centernot{\rightarrow} \struct{B}$ and $\struct{A} \rightarrow \struct{A}' $, then $\struct{A}' \centernot{\rightarrow} \struct{B}$, i.e., $\text{co-CSP}(\struct{B})$ is a class closed under homomorphisms. 
Thus, by Theorem~2 in \cite{feder2003homomorphism}, there exists a $\Datalog$ sentence $\phi_{\struct{B}}$ that defines $\text{co-CSP}(\struct{B})$. 
We conclude that $\CSP(\struct{B})$ is expressible in $\Datalog$.  
\end{proof}
We now introduce a formalism that  simplifies the presentation of algorithms for TCSPs.
\blue{For an $n$-ary tuple $\boldf{t}$ and $I\subseteq [n]$, we use the notation $\pr_{I}(\boldf{t})$ for the tuple $(\boldf{t}\of{i_1},\dots, \boldf{t}\of{i_{m}})$ where $I=\{i_1,\dots, i_{m}\}$ with $i_1<\dots< i_{m}$.
The function $\pr_{I}$ naturally extends to relations. For $R\subseteq B^n$ and $I\subseteq [n]$, the \emph{projection} of $R$ to $I$ is defined as $\pr_{I}(R) $.
We call \red{the projection} \emph{proper} if $I \notin \{ \emptyset,[n]\}$, and \emph{trivial} if it equals $B^{|I|}$.  
For $R\subseteq B^n$ and ${\sim}\subseteq [n]^2$, the \emph{contraction} of $R$ modulo $\sim$, denoted by $\con_{\sim}(R)$, is defined as  $\{\boldf{t} \in R \mid \boldf{t}\of{i}=\boldf{t}\of{j} \text{ for all } (i,j) \in {\sim} \}$.}
Whenever it is convenient, we will assume that the set of relations of a temporal structure is closed under \blue{projections and contractions}.
Note that \red{these relations are pp-definable in the structure, and hence} adding them to the structure does not influence the set of polymorphisms (Proposition~\ref{InvAutPol}), and it also does not influence the expressibility of its CSP in $\Datalog$, $\FP$, $\FPC$, or $\FPR_{2}$ (\red{Theorem~\ref{REDUCTION}}).  
\begin{definition}[Projections and contractions]	Let $\struct{A}$ be an instance of $\CSP(\struct{B})$ for a $\tau$-structure $\struct{B}$.

\blue{The \emph{projection} of $\struct{A}$ to $V\subseteq A$ is the $\tau$-structure $\pr_{V}(\struct{A})$  obtained from $\struct{A}$ as follows. 
The domain of $\pr_{V}(\struct{A})$ is $V$ and, for $R\in\tau$, the relation $R^{\pr_{V}(\struct{A})}$ consists of all tuples $\boldf{t}$ for which there exists $\tilde{R}\in \tau $ such that $\boldf{t}\in \pr_{I}(\tilde{R}^{\struct{A}})$ and $R^{\struct{B}}=\pr_{I}(\tilde{R}^{\struct{B}})$ where $I\coloneqq \{i\in [n] \mid \boldf{t}\of{i}\in V\}$.  

The \emph{contraction} of $\struct{A}$ modulo $C \subseteq A^2$ is the $\tau$-structure $\con_{C}(\struct{A})$  obtained from $\struct{A}$ as follows. 
The domain of $\con_{C}(\struct{A})$ is $A$ and, for $R\in\tau$, the relation $R^{\con_{C}(\struct{A})}$ consists of all tuples $\boldf{t}$ for which there exists $\tilde{R}\in \tau $ such that $\boldf{t}\in \tilde{R}^{\struct{A}}$ and $R^{\struct{B}} = \con_{\sim}(\tilde{R}^{\struct{B}})$ where ${\sim} \coloneqq \{(i,j)\in [n]^2 \mid  (\boldf{t}\of{i},\boldf{t}\of{j})  \in C\}$.}
%
%The domain of $\con_{C}(\struct{A})$ is $A$. For $R\in\tau$, the relation $R^{\con_{C}(\struct{A})}$ contains all tuples $\boldf{t}$ for which there exists $\tilde{R}\in \tau $ such that $\boldf{t}\in \tilde{R}^{\struct{A}}$ and $R^{\struct{B}} = \con_{\sim}(\tilde{R}^{\struct{B}})$ where ${\sim} \coloneqq \{(i,j)\in [n]^2 \mid  (\boldf{t}\of{i},\boldf{t}\of{j})  \in C\}$.
%%
%Moreover, if $R$ interprets as the equality relation in $\struct{B}$, then $C\subseteq R^{\con_{C}(\struct{A})}$.
%%
%There are no other tuples in $R^{\con_{C}(\struct{A})}$.} 
\end{definition}  
\subsection{\blue{Temporal CSPs}}\label{temporal structuresubsection}
A structure with domain ${\mathbb Q}$ is called \emph{temporal} if its relations are first-order definable in $(\mathbb{Q};<)$.  
An important observation is that if $\struct{B}$ is a temporal structure and $f$ is an order-preserving map between two finite subsets of $\mathbb{Q}$, then $f$ can be extended to an automorphism of $\struct{B}$.
This is a consequence of Proposition~\ref{InvAutPol} and the fact that $(\mathbb{Q};<)$ is homogeneous.
Relations which are first-order definable in $(\mathbb{Q};<)$ are called \emph{temporal}.  
The \emph{dual} of a  temporal relation $R$ is defined as \blue{$\{-\boldf{t} \mid \boldf{t}\in R\}$, where the operation $x\mapsto -x$ acts component-wise.}
The \emph{dual} of a temporal structure is the temporal structure whose relations are precisely the duals of the relations of the original one.
Every temporal structure is homomorphically equivalent to its dual via the map $x\mapsto -x$, which means that both structures have the same CSP.
The CSP of a temporal structure is called a \emph{temporal CSP} (TCSP).

\begin{definition}[Min-tuples]  The \emph{min-indicator function} \label{minindicator} $\chi\colon \mathbb{Q}^{k}\rightarrow \{0,1\}^{k}$ is defined by $\chi({\boldf{t}})\of{i}\coloneqq 1$ if and only if ${\boldf{t}}\of{i}$ is a minimal entry in ${\boldf{t}}$; we call $\chi({\boldf{t}})\in \{0,1\}^{k}$  the \emph{min-tuple} of ${\boldf{t}}\in \mathbb{Q}^{k}$. 
As usual, if $R \subseteq {\mathbb Q}^k$, then $\chi(R)$ denotes $\{\chi(\boldf{t}) \mid \boldf{t} \in R\}$.  
For  ${\boldf{t}}\in \mathbb{Q}^{k}$, we set $\Minset(\boldf{t})\coloneqq \{i\in [k] \mid   \chi(\boldf{t})\of{i}=1\}.$
\end{definition}
\begin{definition}[Free sets]\label{def:freeset}
Let $\struct{B}$ be a temporal structure with signature $\tau$ and $\struct{A}$ an instance of $\CSP(\struct{B})$. 
A \emph{free set} of $\struct{A}$ is a non-empty subset $F\subseteq A$ such that, if $R\in \tau$ is $k$-ary and $\boldf{s}\in R^{\struct{A}} $, then either no entry of $\boldf{s}$ is contained in $F$, or there exists a tuple $\boldf{t}\in R^{\struct{B}}$ such that $\Minset(\boldf{t})= \{i\in [k] \mid \boldf{s}\of{i}\in F\}$.
If $R\in \tau$ has arity $k$ and $\boldf{s}\in  A^k$, we define the \emph{system of min-sets} $\Minsystem_{R}(\boldf{s})$ as the set of all $M\subseteq \{\boldf{s}\of{1},\dots, \boldf{s}\of{k}\}$ for which there exists $\boldf{t}\in R^{\struct{B}}$ such that  $\Minset(\boldf{t})= \{i\in [k] \mid \boldf{s}\of{i}\in M\}$.
For a subset $V$ of $\{\boldf{s}\of{1},\dots, \boldf{s}\of{k}\}$, we define $\Ideal{R}{V}(\boldf{s})$ as the set of all $M \in \Minsystem_{R}(\boldf{s})$ such that $M \subseteq V $, and $\Filter{R}{V}(\boldf{s})$ as the set of all $M \in  \Minsystem_{R}(\boldf{s})$ such that $V \subseteq M$. 
\end{definition}
\blue{Let $\struct{B}$ be a temporal structure. If an instance $\struct{A}$ of $\CSP(\struct{B})$ has a solution, then there must exist a non-empty set $F\subseteq A$ consisting of the elements of $A$ which \red{have the} minimal value in the solution.
It is easy to see that every such $F$ is a free set of $\struct{A}$.
However, it is not the case that the existence of a free set guarantees the existence of a solution.
For example, if $\struct{B}=(\mathbb{Q};<)$, then $\emptyset \subsetneq F\subseteq A$ is a free set of $\struct{A}$ if and only if the elements of $F$ do not appear in the second argument of any $<$-constraint of  $\struct{A}$.
But even if such $F$ exists, the remaining part of $\struct{A}$ might contain a directed $<$-cycle and thus be unsatisfiable.
Only a repeated search for free sets can guarantee the existence of a solution, and testing containment in a free set is a decision problem in itself whose complexity depends on the first-order definitions of the relations of $\struct{B}$ in $(\mathbb{Q};<)$.
This topic is covered in Section~\ref{section_tcsps_in_fp}.}
\subsection{\blue{Clones and minions}}
An at least unary operation on a set $A$ is called a \emph{projection onto the $i$-th coordinate}, and denoted by $\pr_{i}$, if it returns the $i$-th argument for each input value.
A set of $\mathscr{A}$ operations over a fixed set $A$ is called a \emph{clone} (over $A$) if it contains all projections and, whenever $f\in \mathscr{A}$ is $n$-ary and $g_1,\dots, g_{n} \in \mathscr{A}$ are $m$-ary, then $f(g_1,\dots,g_n) \in \mathscr{A}$, \red{where} $f(g_1,\dots,g_n)$ is the $m$-ary map $(x_1,\dots, x_m) \mapsto f(g_1(x_1,\dots, x_m),\dots, g_n(x_1,\dots, x_m))$. 
The set of all polymorphisms of a relational structure $\struct{A}$, denoted by $\Pol(\struct{A})$, is a clone. 
For instance, the clone $\Pol(\{0,1\}; 1\textup{IN}3 )$ consists of all projection maps on $\{0, 1\}$, and is called the \emph{projection clone} \cite{bodirsky2015topological}.  
\blue{A \emph{minion} is a set of functions with a common domain which is closed under compositions of a single function with projections; in particular, clones are minions.}
\red{Let $\mathscr{A}$ and $\mathscr{B}$ be sets of operations over $A$ and $B$, respectively.} 
A map $\xi \colon \mathscr{A} \rightarrow \mathscr{B}$ is called
\begin{itemize}
\item  a \emph{clone homomorphism}  
if it preserves arities, projections, and compositions, that is, $$\xi(f(g_{1},\dots,g_{n}))=\xi(f)(\xi(g_{1}),\dots , \xi(g_{n}))$$ holds for all $n$-ary $f$ and  $m$-ary $g_{1},\dots,g_{n}$ from $\mathscr{A}$,
\item  a \emph{minion homomorphism}   
if it preserves arities and those compositions as above where $g_{1},\dots,g_{n}$ are projections, 
\item  \emph{uniformly continuous} if for every finite $B'\subseteq B$ there exists a finite $A'\subseteq A$ such that if $f,g\in \mathscr{A}$ of the same arity agree on $A'$, then $\xi(f)$ and $\xi(g)$ agree on $B'$.
\end{itemize}     
The recently closed finite-domain CSPs tractability conjecture can be reformulated as follows: the polymorphism clone of a finite structure $\struct{A}$ either admits a minion homomorphism to the projection clone in which case $\CSP(\struct{A})$ is NP-complete, or it does not and $\CSP(\struct{A})$ is polynomial-time tractable \cite{barto2018wonderland}.
The former is the case if and only if $\struct{A}$ pp-constructs all finite structures.
\blue{Later, we will need the following lemma.

\begin{lemma}[\cite{barto2018wonderland}] \label{pp_clone_minion} Let $\struct{A}$ and $\struct{B}$ be structures such that  $\struct{B}$ is pp-constructible from $\struct{A}$. Then there is a minion homomorphism from $\Pol(\struct{A}) $ to $ \Pol(\struct{B})$. 
\end{lemma}
\begin{proof} For some $d\geq 1$, there exists a $d$-dimensional pp-power $\struct{A}'$ of $\struct{A}$ and homomorphisms $h'\colon \struct{A}'\rightarrow \struct{B}$ and $h\colon \struct{B}\rightarrow \struct{A}'$. Let $\xi'$ be the map that sends $f\in \Pol(\struct{A})$ to its component-wise action on $A^d$.
By Proposition~\ref{InvAutPol}, $\xi'$ is a clone homomorphism from $\Pol(\struct{A})$ to $\Pol(\struct{A}')$.
Then it is easy to see that $\xi(f)\coloneqq h'\circ \xi'(f)\circ h$ is a minion homomorphism from $\Pol(\struct{A})$ to $\Pol(\struct{B})$.
\end{proof}}

\subsection{\blue{Polymorphisms of temporal structures}}
\label{sect:pols}
The following notions were used in the P versus NP-complete complexity classification of TCSPs \cite{bodirsky2010complexity}. 
Let $\Min$ denote the binary minimum operation on $\mathbb{Q}$.  
The \emph{dual} of a $k$-ary operation $f$ on $\mathbb{Q}$ is the map $(x_1,\dots,x_k)\mapsto -f(-x_{1}, \dots , -x_{n})$. 
Let us fix any endomorphisms $\alpha,\beta,\gamma$ of $(\mathbb{Q};<)$ such that $\alpha(x)<\beta(x)<\gamma(x)<\alpha(x+\varepsilon)$ for every $x\in \mathbb{Q}$ and every $\varepsilon\in \mathbb{Q}_{>0}$. 	
Such unary operations can be constructed inductively, see the paragraph below Lemma~26 in \cite{bodirsky2010complexity}.
\blue{Later in the article, we will need the following observation which highlights the special properties of these endomorphisms.
\begin{lemma} \label{claim:comparisson} For all $x,y\in \mathbb{Q}$, we have the following
\begin{enumerate}
\item \label{item_alpha_beta_1} $\alpha(x) < \beta(y)$ if and only if $x\leq y$,
\item \label{item_alpha_beta_2} $\beta(x)  < \alpha(y)$ if and only if $x<y$.
\end{enumerate} 
\end{lemma}
\begin{proof}[Proof of Lemma~\ref{claim:comparisson}] 
We get \eqref{item_alpha_beta_2} simply by negating \eqref{item_alpha_beta_1} because $\alpha$ and $\beta$ have disjoint images.

For \eqref{item_alpha_beta_1}, \red{arbitrarily choose} $x,y\in \mathbb{Q}$. If $x\leq y$, then $\alpha (x) \leq \alpha(y) $ because $\alpha$ is an endomorphism of $(\mathbb{Q};<)$. Moreover, $\alpha(y)<\beta(y)$ by the definition of $\alpha$ and $\beta$. Thus $\alpha(x) < \beta(y)$ in this case.
If $x>y$, then $\alpha(x)>\beta(y)$ follows directly from the definition of $\alpha$ and $\beta$.  
\end{proof} 
\noindent Then $\mi$ and $\mx$ are the binary operations on $\mathbb{Q}$ defined by
\[ \mi(x,y)\coloneqq \begin{cases}
\alpha(\Min(x,y)) & \text{ if } x= y, \\ \beta(\Min(x,y)) & \text{ if } x<y, \\ \gamma(\Min(x,y)) & \text{ if } x>y,  
\end{cases}
\qquad\text{and}\qquad
\mx(x,y)\coloneqq  \begin{cases}
\alpha(\Min(x,y)) & \text{if } x\neq y, 	\\ \beta(\Min(x,y)) & \text{if } x= y,
\end{cases}
\]   
respectively.
Note that the kernels of the operations $\mi$ and $\mx$ refine the kernel of the operation $\Min$.
Namely, $\mi(x,y) < \mi(x',y')$ if and only if 
\begin{itemize}
\item  $\min(x,y)< \min(x',y')$, or
\item $\min(x,y) = \min(x',y')$ and $-\chi(x,y)$ is smaller than $-\chi(x',y')$ w.r.t.\ the lexicographic order,
\end{itemize}
and $\mx(x,y) < \mx(x',y')$ if and only if 
\begin{itemize}
\item  $\min(x,y)< \min(x',y')$, or
\item $\min(x,y) = \min(x',y')$ and $x\neq y$ while $x'=y'$.
\end{itemize}
In \cite{bodirsky2010complexity}, the operation $\mi$ defined exactly the opposite order on pairs $(x,y)$ and $(y,x)$ for distinct $x,y\in \mathbb{Q}$ than in our case. However, it is easy to see that our version of the operation generates the same clone. We only deviate from the original definition for cosmetic reasons which will become clear in Definition~\ref{definition:PWNUs}.}
Let $\elel$ be an arbitrary binary operation on $\mathbb{Q}$ such that $\elel(x,y)<\elel(x',y')$ if and only if
\begin{itemize}
\item $x\leq 0$ and $x<x'$, or
\item $x\leq 0$ and $x=x'$ and $y<y'$, or
\item $x,x'>0$ and $y<y'$, or
\item $x>0$ and $y=y'$ and $y<y'$.
\end{itemize}  
\begin{theorem}[Bodirsky and K\'ara \cite{bodirsky2010complexity,bodirsky2011reducts}] \label{TCSPdichot} Let $\struct{B}$ be a temporal structure. Either $\struct{B}$ is preserved by $\Min$, $\mi$, $\mx$, $\elel$, the dual of one of these operations, or a constant operation and $\CSP(\struct{B})$ is in P, or $\struct{B}$ pp-constructs $(\{0,1\};1\textup{IN}3)$ and $\CSP(\struct{B})$ is NP-complete.    
\end{theorem} 
There are two additional operations that appear in soundness proofs of algorithms for TCSPs;
$\pp$ is an arbitrary binary operation on $\mathbb{Q}$ that satisfies $\pp(x,y)\leq \pp(x',y')$ if and only if 
\begin{itemize}
\item $x\leq 0$ and $x\leq x'$, or
\item  $0<x$, $0<x'$, and $y\leq y'$, 
\end{itemize}
and $\lex$ is an arbitrary binary operation on $\mathbb{Q}$ that satisfies $\lex(x,y)< \lex(x',y')$ if and only if   
\begin{itemize}
\item $x<x'$, or
\item   $x=x'$ and $y < y'$.  
\end{itemize} 
If a temporal structure is preserved by $\Min, \mi$, or $\mx$, then it is preserved by $\pp$, and if a temporal structure is preserved by $\elel$, then it is preserved by $\lex$ \cite{bodirsky2010complexity}.

\section{\label{section_tcsps_in_fp}Fixed-point algorithms for TCSPs}
In this section, we discuss the expressibility in $\FP$  for some particularly chosen TCSPs that are provably in P. 
%
%These TCSPs have been chosen so that they represent all the different polynomial-time tractable cases that appear in the proof of Theorem~\ref{TCSPdichot}.
%
By Theorem~\ref{TCSPdichot}, a TCSP is polynomial-time tractable if its template is preserved by one of the operations $\Min, \mi, \mx$, or $\elel$.  
In the case of $\Min$, the known algorithm from \cite{bodirsky2010complexity} can be formulated as an FP algorithm.
In the case of $\mi$ and $\elel$, the known algorithms from \cite{bodirsky2010complexity,bodirsky2010fast} cannot be implemented in FP  as they involve choices of arbitrary elements. We show that there exist  choiceless versions that can be turned into FP sentences.
In the case of $\mx$, the known algorithm from \cite{bodirsky2010complexity} cannot be turned into an FP sentence because it relies on the use of linear algebra.
We show in Section~\ref{section_inexpressibility} that, in general, the CSP of a temporal structure preserved by $\mx$ cannot be expressed in FP but it can be expressed in the logic $\FPR_{2}$.

\subsection{\blue{A procedure for TCSPs with a template preserved by pp}}\label{parallelizedalgominmimx}
We first describe a procedure for temporal languages preserved by $\pp$ as it appears in \cite{bodirsky2010complexity}, and then the choiceless version that is necessary for the translation into an $\FP$ sentence.

Let $\struct{A}$ be an instance of $\CSP(\struct{B})$.  
The original procedure searches for a non-empty set $S\subseteq A$  
for which there exists a solution $\struct{A} \rightarrow \struct{B}$ 
under the assumption that the projection of $\struct{A}$ to $A\setminus S$ has a solution as an instance of $\CSP(\struct{B})$.
It was shown in \cite{bodirsky2010complexity} that $S$ has this property if it is a free set of $\struct{A}$, and that $\struct{A}\centernot{\rightarrow} \struct{B}$ if no free set of $\struct{A}$ exists.  
We improve the original result by showing that the same holds if we replace ``a free set'' in the statement above with ``a non-empty union of free sets''. 
\begin{proposition}
\label{blockedSetProjection}
Let $\struct{A}$ be an instance of $\CSP(\struct{B})$ for some temporal structure $\struct{B}$ preserved by $\pp$ and let $S$ a union of free sets of $\struct{A}$. Then $\struct{A}$ has a solution if and only if  $\pr_{A\setminus S}(\struct{A})$ has a solution. 
\end{proposition}
\begin{proof}  \label{proof_blockedSetProjection}  
Let $F_{1},\dots,F_{k}$ be free sets of $\struct{A}$ and set $S\coloneqq F_{1} \cup \cdots \cup F_{k}$. Clearly, if $\struct{A}$ has a solution then so has $\pr_{A\setminus S}(\struct{A})$. For the converse, suppose that $\pr_{A\setminus S}(\struct{A})$ has a solution \blue{$f$}.  Let $S_j \coloneqq  F_j \setminus (F_1 \cup \cdots \cup F_{j-1})$ for every $j \in \{1,\dots,k\}$.
\blue{We claim that a map $f' \colon A \rightarrow \mathbb{Q}$ is a solution to $\struct{A}$ if $f'|_{A\setminus S}=f$,
$f'(S_1) < f'(S_2) < \cdots < f'(S_k) < f'(A \setminus S)$, and $f'$ is constant on $S_i$ for every $i\in [k]$.}
To verify this,
let $\boldf{s}$ be an arbitrary tuple from $R^{\struct{A}}\subseteq A^m$ such that, without loss of generality,
$
\{\boldf{s}\of{1},\dots,\boldf{s}\of{m}\} \cap S
= \{\boldf{s}\of{1},\dots,\boldf{s}\of{\ell}\} \neq \emptyset.
$
By the definition of $\pr_{A\setminus S}(\struct{A})$,  there is a tuple ${\boldf{t}} \in R^{\struct{B}}$ such that ${\boldf{t}}\of{i} = \blue{f}({\boldf{s}}\of{i})$ for every $i \in \{\ell+1,\dots,m\}$.
Since $F_1,\dots,F_k$ are free,
there are tuples ${\boldf{t}}_1,\dots,{\boldf{t}}_k \in R^{\struct{B}}$
such that, for every $i \in [k]$ and every $j \in [m]$, we have 
$
j \in \Minset({\boldf{t}}_{i}) \text{ if and only if } \boldf{s}\of{j} \in F_{i}.
$
For every $i \in [k]$ let $\alpha_i \in \Aut(\mathbb{Q};<)$ be such that $\alpha_i$ maps the minimal entry
of ${\boldf{t}}_i$ to $0$. The tuple
${\boldf{r}}_i \coloneqq  \pp(\alpha_i {\boldf{t}}_i,{\boldf{t}})$ is contained in $R^{\struct{B}}$ because
$R^{\struct{B}}$ is preserved by $\pp$.
It follows from the definition of $\pp$
that, for all $j \in [m]$, $j \in \Minset({\boldf{r}}_i)$ if and only if $\boldf{s}\of{j} \in F_i$.
Moreover, $({\boldf{r}}_i\of{\ell+1},\dots,{\boldf{r}}_i\of{m})$ and
$({\boldf{t}}\of{\ell+1},\dots,{\boldf{t}}\of{m})$ lie in the same orbit of $\Aut(\mathbb{Q};<)$.
Define ${\boldf{p}}_k,{\boldf{p}}_{k-1},\dots,{\boldf{p}}_1 \in \mathbb{Q}^m$ in this order as follows. Define ${\boldf{p}}_k \coloneqq  {\boldf{r}}_k$ and, for $i \in \{1,\dots,k-1\}$,
$
{\boldf{p}}_i \coloneqq  \pp(\beta_i {\boldf{r}}_{i},{\boldf{p}}_{i+1}) 
$
where $\beta_i \in \Aut(\mathbb{Q};<)$ is chosen such that
$\beta_{i}({\boldf{r}}_i\of{j}) = 0$ for all $j \in \Minset({\boldf{r}}_i)$.
We verify by induction that for all $i \in [k]$
\begin{enumerate}
\item ${\boldf{p}}_i$ is contained in $R^{\struct{B}}$.
\item $({\boldf{p}}_i\of{\ell+1},\dots,{\boldf{p}}_i\of{m})$,
$({\boldf{t}}\of{\ell+1},\dots,{\boldf{t}}\of{m})$ lie in the same orbit of $\Aut(\mathbb{Q};<)$.
\item $j \in \Minset({\boldf{p}}_i)$ if and only if $\boldf{s}\of{j} \in F_i$
for all $j \in [m]$.
\item ${\boldf{p}}_i\of{u}={\boldf{p}}_i\of{v}$ for all $a \in  \{i+1,\dots,k\}$ and $u,v \in [m]$ such that $\boldf{s}\of{u},\boldf{s}\of{v} \in S_a$.
\item ${\boldf{p}}_i\of{u} < {\boldf{p}}_i\of{v}$ for all $a,b \in \{i,i+1,\dots,k\}$ with $a<b$ and $u,v \in [m]$ such that $\boldf{s}\of{u} \in S_a$, $\boldf{s}\of{v} \in S_b$.
\end{enumerate}
For $i=k$, the items (1), (2), and (3) follow from
the respective property of ${\boldf{r}}_k$
and items (4) and (5) are trivial.
For the induction step and $i \in {[k-1]}$ we have that ${\boldf{p}}_{i} = \pp(\beta_{i} {\boldf{r}}_{i},{\boldf{p}}_{i+1})$ satisfies items (1) and (2)
because ${\boldf{p}}_{i+1}$ satisfies
items (1) and (2) by inductive assumption.
For item (3), note that $\Minset({\boldf{p}}_i) = \Minset({\boldf{r}}_{i})$.
Finally, if $\boldf{s}\of{u},\boldf{s}\of{v} \in S_{i+1} \cup \cdots \cup S_k$,
then ${\boldf{p}}_i\of{u} \leq {\boldf{p}}_i\of{v}$ if and only if
${\boldf{p}}_{i+1}\of{u} \leq {\boldf{p}}_{i+1}\of{v}$.
This implies items (4) and (5) by induction.
Note that \blue{$(f'(\boldf{s}\of{1}),\dots,f'(\boldf{s}\of{m}))$}
lies in the same orbit as ${\boldf{p}}_1$ and hence  is contained in $R^{\struct{B}}$.
\end{proof} 
A recursive application of Proposition~\ref{blockedSetProjection} shows the soundness of our choiceless version of the original algorithm which can be found in Figure~\ref{solvestar}.
Its completeness follows from the fact that every instance of a temporal CSP which has a solution must have a free set, namely the set of all variables which denote the minimal value in the solution. 
\begin{figure} 
\begin{center}   
\begin{algorithm}[H]   
\vspace{0.25em}
\SetNoFillComment 
\KwIn{An instance $\struct{A}$ of $\CSP(\struct{B})$ for a temporal structure $\struct{B}$}  
\KwOut{\emph{true} or \emph{false}}   
\blue{\While{$\struct{A}$ changes}{
		$S\gets $  the union of all free sets of $\struct{A}$\;  
		\If{$S= \emptyset$} {\Return \emph{false}} \Else {$\struct{A} \gets  \pr_{A\setminus S} \struct{A}$} 
	}  
	\Return \emph{true}\; }
\end{algorithm} 
%	\end{minipage} 
\end{center}
\caption{\label{solvestar} A choiceless algorithm that decides whether an instance of a temporal CSP with a template preserved by $\pp$ has a solution using an oracle for testing the containment in a free set.}	 
\end{figure}  
Suitable Ptime procedures for finding unions of free sets for TCSPs with a template  preserved by $\Min$, $\mi$, or $\mx$ exist by the results of \cite{bodirsky2010complexity}, and they generally exploit the algebraic structure of the CSP that is witnessed by one of these operations. 
We revisit them in Section~\ref{section_min}, Section~\ref{section_mi}, and Section~\ref{section_mx}. 

\blue{\begin{corollary} \label{fp_pp} \blue{Let  $\struct{B}$ be a temporal structure preserved by $\pp$ such that all proper projections of the relations of $\struct{B}$ are trivial.
Let  $\phi(x)$ be an $\FPR_2$ formula in the signature of $\struct{B}$ extended by a unary symbol $U$ such that, for every instance $\struct{A}$ of $\CSP(\struct{B})$ and every $U\subseteq A$, we have $(\struct{A};U) \models \phi(x)$ iff $x$ is not contained in a free set of the substructure of $\struct{A}$ on $U$.}
Then   $\struct{A} \models \exists x [\dfp_{U,x} \phi(x)](x)$ iff $\struct{A} \centernot{\rightarrow} \struct{B}$. 
\end{corollary} 
\begin{proof} Observe that, since all proper projections of the relations of $\struct{B}$ are trivial, for every $U\subseteq A$, the following two statements are equivalent: 
\begin{itemize}
\item $x$ is  contained in a free set of the substructure of $\struct{A}$ on $U$,
\item $x$ is  contained in a free set of the projection of $\struct{A}$ to $U$.
\end{itemize}	
By definition, $\struct{A} \models \exists x [\dfp_{U,x} \phi(x)](x)$ iff $x \in \Dfp \Op^\struct{A} \llbracket  \phi(\cdot)  \rrbracket$.
By the assumptions about $\phi$ and the observation above, for every $U\subseteq A$,
\[\Op^\struct{A}\llbracket  \phi(\cdot)  \rrbracket(U) = \{x\in A \mid  x\text{ is not contained in a free set of } \pr_{U}(\struct{A})\}.\]
By definition, $\Dfp \Op^\struct{A}\llbracket  \phi(\cdot)  \rrbracket$ is the limit of the sequence $U_{i+1}\coloneqq U_i \cap \Op^\struct{A}\llbracket  \phi(\cdot)  \rrbracket(U_i)$  with $U_0\coloneqq A$.
Now we can easily conclude the proof of the corollary.

``$\Leftarrow$'': This is a direct consequence of Proposition~\ref{blockedSetProjection}.

``$\Rightarrow$'': This follows from Proposition~\ref{blockedSetProjection} and the fact that the set of all elements taking the minimal value in a solution for an instance of $\CSP(\struct{B})$ is a free set of the instance.
\end{proof}}

\subsection{\blue{An FP algorithm for TCSPs preserved by min}}\label{section_min}
For TCSPs with a template preserved by $\Min$,  the algorithm in Figure~\ref{algo:FreeSetsMIN} can be used for finding  the union of all free sets due to the following lemma.  It can be proved by a simple induction using the observation that, for every $\boldf{s}\in R^{\struct{A}}$ and every $V \subseteq \{\boldf{s}\of{1},\dots,\boldf{s}\of{\ar(R)}\}$, \blue{the set $\Ideal{R}{V}(\boldf{s})$ is closed under taking unions.}
\begin{lemma}[\cite{bodirsky2010complexity}]  \label{UCfree} Let $\struct{A}$ be an instance of $\CSP(\struct{B})$ for a temporal structure $\struct{B}$ preserved by a binary operation $f$ such that  $f(0,0)=f(0,x)=f(x,0)$ for every $x>0$. 
Then the set returned by the algorithm in Figure~\ref{algo:FreeSetsMIN} is the union of all free sets of $\struct{A}$.  
\end{lemma}    
\begin{figure}  
\begin{center}   
\begin{algorithm}[H]   
\vspace{0.25em}
\SetNoFillComment 
\KwIn{An instance $\struct{A}$ of $\CSP(\struct{B})$ for a temporal structure $\struct{B}$} 
\KwOut{A subset $F\subseteq A$}  
$F \gets A$\; 
\While{$F$ changes}{ 
\ForAll{$\boldf{s}\in R^{\struct{A}}$}{
	\If{$\{\boldf{s}\of{1},\dots, \boldf{s}\of{\ar(R)}\}\cap F \neq \emptyset$}{   
		$\displaystyle F \gets  (F\setminus \{\boldf{s}\of{1},\dots, \boldf{s}\of{\ar(R)}\})\cup \bigcup \Ideal{R}{\{\boldf{s}\of{1},\dots, \boldf{s}\of{\ar(R)}\}\cap F}(\boldf{s})$  
	}
}
} 
\Return \textup{$F$}\; 
\end{algorithm}  	 
\end{center}
\caption{\label{algo:FreeSetsMIN} A choiceless algorithm that computes the union of all free sets   for temporal CSPs with a template preserved by the operation $\Min$.
}	
\end{figure}  
The following lemma in combination with \red{Theorem~\ref{REDUCTION}} shows that instead of
presenting an $\FP$ algorithm for each TCSP with a template preserved by $\Min$, it
suffices to present one for $\CSP(\mathbb{Q};\mathrm{R}^{\leq}_{\Min},<)$   
where 
\[
\mathrm{R}^{\leq}_{\Min}\coloneqq \{ (x,y,z)\in \mathbb{Q}^{3}\mid  y \leq x \vee z \leq x\}. 
\] 
\begin{lemma}
\label{lemma:RelBaseMin}  
A temporal relation is preserved by $\Min$ if and only if it is pp-definable in $(\mathbb{Q};<,\mathrm{R}^{\leq}_{\Min})$. 
\end{lemma}
\blue{\begin{remark} The importance of Lemma~\ref{lemma:RelBaseMin} lies in the fact that it presents a finite relational base for the clone generated by $\Aut(\mathbb{Q};<)\cup \{\Min\}$. Moreover, all proper projections of the relations are trivial.
This eliminates the necessity to use projections of instances for CSPs of temporal structures preserved by $\min$ (they can be replaced by substructures).
\end{remark}}
The following syntactic description of the temporal
relations preserved by $\min$  is due to Bodirsky, Chen, and Wrona \cite{bodirsky2014tractability}.
\begin{proposition}[\cite{bodirsky2014tractability}, page 9] \label{minsynt}
A temporal relation is preserved by $\Min$ if and only if it can be defined by a conjunction of formulas of the form  $z_{1} \circ_{1} x\vee \dots \vee  z_{n} \circ_{n} x$, where $\circ_{i}\in \{<,\leq\}$.  
\end{proposition}   
\begin{proof}[Proof of Lemma~\ref{lemma:RelBaseMin}]  \label{proof_RelationalBaseMin} 
The backward implication is a direct consequence of Proposition~\ref{minsynt}. 

For the forward implication, we show that every temporal relation defined by a formula of the form  $z_{1} \circ_{1} x\vee \dots \vee  z_{n} \circ_{n} x$, where $\circ_{i}\in \{<,\leq\}$, has a pp-definition in   $ (\mathbb{Q};\mathrm{R}^{\leq}_{\Min},<)$. Then the statement follows from Proposition~\ref{minsynt}.
A pp-definition $\phi'_{n}(x,z_{1},\dots,z_{n})$ for the relation defined by  $z_{1} \leq x \vee \dots \vee z_{n} \leq x$   can be obtained by the following simple induction. 

In the \emph{base case} $n=3$ we set $\phi'_{3}(x,z_{1},z_{2})\coloneqq \mathrm{R}^{\leq}_{\Min}(x,z_{1},z_{2})$. 

In the \emph{induction step}, we suppose that $n>3$ and that  $\phi'_{n-1}$ is a pp-definition for the relation defined by $z_{1} \leq x \vee \dots \vee z_{n-1} \leq x$. Then 
\begin{align*}\phi'_{n}(x,z_{1},\dots,z_{n})\coloneqq  \exists h  \big(\mathrm{R}^{\leq}_{\Min}(x,z_{1},h)\wedge   \phi'_{n-1}(h,z_{2},\dots,z_{n})\big)
\end{align*}  
is a pp-definition of the relation defined by $z_{1} \leq x \vee \dots \vee z_{n} \leq x$. 
Finally,
\begin{align*} \phi_{n}(x,z_{1},\dots,z_{n})=\exists z'_{1},\dots,z'_{n}  ( \phi_{n}'(x,z'_{1},\dots,z'_{n})   \wedge    \bigwedge_{i\in I}z_{i}< z'_{i}   \wedge   \bigwedge_{i\notin I}z'_{i}=z_{i}   )
\end{align*}  
is a pp-definition of the relation defined by $z_{1} \circ_{1} x \vee \dots \vee z_{n} \circ_{n} x$  where $\circ_{i}$ equals $<$ if $i\in I$ and  $\leq$ otherwise. 
\end{proof}  
In the case of $\CSP(\mathbb{Q};<,\mathrm{R}^{\leq}_{\Min})$ a procedure from \cite{bodirsky2010complexity} for finding free sets can be directly implemented in FP. 
\begin{proposition} \label{LFPminCorrectnessSoundness}  
$\CSP(\mathbb{Q};<,\mathrm{R}^{\leq}_{\Min})$ is expressible in $\FP$.  
\end{proposition}  
\begin{proof} 
\red{Recall that $\struct{B}\coloneqq (\mathbb{Q};<,\mathrm{R}^{\leq}_{\Min})$ is preserved by $\pp$.} \blue{Since all proper projections of the relations of $\struct{B}$ are trivial, $\struct{B}$ satisfies the prerequisites of Corollary~\ref{fp_pp}.}
Our aim is to construct a formula $\phi(x)$ satisfying the requirements of Corollary~\ref{fp_pp} by rewriting the algorithm in Figure~\ref{algo:FreeSetsMIN} in the syntax of FP.
In addition to the unary fixed-point variable $U$ coming from Corollary~\ref{fp_pp}, we introduce a fresh unary fixed-point variable $V$ for the union $F$ of all free sets of the current projection.
The algorithm in Figure~\ref{algo:FreeSetsMIN} computes $F$ using a  deflationary induction where parts of the domain which cannot be contained in any free set are gradually cut off.
Thus, we may choose $\phi(x)$ to be of the form $\neg [\dfp_{V,x} \psi(x)](x)$ for some formula $\psi(x)$ testing whether whenever the variable 
$x$ is contained in $ \{\boldf{s}\of{1},\dots, \boldf{s}\of{\ar(R)}\}\cap F$  
for some constraint $R(\boldf{s})$, then it is also contained in the largest min-set within $\{\boldf{s}\of{1},\dots, \boldf{s}\of{\ar(R)}\}\cap F$.
It is easy to see that
\begin{align}
\bigcup\Ideal{<}{\{\boldf{s}\of{1},\boldf{s}\of{2}\}\cap F}(\boldf{s}) & =  (\{\boldf{s}\of{1},\boldf{s}\of{2} \}\cap F) \setminus \{\blue{\boldf{s}\of{2}}\} 
\label{minctr2} \\
\bigcup\Ideal{\mathrm{R}^{\leq}_{\Min}}{\{\boldf{s}\of{1},\boldf{s}\of{2}, \boldf{s}\of{3}\}\cap F}(\boldf{s}) & = \begin{cases}
\{\boldf{s}\of{1},\boldf{s}\of{2}, \boldf{s}\of{3}\}\cap F & \text{ if } F\cap \{ \boldf{s}\of{2},\boldf{s}\of{3} \} \neq \emptyset \\
\emptyset & \text{ otherwise.} \end{cases}
\label{minctr1}
\end{align}  
This leads to the formula
\[
\psi(x)   \coloneqq     U(x) \wedge \forall y,z     \big(( \overset{\eqref{minctr2}}{\overbrace{   U(y) \Rightarrow \neg  (y < x)}})    
\wedge   (  \overset{\eqref{minctr1}}{\overbrace{(U(y) \wedge U(z))  \Rightarrow (V(y)\vee  V(z) \vee \neg  \mathrm{R}^{\leq}_{\Min}(x,y,z)) }} ) \big). 
\]
Therefore, the statement of the proposition  follows from Corollary~\ref{fp_pp}. 
\end{proof} 
To increase readability, the formula $\phi(x)$ in the proof of Proposition~\ref{LFPminCorrectnessSoundness}
can we rewritten into the following formula, using the conversion rule from $\dfp$ to $\ifp$: 
\begin{align*} 
[\ifp_{V,x} U(x) \Rightarrow \exists y,z \big( (U(y) \wedge  y < x) \vee ( U(y)\wedge U(z) \wedge V(y) \wedge V(z) \wedge \mathrm{R}^{\leq}_{\Min}(x,y,z))     \big)](x) .
\end{align*} 
\subsection{\blue{An FP algorithm for TCSPs preserved by mi}}\label{section_mi}
For TCSPs with a template preserved by $\mi$,  the algorithm in Figure~\ref{algo:FreeSetsMIN} can be used for finding  the union of all free sets due to the following lemma.  It can be proved by a simple induction using the observation that, for every $\boldf{s}\in R^{\struct{A}}$  and every $V \subseteq \{\boldf{s}\of{1},\dots,\boldf{s}\of{\ar(R)}\}$, the set $ \Filter{R}{V}(\boldf{s})\cup \{\emptyset\}$ is closed under taking  intersections. 
\begin{lemma}[\cite{bodirsky2010complexity}]  \label{corrFree} Let $\struct{A}$ be an instance of $\CSP(\struct{B})$ for a temporal structure $\struct{B}$ preserved by a binary operation $f$ such that $f(0,0)<f(0,x)$ and $f(0,0)<f(x,0)$ for every $x>0$. 
Then the set returned by the algorithm in Figure~\ref{algo:FreeSetsMI} is the union of all free sets of $\struct{A}$.  
\end{lemma}

\begin{figure} 
\begin{center}    
\begin{algorithm}[H]  
\vspace{0.25em}
\SetNoFillComment 
\KwIn{An instance $\struct{A}$ of $\CSP(\struct{B})$ for a temporal structure $\struct{B}$} 
\KwOut{A subset $F\subseteq A$} 
$F \gets $ the empty subset of $A$\;
\ForAll{$x\in A$}{
$F_{x} \gets \{x\}$\;
\While{$F_{x}$ changes}{ 
	\ForAll{$\boldf{s}\in R^{\struct{A}}$ such that $\{\boldf{s}\of{1},\dots,\boldf{s}\of{\ar(R)}\}\cap F_{x} \neq \emptyset$}{
		\If{$\Filter{R}{F_{x}\cap \{\boldf{s}\of{1},\dots,\boldf{s}\of{\ar(R)}\}}(\boldf{s}) \neq \emptyset$}{ $\displaystyle F_{x}\gets F_{x}\cup \bigcap \Filter{R}{F_{x}\cap \{\boldf{s}\of{1},\dots,\boldf{s}\of{\ar(R)}\}}(\boldf{s})$\; 
		}
		\Else{$F_{x} \gets$ the empty subset of $A$} 
	} 
}
$F \gets F\cup F_{x}$\;
} 
\Return $F$\;	
%	}  
\end{algorithm}   
\end{center}
\caption{\label{algo:FreeSetsMI}  A choiceless algorithm that computes the union of all free sets for temporal CSPs with a template preserved by a binary operation $f$ such that $f(0,0)<f(0,x)$ and $f(0,0)<f(x,0)$ for every $x>0$.}	
\end{figure}  

The following lemma in combination with 
\red{Theorem~\ref{REDUCTION}}
%Theorem~\ref{theorem:reduction_pp_constructible} 
shows that instead of
presenting an $\FP$ algorithm for each TCSP with a template preserved by $\mi$, it
suffices to present one for $ \CSP(\mathbb{Q};\mathrm{R}_{\mi},\mathrm{S}_{\mi},{\neq})$ where  
\begin{align*} \mathrm{R}_{\mi}\coloneqq   \{ (x,y,z)\in \mathbb{Q}^{3}\mid y < x \vee z \leq x \}  \quad \text{ and } \quad 
\mathrm{S}_{\mi}\coloneqq  \{ (x,y,z)\in \mathbb{Q}^{3}\mid x \neq y \vee z \leq x \}.
\end{align*}  

\begin{lemma}
\label{RelationalBaseMi}
A temporal relation is preserved by $\mi$ if and only if it is pp-definable in $(\mathbb{Q};\mathrm{R}_{\mi},\mathrm{S}_{\mi},{\neq})$.  
\end{lemma}
\blue{\begin{remark}
Analogously to Lemma~\ref{lemma:RelBaseMin}, Lemma~\ref{RelationalBaseMi} presents a finite relational base for the clone generated by $\Aut(\mathbb{Q};<)\cup \{\mi\}$.
Moreover, all proper projections of the relations are trivial.
This eliminates the necessity to use projections of instances for CSPs of temporal structures preserved by $\mi$ (they can be replaced by substructures).
\end{remark}}
The following syntactic description is due to Micha\l\ Wrona.
\begin{proposition}[see, e.g., \cite{bodirsky2021complexity}] \label{misynt}A temporal relation is preserved by $\mi$ if and only if it can be defined as conjunction of formulas of the form \[  z_{1} \neq x\vee\dots \vee  z_{n} \neq x \vee y_{1} < x \vee \dots \vee  y_{m} <x \vee  y \leq x \tag{$\dagger$}  \label{mi_clause} \] where the last disjunct $y \leq x$ can be omitted.  
\end{proposition}
\begin{proof}[Proof of Lemma~\ref{RelationalBaseMi}] 
The backward direction is a direct consequence of Proposition~\ref{misynt}. 

For the forward direction, it suffices by Proposition~\ref{misynt} 
to show that every temporal relation defined by a formula of the form  \eqref{mi_clause}, where the last disjunct $y \leq x$ can be omitted, has a pp-definition in  \blue{$(\mathbb{Q};\mathrm{R}_{\mi},\mathrm{S}_{\mi},{\neq})$}.
We prove the statement by induction on $m$ and $n$.
Note that both ${\leq}$ and ${<}$ have a pp-definition in $(\mathbb{Q};\mathrm{R}_{\mi},\mathrm{S}_{\mi},{\neq})$. 
For $m,n\geq 0$, let $R_{m,n}$ denote the $(m+n+2)$-ary temporal relation defined by the formula  \eqref{mi_clause}, where we assume that all variables are distinct and in their respective order $x,y_{1},\dots,y_{m},z_{1},\dots,z_{n},y$. 

In the \emph{base case} \blue{$m+n=1$},  we set $\phi_{1,0}(x,y_{1},y)=\mathrm{R}_{\mi}(x,y_{1},y)$ and $\phi_{0,1}(x,z_{1},y)=\mathrm{S}_{\mi}(x,z_{1},y)$. 
The \emph{induction step} is divided into three individual claims.
\begin{claim} \label{claim_mi_m} If $\phi_{m-1,0}(x,y_{1},\dots,y_{m-1},y)$ is a pp-definition of $R_{m-1,0}$, then  
\begin{align*} 
\phi_{m,0}(x,y_{1},\dots,y_{m},y) \coloneqq\exists h  \big(\phi_{1,0}(h,y_{m},y) 
\wedge \phi_{m-1,0}(x,y_{1},\dots,y_{m-1},h) \big)
\end{align*} 
is a pp-definition of $R_{m,0}$.  
\end{claim}
\begin{proof}[Proof of Claim~\ref{claim_mi_m}]  ``$\Rightarrow$'':  
Let ${\boldf{t}}\in R_{m,0}$.  We have to show that ${\boldf{t}}$ satisfies $\phi_{m,0}$. 
In case that ${\boldf{t}}\of{x}> \Min({\boldf{t}}\of{y_{1}},\dots,{\boldf{t}}\of{y_{m-1}})$ we set $h\coloneqq  {\boldf{t}}\of{y}$.  Otherwise, ${\boldf{t}}\of{x}> {\boldf{t}}\of{y_{m}}$ or ${\boldf{t}}\of{x}\geq {\boldf{t}}\of{y}$, in which case we set $h\coloneqq {\boldf{t}}\of{x}$.    

``$\Leftarrow$'': Suppose for contradiction that ${\boldf{t}}\notin R_{m,0}$ satisfies $\phi_{m,0}$ and that this is witnessed by some $h \in \mathbb{Q}$. Since ${\boldf{t}}\of{x}\leq \Min({\boldf{t}}\of{y_{1}},\dots,{\boldf{t}}\of{y_{m-1}})$, we must have ${\boldf{t}}\of{x}\geq h$.  But since ${\boldf{t}}\of{x}\leq {\boldf{t}}\of{y_{m}}$ and ${\boldf{t}}\of{x}<{\boldf{t}}\of{y}$, we get a contradiction to $\phi_{1,0}(h,{\boldf{t}}\of{y_{m}},{\boldf{t}}\of{y})$ being satisfied.
\end{proof} 
\begin{claim} \label{claim_mi_n} If $\phi_{0,n-1}(x,z_{1},\dots,z_{n-1},y)$ is a pp-definition of $R_{0,n-1}$, then 
\begin{align*} \phi_{0,n}(x,z_{1},\dots,z_{n},y) \coloneqq \exists h     \big(\phi_{0,1}(h,z_{n},y)   
\wedge \phi_{0,n-1}(x,z_{1},\dots,z_{n-1},h) \big)
\end{align*} 
is a pp-definition of $R_{0,n}$. 
\end{claim}
The proofs of this claim and the next claim are similar to the proof of the previous claim and omitted.  
\begin{claim} \label{claim_mi_mn} Let $\phi_{m,0}(x,y_{1},\dots,y_{m},y)$ and $\phi_{0,n}(x,z_{1},\dots,z_{n},y)$ be pp-definitions of $R_{m,0}$ and $R_{0,n}$, respectively. Then  
\begin{align*} \blue{\phi_{m,n}(x,y_{1},\dots,y_{m},z_{1},\dots,z_{n},y) =  \exists h   \big(\phi_{0,n}(x,z_{1},\dots,z_{n},h)   \wedge  \phi_{m,0}(h,y_{1},\dots,y_{m},y) \big)}
\end{align*}	is a pp-definition of $R_{m,n}$.
\end{claim}
This completes the proof of the lemma because the last clause $y \leq x$ in \eqref{mi_clause} can be easily eliminated using an additional existentially quantified variable and the relation ${<}$.
\end{proof}  
\begin{proposition}\label{LFPmiCorrectnessSoundness}  
$\CSP(\mathbb{Q};\mathrm{R}_{\mi},\mathrm{S}_{\mi},{\neq})$ is expressible in $\FP$.  
\end{proposition}  
% 
%As in the case of the  structure considered in Proposition~\ref{LFPminCorrectnessSoundness}, the temporal structure above was particularly chosen so that proper projections of all its relations are trivial. Thus, taking a projection of an instance of  $ \CSP(\mathbb{Q};\mathrm{R}_{\mi},\mathrm{S}_{\mi},{\neq})$  is equivalent w.r.t.\ satisfiability to taking a substructure on the same set.
%
\begin{proof}  
Let $\struct{B} \coloneqq  (\mathbb{Q};\mathrm{R}_{\mi},\mathrm{S}_{\mi},{\neq})$.  
\blue{Recall that $\struct{B}$ is preserved by $\pp$. Since all proper projections of the relations of $\struct{B}$ are trivial, $\struct{B}$ satisfies the prerequisites of Corollary~\ref{fp_pp}.}
Our aim is  to construct a formula $\phi(x)$ satisfying the requirements of Corollary~\ref{fp_pp} by rewriting the algorithm in Figure~\ref{algo:FreeSetsMI} in the syntax of FP.
In addition to the unary fixed-point variable $U$ coming from Corollary~\ref{fp_pp}, we introduce a fresh binary fixed-point variable $V$ for the free set propagation relation $\{(x,y) \mid y\in F_{x} \}$ computed during the algorithm in Figure~\ref{algo:FreeSetsMIN}.
The computation takes place through inflationary induction where a pair $(x,y)$ is added to the relation if the containment of $x$ in a free set implies the containment of $y$. The algorithm concludes that a variable $x$ is contained in a free set if there are no $x_1,\dots,x_k \in F_x$ whose containment in the same free set would lead to a contradiction.
Note that $\neq$ is the only relation without a constant polymorphism among the relations of $\struct{B}$, i.e., the only relation for which the \emph{if} condition in the algorithm in Figure~\ref{algo:FreeSetsMI} can evaluate as false.
Thus $\phi(x)$  may be chosen to be  of the form 
\[  U(x) \Rightarrow \exists y, z  \big(  [\ifp_{V,x,y} \psi(x,y)](x,y)      \wedge [\ifp_{V,x,z} \psi(x,z)](x,z)  \wedge {{\neq}}(y,z) \big)\] 
for some formula $\psi(x,y)$ defining the (transitive) free-set propagation relation. 
\blue{The notation ${\neq}(y,z)$ should not be confused with $\neg (y = z)$: the former is an atomic $\tau$-formula because $\neq$ is part of the signature \red{of $\struct{B}$}, while the latter is a valid first-order formula because equality is a built-in part of first-order logic.}
It is easy to see that
\begin{align*}  
\bigcap \Filter{\mathrm{R}_{\mi}}{F_{x}\cap \{\boldf{s}\of{1}, \boldf{s}\of{2},\boldf{s}\of{3} \}}(\boldf{s}) = & 
\left\{\renewcommand\arraystretch{1}\setlength{\arraycolsep}{1pt}
\begin{array}{ll} (F_{x}\cap \{\boldf{s}\of{1}, \boldf{s}\of{2},\boldf{s}\of{3} \}) \cup \{\boldf{s}\of{3}\} & \text{if } F_{x}\cap \{\boldf{s}\of{1},\boldf{s}\of{3}\} = \{\boldf{s}\of{1}\}, \\ 
F_{x}\cap \{\boldf{s}\of{1}, \boldf{s}\of{2},\boldf{s}\of{3} \} & \text{otherwise,}  
\end{array} 
\right.   \\
\bigcap \Filter{\mathrm{S}_{\mi}}{F_{x}\cap \{\boldf{s}\of{1}, \boldf{s}\of{2},\boldf{s}\of{3} \}}(\boldf{s}) = & 
\left\{\renewcommand\arraystretch{1}\setlength{\arraycolsep}{1pt}
\begin{array}{ll} (F_{x}\cap \{\boldf{s}\of{1}, \boldf{s}\of{2},\boldf{s}\of{3} \}) \cup \{\boldf{s}\of{3}\} & \text{if } F_{x}\cap \{\boldf{s}\of{1},\boldf{s}\of{2},\boldf{s}\of{3}\} = \{\boldf{s}\of{1},\boldf{s}\of{2}\}, \\ 
F_{x}\cap \{\boldf{s}\of{1}, \boldf{s}\of{2},\boldf{s}\of{3} \} & \text{otherwise.}  
\end{array} 
\right.  
\end{align*} 
This leads to the formula
\begin{align*}
\psi(x,y)\coloneqq  \ &  U(x) \wedge U(y) \wedge \big((x=y)  \\ &  \vee  \exists a,b, c   \big(U(a) \wedge U(b)\wedge U(c)     \wedge V(x,a) \wedge V(x,b)  \wedge  (\mathrm{R}_{\mi}(a,c,y) \vee \mathrm{S}_{\mi}(a,b,y)) \big) \big)
\end{align*}
Now the  statement of the proposition follows from Corollary~\ref{fp_pp}.
\end{proof}  
\subsection{\blue{An FP algorithm for TCSPs preserved by ll}}  
If a temporal structure $\struct{B}$ is preserved by $\elel$, then it is also preserved by $\lex$, but not necessarily by $\pp$ \cite{bodirsky2010complexity}. In general, the choiceless procedure based on Proposition~\ref{blockedSetProjection} is then not correct for $\CSP(\struct{B})$. 
We present a modified version of this procedure, motivated by the approach of repeated contractions from \cite{bodirsky2010fast}, and show that this version is correct for $\CSP(\struct{B})$.

Let $\struct{A}$ be an instance of $\CSP(\struct{B})$.  
We repeatedly simulate on $\struct{A}$ the choiceless procedure based on Proposition~\ref{blockedSetProjection}  and, each time a union $S$ of free sets is computed, we contract in $\struct{A}$ all variables in every free set within $S$ that is minimal among all existing free sets in the current projection with respect to set inclusion. 
This loop terminates when a fixed-point is reached, where $\struct{A}$ no longer changes, in which case we accept.
The resulting algorithm can be found in Figure~\ref{algo:master_ll}. 
\begin{definition}  A free set (Definition~\ref{def:freeset}) of an instance $\struct{A}$ of a temporal CSP is called \emph{\blue{inclusion-minimal}} if it does not contain any other free set of $\struct{A}$ as a proper subset. 
\end{definition} 
\blue{\begin{theorem}  \label{theorem:ll_algo_correctness}
The algorithm in Figure~\ref{algo:master_ll} is correct for CSPs of temporal structures preserved by  $\elel$.
\end{theorem}
Theorem~\ref{theorem:ll_algo_correctness} is proved using the following three lemmata. 
First, Lemma~\ref{atomic} explains why we may \red{(and in fact why we \emph{must})} contract \blue{inclusion-minimal} free sets.}
\begin{lemma}
\label{atomic}
Let $\struct{B}$ be a temporal structure preserved by $\lex$ and $\struct{A}$ an instance of $\CSP(\struct{B})$. 
Then all variables in an \blue{inclusion-minimal} free set of $\struct{A}$ denote the same value in every solution for $\struct{A}$. 
\end{lemma}
\begin{proof}  Let $F$ be an \blue{inclusion-minimal} free set of $\struct{A}$.  Suppose that   $\struct{A}$ has a solution $f$. We assume that $|F|>1$; otherwise, the statement is trivial. Let $F'$ be the set of all elements from $F$ that denote the minimal value in $f$ among all elements from $F$. Suppose that  $F\setminus F'$ is not empty. 
We show that then $F'$ is a free set that is properly contained in $F$.
Let $R$ be an arbitrary symbol from the signature of $\struct{B}$. We set $k\coloneqq \ar(R)$.
Let $\boldf{s}\in R^{\struct{A}}$ be such that $\{\boldf{s}\of{1},\dots,\boldf{s}\of{k}\}\cap F' \neq \emptyset$.
Without loss of generality, let $1\leq k_{F'}\leq k_{F} \leq k$ be such that $\{\boldf{s}\of{1},\dots, \boldf{s}\of{k_{F'}}\} = \{\boldf{s}\of{1},\dots, \boldf{s}\of{k}\}\cap F'$ and $\{\boldf{s}\of{1},\dots, \boldf{s}\of{k_F}\}= \{\boldf{s}\of{1},\dots, \boldf{s}\of{k}\}\cap F$.
There exists a tuple ${\boldf{t}} \in R^{\struct{B}}$ such that $\Minset(\boldf{t}) = [k_F]$ because $F$ is a free set.
Also, by the definition of $F'$, there exists a tuple $\boldf{t}' \in R^{\struct{B}}$ such that $\Minset\big((\boldf{t}'\of{1},\dots, \boldf{t}'\of{k_{F}})\big)=  [k_{F'}] $.
Let $\boldf{t}''\coloneqq \lex(\boldf{t},\boldf{t}')$.
It is easy to see that $\Minset(\boldf{t}'')=[k_{F'}]$. \blue{Since $\boldf{s}$ was chosen arbitrarily, we conclude that $F'$ is a free set, a contradiction to inclusion-minimality of $F$. Thus $F'=F$.}
\end{proof}
Next, Lemma~\ref{lemma:atomic_intersections} guarantees that distinct \blue{inclusion-minimal} free sets are disjoint. 
\blue{\begin{lemma} \label{lemma:atomic_intersections}
Let $\struct{B}$ be a temporal structure preserved by $\lex$, and $\struct{A}$ an instance of $\CSP(\struct{B})$. 
If $F_1, F_2 $ are free sets of $\struct{A}$ such that $F_1\cap F_2\neq \emptyset$, then $F_1\cap F_2$ is a free set of $\struct{A}$.
\end{lemma} 
\begin{proof}
Let $F_1,F_2 $ be free sets of $\struct{A}$ such that $F_1\cap F_2\neq \emptyset$. 
Let $R$ be a symbol from the signature of $\struct{B}$. We set $k\coloneqq \ar(R)$.
Let $\boldf{s}\in R^{\struct{A}} $ be such that $\{\boldf{s}\of{1},\dots,\boldf{s}\of{k}\}\cap F_1\cap F_2 \neq \emptyset$.
%
%	If $\{\boldf{s}\of{1},\dots,\boldf{s}\of{k}\}\cap (F\setminus F')=\emptyset$ or $\{\boldf{s}\of{1},\dots,\boldf{s}\of{k}\}\cap (F' \setminus F)=\emptyset$, then there is a tuple ${\boldf{t}}\in R^{\struct{B}}$ such that $\Minset(\boldf{t}) = \{i\in [k] \mid \boldf{s}\of{i}\in \{\boldf{s}\of{1},\dots,\boldf{s}\of{k}\}\cap F\cap F' \}$ because $F$ and $F'$ are both free sets. Now suppose that $\{\boldf{s}\of{1},\dots,\boldf{s}\of{k}\}\cap (F \cup F')$ is contained neither in $F$ nor   in $F'$. 
%
Then there are tuples $ \boldf{t}_1,\boldf{t}_2\in R^{\struct{B}}$ such that $\Minset(\boldf{t}_j) = \{i\in [k] \mid \boldf{s}\of{i}\in \{\boldf{s}\of{1},\dots,\boldf{s}\of{k}\}\cap F \}$ for $j\in[2]$  because $F_1$ and $F_2$ are both free sets.
Let $\boldf{t}\coloneqq \lex( \boldf{t}_1,\boldf{t}_2)$.
Then $\Minset(\boldf{t}') = \{i\in [k] \mid \boldf{s}\of{i}\in \{\boldf{s}\of{1},\dots,\boldf{s}\of{k}\}\cap F_1\cap F_2 \}$. 
Since $\boldf{s}$ was chosen arbitrarily, we conclude that $F_1\cap F_2$ is a free set.
\end{proof}}
\blue{Finally, Lemma~\ref{blockedSetProjection2} is an analogue to Proposition~\ref{blockedSetProjection} for the operation $\elel$ instead of $\pp$.
It allows us to recursively reduce an instance of the CSP to a smaller one, unless a certain condition involving free sets is not met, in which case there is no solution and we may reject.
The proof is quite similar to the one of Proposition~\ref{blockedSetProjection}, because $\elel$ behaves similarly to the operation $\pp$, except that $\elel$ is injective.
The injectivity of $\elel$ has several consequences which need to be handled carefully, e.g.,  the fact that we can no longer work with overlapping free sets.  We also need to do some bookkeeping on the kernel of the solution.}
%
%Every instance of a temporal CSP which is satisfiable must have a free set, namely the set of all variables which denote the minimal value in the solution.
%%
%Per definition it must also have a free set which is \blue{inclusion-minimal}.
%%
%Moreover, the projection of the given instance to the complement of the union of all \blue{inclusion-minimal} free sets is again an instance which has a solution.
%%
%This observation combined with Lemma~\ref{atomic} shows that the algorithm in Figure~\ref{algo:master_ll} is complete. 
%%
%Its soundness can be shown by a recursive application of Lemma~\ref{blockedSetProjection2} to the sequence of structures produced during the last run of the inner loop of the algorithm.
%   
\begin{figure} 
\begin{center}   
\begin{minipage}{\textwidth} 
\begin{algorithm}[H]  
\vspace{0.25em}
\SetNoFillComment 
\KwIn{An instance $\struct{A}$ of $\CSP(\struct{B})$ for a temporal structure $\struct{B}$}  
\KwOut{\emph{true} or \emph{false}}     
$C \gets$ the empty binary relation over $A$\;
\While{$\struct{A}$ changes}{  
$\struct{A}' \gets\struct{A}$\;
\While{$\struct{A}'$ changes}{
	\ForAll{$a,b\in A'$}{\If{$a,b$ are in the same \blue{inclusion-minimal} free set of $\struct{A}'$}{$ C \gets C \cup \{(a,b)\}$}} 
	$S\gets $ the union of all \blue{inclusion-minimal} free sets of $\struct{A}'$\; 
	$\struct{A}' \gets\pr_{A'\setminus S}(\struct{A}')$\;  
}  
\If{$A'\neq \emptyset$} {\Return \emph{false}} \Else {$\struct{A} \gets \con_{C}(\struct{A})$}  
}
\Return \emph{true}\; 
\end{algorithm}  
\end{minipage} 
\end{center}
\caption{\label{algo:master_ll}  A choiceless algorithm for temporal CSPs with a template preserved by $\elel$ using an oracle for the computation of \blue{inclusion-minimal} free sets.}	 
\end{figure}  
\begin{lemma}
\label{blockedSetProjection2}
%	Let $\struct{A}$ be an instance of $\CSP(\struct{B})$ for some temporal structure $\struct{B}$ preserved by $\elel$.
%
\blue{Let $\struct{A}$ be an instance of $\CSP(\struct{B})$ for a temporal structure $\struct{B}$ preserved by $\elel$.
Let $S$ be the union of all  inclusion-minimal free sets of $\struct{A}$.
Let $C$ be a binary relation over $A$ such that 	 $\struct{A} = \con_{C}(\struct{A})$ and $C\cap S^2$ consists of the pairs of elements contained in the same inclusion-minimal free set of $\struct{A}$.
If $\pr_{A\setminus S}(\struct{A})$ has a solution with kernel $C \cap (A\setminus S)^2$, then $ \struct{A}$ has a solution with kernel $C$.}
\end{lemma} 
\begin{proof}[Proof of Lemma~\ref{blockedSetProjection2}]   
Let $F_{1},\dots,F_{k}$ be the \blue{inclusion-minimal} free sets of $\struct{A}$ and set $S\coloneqq F_{1} \cup \cdots \cup F_{k}$. 
\blue{%Clearly, if $\struct{A}$ has a solution $f$ with kernel $C$, then $f|_{A\setminus S}$ is a solution to $\pr_{A\setminus S}(\struct{A})$ with kernel $C \cap (A\setminus S)^2$.
Suppose that $\pr_{A\setminus S}(\struct{A})$ has a solution $f \colon A \to \mathbb{Q}$ with $\ker f  = C \cap (A\setminus S)^2$.  
Since $F_{1},\dots,F_{k}$ are \blue{inclusion-minimal}, by Lemma~\ref{lemma:atomic_intersections}, we have $F_{i}\cap F_{j} = \emptyset$ for all  distinct  $i,j \in [k]$.
Let $f' \colon A \rightarrow \mathbb{Q}$ be such that $f'|_{A\setminus S}=f$,
$f'(F_1) < f'(F_2) < \cdots < f'(F_k) < f'(A \setminus S)$,
and $f'$ is constant on $F_i$ for every $i\in [k]$.
We claim that $f'$ is a solution to $\struct{A}$ with $\ker f'=C$. 
To verify this,
let $\boldf{s}$ be an arbitrary tuple from $R^{\struct{A}}\subseteq A^m$ such that, without loss of generality,
$  \{\boldf{s}\of{1},\dots,\boldf{s}\of{m}\} \cap S
= \{\boldf{s}\of{1},\dots,\boldf{s}\of{\ell}\} \neq \emptyset.$
By the definition of $\pr_{A\setminus S}(\struct{A})$,  there is a tuple $\boldf{t} \in R^{\struct{B}}$ such that $\boldf{t}\of{i} = f({\boldf{s}}\of{i})$ for every $i \in \{\ell+1,\dots,m\}$.
Since $\struct{A} = \con_{C}(\struct{A})$,  we have $\boldf{t}\of{u}=\boldf{t}\of{v}$ whenever $(\boldf{s}\of{u},\boldf{s}\of{v})\in C$.
For $u,v \geq  \ell+1$, we even have $\boldf{t}\of{u}=\boldf{t}\of{v}$ if and only if $(\boldf{s}\of{u},\boldf{s}\of{v})\in C$ because $ \ker f = C \cap (A\setminus S)^2$.
Since $F_1,\dots,F_k$ are free sets,
there are tuples ${\boldf{t}}_1,\dots,{\boldf{t}}_k \in R^{\struct{B}}$
such that, for every $i \in [k]$ and every $j \in [m]$, we have 
$j \in \Minset({\boldf{t}}_{i}) $ if and only if $ \boldf{s}\of{j} \in F_{i}.$
Again, for every $i \in [k]$, we have $\boldf{t}_{i}\of{u}=\boldf{t}_i\of{v}$ whenever $(\boldf{s}\of{u},\boldf{s}\of{v})\in C$ because $\struct{A} = \con_{C}(\struct{A})$.
This time, we do not obtain a necessary and sufficient condition concerning the entries with indices $u,v \geq  \ell+1$.
For every $i \in [k]$, let $\alpha_i \in \Aut(\mathbb{Q};<)$ be such that $\alpha_i$ maps the minimal entry
of $\boldf{t}_i$ to $0$.
The tuple
${\boldf{r}}_i \coloneqq  \elel(\alpha_i \boldf{t}_i,\boldf{t})$ is contained in $R^{\struct{B}}$ because
$R^{\struct{B}}$ is preserved by $\elel$.
It follows from the definition of $\elel$
that, for every $j \in [m]$, $j \in \Minset({\boldf{r}}_i)$ if and only if $\boldf{s}\of{j} \in F_i$.
Moreover, $({\boldf{r}}_i\of{\ell+1},\dots,{\boldf{r}}_i\of{m})$ and
$({\boldf{t}}\of{\ell+1},\dots,{\boldf{t}}\of{m})$ lie in the same orbit of $\Aut(\mathbb{Q};<)$,
and $\boldf{r}_{i}\of{u}=\boldf{r}_i\of{v}$ whenever $(\boldf{s}\of{u},\boldf{s}\of{v})\in C$.}
Define ${\boldf{p}}_k,{\boldf{p}}_{k-1},\dots,{\boldf{p}}_1 \in \mathbb{Q}^m$ in this order as follows. Define ${\boldf{p}}_k \coloneqq  {\boldf{r}}_k$ and, for $i \in \{1,\dots,k-1\}$,
${\boldf{p}}_i \coloneqq  \elel(\beta_i {\boldf{r}}_{i},{\boldf{p}}_{i+1}) $
where $\beta_i \in \Aut(\mathbb{Q};<)$ is chosen such that
$\beta_{i}({\boldf{r}}_i\of{j}) = 0$ for all $j \in \Minset({\boldf{r}}_i)$.
We verify by induction that for all $i \in [k]$
\begin{enumerate}
\item ${\boldf{p}}_i$ is contained in $R^{\struct{B}}$;
\item $({\boldf{p}}_i\of{\ell+1},\dots,{\boldf{p}}_i\of{m})$;
$({\boldf{t}}\of{\ell+1},\dots,{\boldf{t}}\of{m})$ lie in the same orbit of $\Aut(\mathbb{Q};<)$;
\item $j \in \Minset({\boldf{p}}_i)$ if and only if $\boldf{s}\of{j} \in F_i$
for all $j \in [m]$;
\item ${\boldf{p}}_i\of{u}={\boldf{p}}_i\of{v}$ for all $a \in  \{i+1,\dots,k\}$ and $u,v \in [m]$ such that $\boldf{s}\of{u},\boldf{s}\of{v} \in S_a$;
\item ${\boldf{p}}_i\of{u} < {\boldf{p}}_i\of{v}$ for all $a,b \in \{i,i+1,\dots,k\}$ with $a<b$ and $u,v \in [m]$ such that $\boldf{s}\of{u} \in F_a$, $\boldf{s}\of{v} \in F_b$.
\end{enumerate}
For $i=k$, the items (1), (2), and (3) follow from
the respective property of ${\boldf{r}}_k$
and items (4) and (5) are trivial.
For the induction step and $i \in {[k-1]}$ we have that ${\boldf{p}}_{i} = \elel(\beta_{i} {\boldf{r}}_{i},{\boldf{p}}_{i+1})$ satisfies items (1) and (2)
because ${\boldf{p}}_{i+1}$ satisfies
items (1) and (2) by inductive assumption.
For item (3), note that $\Minset({\boldf{p}}_i) = \Minset({\boldf{r}}_{i})$.
Finally, if $\boldf{s}\of{u},\boldf{s}\of{v} \in F_{i+1} \cup \cdots \cup F_k$,
then ${\boldf{p}}_i\of{u} \leq {\boldf{p}}_i\of{v}$ if and only if
${\boldf{p}}_{i+1}\of{u} \leq {\boldf{p}}_{i+1}\of{v}$.
This implies items (4) and (5) by induction.
Note that $(s'(\boldf{s}\of{1}),\dots,s'(\boldf{s}\of{m}))$
lies in the same orbit as $\boldf{p}_1$ and hence  is contained in $R^{\struct{B}}$. Moreover, it follows from the injectivity of $\elel$  that $\boldf{p}_1\of{u}=\boldf{p}_1\of{v}$ if and only if $(\boldf{s}\of{u},\boldf{s}\of{v})\in C$. 
\end{proof} 
\blue{\begin{proof}[Proof of Theorem~\ref{theorem:ll_algo_correctness}]
Let $\struct{A}$ be an instance of $\CSP(\struct{B})$ for a temporal structure $\struct{B}$ preserved by $\elel$.
First, suppose that $\struct{A}$ has a solution $f$. If $A=\emptyset$, then the algorithm trivially accepts $\struct{A}$ and there is nothing to \red{be shown}. So suppose that $A\neq \emptyset$.
For an arbitrary $\emptyset \subsetneq A' \subseteq A$, let $\struct{A}'\coloneqq \pr_{A'}(\struct{A})$.
By definition, $f'\coloneqq f|_{A'}$ is a solution to $\struct{A}'$.
Let $F$ be the set of all elements of $A'$ on which $f'$ takes the minimal value.
Then clearly $F$ is a free set of $\struct{A}'$.
By definition, $F$ contains an inclusion-minimal free set of $\struct{A}'$ as a subset.
Recall that, since $\struct{B}$ is preserved by $\elel$, it is also preserved by $\lex$. 
Thus, by Lemma~\ref{atomic}, $f'(a)=f'(b)$ whenever $a$ and $b$ are contained in the same inclusion-minimal free set of $\struct{A}'$.
Since $A'$ was chosen arbitrarily, it follows by induction over the inner loop of the the algorithm that \emph{false} is not returned at the end of the inner loop, and that the relation $C$ computed during the inner loop satisfies $C\subseteq \ker f$.
Since $C\subseteq \ker f$, we have that $f$ is also a solution to $\con_{C}(\struct{A})$.
It now follows by induction over the outer loop of the algorithm using the argument above for the inner loop that \emph{true} is returned at the end of the outer loop. 

Now suppose that $\struct{A}$ is accepted by the algorithm.
Let $C$ be the binary relation computed during the algorithm.
Clearly, the algorithm computes the same binary relation when given $\con_{C}(\struct{A})$ as an input, and also accepts on this input.
Moreover, every solution to $\con_{C}(\struct{A})$ is also a solution to $\struct{A}$.
Thus, it is enough to show that $\con_{C}(\struct{A})$ has a solution.
Without loss of generality, we may assume that $\struct{A} = \con_{C}(\struct{A})$.
The inner loop of the algorithm produces a sequence $A_1,\dots,A_\ell$ of subsets of $A$ where $A_1 \coloneqq A$, and for every $i<\ell$ the set
$A_{i+1}$ is the subset of $A_{i}$
where we remove all elements which are contained in an inclusion-minimal free set of  $\pr_{A_i}(\struct{A})$.
Since $\struct{A}$ is accepted, it must be the case that $A_\ell = \emptyset$.
Hence, $\pr_{A_{\ell}}(\struct{A})$ trivially has a solution whose kernel is empty.
Now the existence of a solution for $\struct{A}$ follows by induction on $i\in [\ell]$ using Lemma~\ref{blockedSetProjection2} with the relation $C\cap A^2_{i-1}$ in the induction step from $i $ to $ i-1$.
\end{proof}}

The following lemma in combination with 
\red{Theorem~\ref{REDUCTION}}
%Theorem~\ref{theorem:reduction_pp_constructible} 
shows that instead of
presenting an $\FP$ algorithm for each TCSP with a template preserved by $\elel$, it
suffices to present one for $\CSP(\mathbb{Q};\mathrm{R}_{\elel},\mathrm{S}_{\elel},\neq)$ where 
\begin{align*}
\mathrm{R}_{\elel} & \coloneqq  \{(x,y,z)\in \mathbb{Q}^{3} \mid  y < x \vee z<x  \vee x=y=z \} \\\text{ and }\quad  \mathrm{S}_{\elel} & \coloneqq  \{ (x,y,u,v)\in \mathbb{Q}^{4} \mid x\neq y \vee u \leq v\}. 
\end{align*}  

\begin{lemma} 
\label{RelationalBaseLl}    
A temporal relation is preserved by $\elel$ if and only if it is pp-definable in $(\mathbb{Q};\mathrm{R}_{\elel},\mathrm{S}_{\elel},\neq)$.
\end{lemma}
\blue{\begin{remark} 	Analogously to Lemma~\ref{lemma:RelBaseMin} and  Lemma~\ref{RelationalBaseMi},  Lemma~\ref{RelationalBaseLl}  presents a finite relational base for the clone generated by $\Aut(\mathbb{Q};<)\cup \{\elel\}$.
Moreover, all proper projections of the relations are trivial.
However, we cannot use this fact to eliminate the use of projections of instances in the algorithm in Figure~\ref{algo:master_ll}.
The reason is that the necessary contractions of variables due to Lemma~\ref{atomic} might introduce new tuples to relations with non-trivial projections. \red{For example, }   $\pr_{\{3,4\}}(\con_{\{(1,2)\}}\mathrm{S}_{\elel})$ equals $\leq$.
\end{remark}}
The following syntactic description is due to Bodirsky, K\'ara, and Mottet. 
\begin{proposition}[\cite{bodirsky2021complexity}] \label{llSynt} A temporal relation is preserved by $\elel$ if and only if it can be defined by a conjunction of formulas of the form 
\begin{align*}    &    x_{1} \neq y_{1}\vee \dots\vee x_{m}\neq y_{m}  \vee
z_{1} < z\vee \dots \vee z_{n} < z  \vee   (z=z_{1}=\dots=z_{n})   
\end{align*}
where the last  disjunct   $(z=z_{1}=\dots=z_{n})$ can be omitted.  
\end{proposition}  
\begin{proof}[Proof of Lemma~\ref{RelationalBaseLl}]  
The backward implication is a direct consequence of Proposition~\ref{llSynt}. 

For the forward implication, we show that every temporal relation defined by a formula of the form  $ x_{1} \neq y_{1}\vee \dots\vee x_{m}\neq y_{m} \vee z_{1} < z\vee \dots \vee z_{n}<z  \vee   (z=z_{1}=\dots=z_{n})$, where the last disjunct $(z=z_{1}=\dots=z_{n})$ can be omitted, has a pp-definition in   $(\mathbb{Q};\mathrm{R}_{\elel},\mathrm{S}_{\elel},\neq)$. Then the statement follows from Proposition~\ref{llSynt}.	
We prove the statement by induction on $m$ and $n$. 
Note that both ${\leq}$ and ${<}$ have a pp-definition in $(\mathbb{Q};\mathrm{R}_{\elel},\mathrm{S}_{\elel},{\neq})$.   
For $m,n\geq 0$, let $R_{m,n}$ denote the $(2m+n+1)$-ary relation with the syntactic definition by a single formula  from Proposition~\ref{llSynt}  where we assume that all variables are distinct, $x_{1},\dots,x_{m}$ refer to the odd entries among $1,\dots,2m$,
$y_{1},\dots,y_{m}$ refer to the even entries among $1,\dots,2m$, and 
$z,z_{1},\dots,z_{n}$ refer to the entries $2m+1,\dots,2m+n+1$.  
\blue{In the \emph{base cases}}, we set $\phi_{1,1}(x_{1},y_{1},z,z_{1}) \coloneqq  \mathrm{S}_{\elel}(x_{1},y_{1},z_{1},z)$ and $\phi_{0,2}(z,z_{1},z_{2}) \coloneqq  \mathrm{R}_{\elel}(z,z_{1},z_{2})$. 
\begin{claim} \label{claim_ll_m} If $\phi_{m-1,1}(x_{1},y_{1},\dots,x_{m-1},y_{m-1},z,z_{1})$ is a pp-definition of $R_{m-1,1}$, then 
\begin{align*}  \phi_{m,1}(x_{1},y_{1},\dots,x_{m},y_{m},z,z_{1}) \coloneqq  \exists a,b     
\big( & \phi_{m-1,1}(x_{1},y_{1},\dots,x_{m-1},y_{m-1},a,b) \\
& \wedge \phi_{1,1}(x_{m},y_{m},b,a)  \wedge \phi_{1,1}(a,b,z,z_{1})  \big)
\end{align*}
is a pp-definition of $R_{m,1}$.   
\end{claim}
\begin{proof}[Proof of Claim~\ref{claim_ll_m}] 	``$\Rightarrow$'': Arbitrarily choose ${\boldf{t}}\in R_{m,1}$. We verify that ${\boldf{t}}$ satisfies $\phi_{m,1}$.
%  
%		For convenience, we refer to the entries of ${\boldf{t}}$ through the free variables of $\phi_{m,n}$ while assuming that these are distinct and in their respective order. 
%
If ${\boldf{t}}\of{x_{i}}\neq {\boldf{t}}\of{y_{i}}$ for some $1 \leq i \leq m-1$, then choose any $b>a$. 
If ${\boldf{t}}\of{x_{m}}\neq {\boldf{t}}\of{y_{m}}$, then we pick any $a,b\in \mathbb{Q}$ with $a>b$. Otherwise, ${\boldf{t}}\of{z}\geq {\boldf{t}}\of{z_{1}}$ and we pick any $a,b\in \mathbb{Q}$ with $a=b$.

``$\Leftarrow$'':	Suppose that ${\boldf{t}}\notin R_{m,1}$ satisfies $\phi_{m,1}$ with some witnesses $a,b$. Since ${\boldf{t}}\of{x_{i}}={\boldf{t}}\of{y_{i}}$ for every $1 \leq i \leq m$, we have $a \geq b$ and $b\geq a$, thus $a=b$. But then $\phi_{1,1}(a,b,{\boldf{t}}\of{z},{\boldf{t}}\of{z_{1}})$ cannot hold, a contradiction.  
\end{proof}
\noindent It is easy to see that $R_{m,0}$ has the pp-definition \[ \phi_{m,0}(x_{1},y_{1},\dots,x_{m},y_{m})=\exists a,b   \big((b>a)    \wedge  \phi_{m,1}(x_{1},y_{1},\dots,x_{m},y_{m},a,b)\big). \]

\begin{claim} \label{claim_ll_n}
If $\phi_{0,n-1}(z,z_{1},\dots,z_{n-1})$ is a pp-definition of $R_{0,n-1}$, then  
\begin{align*} \phi_{0,n}(z,z_{1},\dots,z_{n}) \coloneqq   \exists h     \big(\phi_{0,2}(h,z_{n-1},z_{n})  \wedge \phi_{0,n-1}(z,z_{1},\dots,z_{n-2},h)  \big)
\end{align*} 
is a pp-definition of $R_{0,n}$. 
\end{claim} 
The proofs of this claim and the next claim are similar to the proof of the previous claim and omitted. 

\begin{claim} \label{claim_ll_mn} Let $\phi_{m,1}(x_{1},y_{1},\dots,x_{m},y_{m},z,z_{1})$ and $\phi_{0,n}(z,z_{1},\dots,z_{n})$ be pp-defi\-nitions of $R_{m,1}$ and $R_{0,n}$, respectively, then 
\begin{align*} \phi_{m,n}(x_{1},y_{1},\dots,x_{m},y_{m},z,z_{1},\dots,z_{n}) \coloneqq   \exists h  \big(\phi_{0,n}(h,z_{1},\dots,z_{n})  
\wedge \phi_{m,1}(x_{1},y_{1},\dots,x_{m},y_{m},z,h) \big)
\end{align*}
is a pp-definition of $R_{m,n}$.
\end{claim}
This completes the proof of the lemma because the part $(z=z_{1}=\dots=z_{n})$ in the formula from Proposition~\ref{llSynt} can be easily eliminated using an additional existentially quantified variable and the relation ${<}$.
\end{proof}

In the case of $\CSP(\mathbb{Q};\mathrm{R}_{\elel},\mathrm{S}_{\elel},\neq)$, we can use the same FP procedure for finding free sets from \cite{bodirsky2010complexity} that we use for instances of
$\CSP(\mathbb{Q};\mathrm{R}_{\mi},\mathrm{S}_{\mi},{\neq})$.

\begin{proposition} \label{LFPllCorrectnessSoundness}  
$\CSP(\mathbb{Q};\mathrm{R}_{\elel},\mathrm{S}_{\elel},\neq)$ is expressible in $\FP$.  
\end{proposition} 
%
%As with the structures considered in  Proposition~\ref{LFPminCorrectnessSoundness} and Proposition~\ref{LFPmiCorrectnessSoundness} we have that proper projections of the relations of the temporal structure above are trivial.
%%
%However, this is not the case for proper projections of contractions of the relations, e.g., the relation  $\pr_{\{3,4\}}(\con_{\{(1,2)\}}\mathrm{S}_{\elel})$ equals $\leq$.
%%
%Thus, in the context of this particular CSP, we cannot ignore projections if we want to obtain an FP sentence using the algorithm in Figure~\ref{algo:master_ll}.
%
\begin{proof}  
Let $\struct{B}\coloneqq (\mathbb{Q};\mathrm{R}_{\elel},\mathrm{S}_{\elel},\neq)$, and let $\struct{A}$ be an arbitrary instance of $\CSP(\struct{B})$.
Note that the operation $\elel$ satisfies the requirements of Lemma~\ref{corrFree}.	
Thus, the algorithm in Figure~\ref{algo:FreeSetsMI} can be used for computation of free sets for instances of $\CSP(\struct{B})$.
Also note that the algorithm builds free sets from singletons using only necessary conditions for containment.
Thus, for every $x\in A$, the set $F_x$ computed during the algorithm in Figure~\ref{algo:FreeSetsMI} is an \blue{inclusion-minimal} free set iff it is non-empty and does not contain any other non-empty set of the form $F_y$ as a proper subset.
It follows that two variables $x,y \in A$ are contained in the same \blue{inclusion-minimal} free set if $x\in F_y$, $y \in F_x$, and whenever $z\in F_x$ for some $z\in A$, then $x\in F_z$. 

Now suppose that there exists an $\FP$ formula $\phi(x,y)$ in the signature of $\struct{B}$ extended by binary symbols $E,C$  such that, for every instance $\struct{A}$ of $\CSP(\struct{B})$ and all $E,C \subseteq A^2$, $(\struct{A};E,C) \models \phi(x,y)$ iff $x,y$ are contained in the same \blue{inclusion-minimal} free set of $\pr_{U}(\con_{C}(\struct{A}))$ where 
$U= A \setminus \{x \in A \mid (x,x) \in E\}$.
Then, given \blue{$C$} as a parameter, $(\struct{A};C) \models [\ifp_{E,x,y}  \phi(x,y) ](x,y) $ iff $x,y$ are contained in the same \blue{inclusion-minimal} free set at some point during the iteration of the inner loop of the algorithm in Figure~\ref{algo:master_ll}.
Consequently, $\struct{A} \centernot{\rightarrow} \struct{B}$ if and only if $\exists x.\, \neg [\ifp_{C,x,y} [\ifp_{E,x,y}  \phi(x,y) ](x,y) ] (x,x)$, by the soundness and completeness of the algorithm in Figure~\ref{algo:master_ll}. 
We can obtain such a formula $\phi$ by translating the algorithm in Figure~\ref{algo:FreeSetsMI} into the syntax of FP and applying the reasoning from the first paragraph of this proof:
\begin{align*}
\phi(x,y)   \coloneqq\ &  \neg E(x,x) \wedge \neg E(y,y) \wedge   [\ifp_{V,x,y} \psi(x,y)](x,y)   \wedge [\ifp_{V,y,x}   \psi(y,x)](y,x)    \\ & \wedge   \forall a, b  \big(  ([\ifp_{V,x,a} \psi(x,a)](x,a)      \wedge [\textup{ifp}_{V,x,b} \psi(x,b)](x,b) )  \Rightarrow  \neg\, {\neq}(a,b)    \big)    \\ & \wedge \forall z \big( [\ifp_{V,x,z} \psi(x,z)](x,z) \Rightarrow  [\ifp_{V,z,x} \psi(z,x)](z,x)  \big) 
%\\[0.25\baselineskip]
%\psi(x,y)\coloneqq  \ &  \neg U(x,x) \wedge \neg U(y,y) \wedge \big((x=y) \vee  E(x,y) \\ & \vee \exists a \exists b \exists c\exists d   \big(\neg U(a,a) \wedge \neg U(b,b)\wedge \neg U(c,c)  \wedge \neg U(d,d)   \wedge V(x,a) \wedge V(x,b) \\ &  \wedge  (\mathrm{R}_{\elel}(a,c,y) \vee \mathrm{R}_{\elel}(a,y,c)    ) \big) \big) \qedhere
\end{align*}
where $\psi$ can be defined similarly as in Proposition~\ref{LFPmiCorrectnessSoundness}
except that each subformula of the form $U(x)$ must be replaced with $\neg E(x,x)$, and taking into consideration all projections of contractions of relations with respect to $C$.
\end{proof} 

\section{A TCSP in FPR$_\text{2}$ which is not in FP\label{section_inexpressibility}}
Let $\mathrm{X}$ be the temporal relation as defined in the introduction.
In this section, we show that $\CSP(\mathbb{Q};\mathrm{X})$ is  expressible in $\FPR_2$ (Proposition~\ref{LFPmxCorrectnessSoundness}) but inexpressible in $\FPC$ (Theorem~\ref{FPxorsat}).  
\subsection{\blue{An FPR$_\text{2}$ algorithm for TCSPs  preserved by mx}}\label{section_mx} 
It is straightforward to verify that the relation $\mathrm{X}$ is preserved by the operation $\mx$ introduced in Section~\ref{sect:pols}~\cite{bodirsky2010complexity}. 
For TCSPs with a template preserved by $\mx$,  the algorithm in Figure~\ref{algo:FreeSetsMX} can be used for finding  the union of all free sets due to the following lemma.  It can be proved by a simple induction using the observation that, for every $\boldf{s}\in R^{\struct{A}}$, \red{the set} $\Minsystem_{R}(\boldf{s})\cup \{\emptyset\} $  is closed under taking symmetric difference. 
\blue{\red{Note that we can interpret every entry of $\boldf{s}$ as a $\{0,1\}$-variable whose value represents whether or not the entry is contained in a particular min-set.} Then closure under symmetric difference implies that $\Minsystem_{R}(\boldf{s})\cup \{\emptyset\}$ is the solution set of a system of \blue{mod-2} equations of the form  $A \boldf{s}=\boldf{0}$. 
In the algorithm in Figure~\ref{algo:FreeSetsMX} we simply take the largest such system.}
\begin{lemma}[\cite{bodirsky2010complexity}] \label{freeMX}  
Let $\struct{B}$ be a template of a temporal CSP which is preserved by $\mx$. 
Let $\struct{A}$ be an instance of $\CSP(\struct{B})$. 
Then the set returned by the algorithm in Figure~\ref{algo:FreeSetsMX} is the union of all free sets of $\struct{A}$. 
\end{lemma}  
\begin{figure} 
\begin{center}     
\begin{algorithm}[H]  
\vspace{0.25em}
\SetNoFillComment 
\KwIn{An instance $\struct{A}$ of $\CSP(\struct{B})$ for a temporal structure $\struct{B}$} 
\KwOut{A subset $F\subseteq A$} 
$E \gets  $ the empty set of \blue{mod-2} equations\;
\ForAll{$\boldf{s}\in R^{\struct{A}}$} 
{   
\blue{\ForAll{$I\subseteq [\ar(R)]$} 
{   
	\If{$ M \cap \{\boldf{s}\of{i} \mid i\in I \}$  has even cardinality for every $M\in \Minsystem_{R}(\boldf{s})$
	}{$E \gets E \cup \{ \sum_{i\in I}\boldf{s}\of{i}=0\}$}
} }
}
$F \gets$ the empty subset of $A$\; 
\ForAll{$x\in A$}{ \If{$E\cup\{x=1\}$ has a solution over $\mathbb{Z}_{2}$
}{$F \gets  F\cup \{x\}$}
}
\Return $F$\;
%}  
\end{algorithm}   
\end{center}
\caption{ \label{algo:FreeSetsMX}  A choiceless algorithm that computes the union of all free sets  for   temporal CSPs with a template preserved by $\mx$.}
\end{figure}  
The following lemma in combination with
\red{Theorem~\ref{REDUCTION}} %Theorem~\ref{theorem:reduction_pp_constructible} 
shows that instead of
presenting an $\FPR_2$ algorithm for each TCSP with a template preserved by $\mx$, it
suffices to present one for $\CSP(\mathbb{Q};\mathrm{X})$.
\begin{lemma} 
\label{lemma:RelationalBaseMx}    A temporal relation is preserved by $\mx$ if and only if it has a pp-definition in $(\mathbb{Q};\mathrm{X})$.
\end{lemma} 
\blue{\begin{remark}
Analogously to Lemma~\ref{lemma:RelBaseMin}, Lemma~\ref{RelationalBaseMi}, and Lemma~\ref{RelationalBaseLl}, Lemma~\ref{lemma:RelationalBaseMx} presents a finite relational base for the clone generated by $\Aut(\mathbb{Q};<)\cup \{\mx\}$.
Moreover, all proper projections of the relations are trivial.
This eliminates the necessity to use projections of instances for CSPs of temporal structures preserved by $\mx$ (they can be replaced by substructures).
\end{remark}} 
Recall the min-indicator function $\chi$ from Definition~\ref{minindicator}.   
\begin{definition} \label{ordxordefinable} 
For a temporal relation $R$, we set $\MS(R)\coloneqq \chi(R)\cup\{{\boldf{0}}\}$. 
A \emph{basic Ord-Xor relation} is a temporal relation $R$ for which there exists a homogeneous system ${{A}}{\boldf{x}}={\boldf{0}}$ of \blue{mod-2} equations such that $\MS(R)$ is the solution set of ${{A}}{\boldf{x}}={\boldf{0}}$,  and $R$ contains all tuples ${\boldf{t}}\in \mathbb{Q}^{n}$ with ${{A}}\chi({\boldf{t}})={\boldf{0}}$. 
If the system ${{A}}{\boldf{x}}={\boldf{0}}$ for the relation specifying a basic Ord-Xor relation consists of a single equation $\sum_{i\in I} x_{i}=0$ for  $I\subseteq [n]$, then we denote this relation by $R^{\mx}_{I,n}$.
A \emph{basic Ord-Xor formula} is 
a $\{<\}$-formula $\phi(x_1,\dots,x_n)$ that defines
a basic Ord-Xor relation.
An \emph{Ord-Xor formula} is a conjunction of basic Ord-Xor formulas. 
\end{definition}  
The next lemma is a straightforward consequence of Definition~\ref{ordxordefinable}.
\begin{lemma} \label{general}   If $\{\sum_{i\in I_{j}}x_{i}=0 \mid j\in J\}$
is the homogeneous system of \blue{mod-2} equations for a basic Ord-Xor relation $R$, then $R = \bigcap_{j\in J}  R^{\mx}_{I_{j},n}.$ 
\end{lemma}  
\blue{If a temporal relation $R$ is preserved by $\mx$, then  $\MS(R)$ is closed under the \blue{mod-2} addition and forms a linear subspace of $\{0,1\}^{\ar(R)}$ ~\cite{bodirsky2010complexity}.
In general, $R$ does not contain all tuples over $\mathbb{Q}$ whose min-tuple is contained in this subspace, e.g., the $6$-ary temporal relation defined by $\mathrm{X}(x_1,x_2,x_3)\wedge \mathrm{X}(x_4,x_5,x_6)$ does not contain $(0,0,1,1,1,1)$. 
Therefore, basic Ord-Xor formulas are not sufficient for describing all temporal relations preserved by $\mx$. One must instead consider general Ord-Xor formulas.}
The following syntactic description is due to Bodirsky, Chen, and Wrona.
\begin{theorem}[\cite{bodirsky2014tractability}, Thm. 6] \label{partialbasismx} 
A temporal relation can be defined by an Ord-Xor formula if and only if it is preserved by  $\mx$.
\end{theorem}

%As usual, a solution to a homogeneous system of Boolean equations is called \emph{trivial} if all variables take value $0$, and \emph{non-trivial} otherwise. 
%

%

%
\begin{proof}[Proof of Lemma~\ref{lemma:RelationalBaseMx}] 
We have already mentioned that $\mathrm{X}$ is preserved by $\mx$, and hence all relations that are pp-definable in $(\mathbb{Q};\mathrm{X})$ are preserved by $\mx$ as well. For the converse direction, we  show that $R^{\mx}_{I,n}$ has a pp-definition in $(\mathbb{Q};\mathrm{X})$ for every \red{integer $n >0$} and $I\subseteq [n]$; then the claim follows from Theorem~\ref{partialbasismx} together with Lemma~\ref{general}. Note that we trivially have a pp-definition of ${<}$ in  $(\mathbb{Q};\mathrm{X})$ via \blue{$\phi^{\mx}_{\!\{2\},2}(x,y)\coloneqq \mathrm{X}(x,x,y)$}. 
We first show that the relations  
\begin{align*} 
R^{\mx}_{\!\{1\},3} = \ & \mathrm{R}_{\Min} =   \{{\boldf{t}}\in \mathbb{Q}^{3} \mid {\boldf{t}}\of{2} < {\boldf{t}}\of{1} \vee {\boldf{t}}\of{3} < {\boldf{t}}\of{1}\} \\
\text{and }\quad R^{\mx}_{[3],4}=  \ & \{{\boldf{t}}\in \mathbb{Q}^{4} \mid {\boldf{t}}\of{4} < \Min({\boldf{t}}\of{1},{\boldf{t}}\of{2},{\boldf{t}}\of{3})   \vee ({\boldf{t}}\of{1},{\boldf{t}}\of{2},{\boldf{t}}\of{3})\in \mathrm{X}  \}
\end{align*} 
have pp-definitions in $(\mathbb{Q};\mathrm{X})$ and then proceed with treating the other relations of the form $R^{\mx}_{I,n}$.
\begin{claim}
The following primitive positive formula defines
$R^{\mx}_{[3],4}$ in $(\mathbb{Q};\mathrm{X})$.
\begin{align*} \phi^{\mx}_{[3],4}(x_{1},x_{2},x_{3},x_{4}) \coloneqq   \exists & x_{1}',x_{2}',x_{3}',x_{1}'',x_{2}'',x_{3}''    \big(  \blue{x_{4}< x_{1}''
\wedge x_{4}< x_{2}'' \wedge x_{4}< x_{3}''}\\   
& \wedge \mathrm{X}(x_{1}',x_{2}',x_{3}')  \wedge \mathrm{X}(x_{1},x_{1}',x_{1}'') \wedge   \mathrm{X}(x_{2},x_{2}',x_{2}'')   \wedge \mathrm{X}(x_{3},x_{3}',x_{3}'') \big)
\end{align*}
\end{claim}
\begin{proof}
``$\Rightarrow$'': We first prove that every ${\boldf{t}}\in R^{\mx}_{[3],4}$ satisfies $\phi^{\mx}_{[3],4} $. 

\textit{Case 1:  $({\boldf{t}}\of{1},{\boldf{t}}\of{2},{\boldf{t}}\of{3})\in \mathrm{X}$}. We choose  witnesses for the quantifier-free part of $\phi^{\mx}_{[3],4}$ as follows: $x_{1}'\coloneqq {\boldf{t}}\of{1},$ $ x_{2}'\coloneqq {\boldf{t}}\of{2}$, $x_{3}'\coloneqq {\boldf{t}}\of{3}$, and for $x_{1}'',x_{2}'',x_{3}''$  we choose values arbitrarily such that  $\blue{\Max({\boldf{t}}\of{1},{\boldf{t}}\of{2},{\boldf{t}}\of{3},{\boldf{t}}\of{4})} < \Min(x_{1}'',x_{2}'',x_{3}'')$. It is easy to see that this choice satisfies the quantifier-free part of $\phi^{\mx}_{[3],4}$. 

\textit{Case 2: $({\boldf{t}}\of{1},{\boldf{t}}\of{2},{\boldf{t}}\of{3})\notin \mathrm{X}$}. We have ${\boldf{t}}\of{4} < \Min({\boldf{t}}\of{1},{\boldf{t}}\of{2},{\boldf{t}}\of{3})$ by the definition of $R^{\mx}_{[3],4}$.  By symmetry, it suffices consider 
%all the possible weak linear orders on the set $\{{\boldf{t}}\of{1},{\boldf{t}}\of{2},{\boldf{t}}\of{3}\}$ and choose suitable values for the witnesses. Without loss of generality, we only have 
the following three subcases. 

\textit{Subcase 2.i: ${\boldf{t}}\of{3}< {\boldf{t}}\of{2} <{\boldf{t}}\of{1}$.} We choose $x_{1}'=x_{2}'=x_{1}''=x_{2}''=x_{3}''\coloneqq {\boldf{t}}\of{3}$ and  $x_{3}'\coloneqq {\boldf{t}}\of{1}$. 

\textit{Subcase 2.ii: ${\boldf{t}}\of{3} < {\boldf{t}}\of{1}={\boldf{t}}\of{2}$.} We choose the same witnesses as in the previous case.

\textit{Subcase 2.iii: ${\boldf{t}}\of{1}={\boldf{t}}\of{2}={\boldf{t}}\of{3}$}. We choose any combination of  $x_{1}',x_{2}',x_{3}'$,  
$x_{1}'',x_{2}'',x_{3}''$ that satisfies 
${\boldf{t}}\of{4}<x_{1}'=x_{2}'=x_{1}''=x_{2}''<x_{3}'=x_{3}''<{\boldf{t}}\of{1}$.

In each of the subcases~2.i-iii above, our choice satisfies the quantifier-free part of $\phi^{\mx}_{[3],4}$. 

``$\Leftarrow$'': Suppose for contradiction that there exists a tuple ${\boldf{t}}\notin R^{\mx}_{[3],4}$ that satisfies $\phi^{\mx}_{[3],4} $. 
Then $({\boldf{t}}\of{1},{\boldf{t}}\of{2},{\boldf{t}}\of{3})\notin \mathrm{X}$ and ${\boldf{t}}\of{4}\geq \Min({\boldf{t}}\of{1},{\boldf{t}}\of{2},{\boldf{t}}\of{3})$.  
Consider the witnesses $x_{1}',x_{2}',x_{3}',x_{1}'',x_{2}'',x_{3}'' $ for the fact that ${\boldf{t}}$ satisfies  $\phi^{\mx}_{[3],4}$.
Without loss of generality, we only have the following three cases. 

\textit{Case 1: ${\boldf{t}}\of{1}>{\boldf{t}}\of{2}>{\boldf{t}}\of{3}$}. We  have $x_{3}'={\boldf{t}}\of{3}$  because $({\boldf{t}}\of{3},x_{3}',x_{3}'')\in \mathrm{X}$ and $x_{3}''>{\boldf{t}}\of{4}\geq \Min({\boldf{t}}\of{1},{\boldf{t}}\of{2},{\boldf{t}}\of{3})={\boldf{t}}\of{3}.$

\textit{Subcase 1.i: $x_{3}'>\Min(x_{1}',x_{2}')$}. We have $x_{1}'=x_{2}'<x_{3}'$, because $(x_{1}',x_{2}',x_{3}')\in \mathrm{X}$. This implies $x_{1}''=x_{1}'$, because $x_{1}'<x_{3}'={\boldf{t}}\of{3}<{\boldf{t}}\of{1}$ and $({\boldf{t}}\of{1},x'_{1},x''_{1})\in \mathrm{X}$. But then 
$x_{1}''<{\boldf{t}}\of{3}\leq {\boldf{t}}\of{4}$, a contradiction.  

\textit{Subcase 1.ii: $x_{3}'=\Min(x_{1}',x_{2}',x_{3}')$}. 
Either $x_{1}'=x_{3}'<x_{2}'$ or $x_{2}'=x_{3}'<x_{1}'$ because $(x_{1}',x_{2}',x_{3}')\in \mathrm{X}$.

\textit{Subcase 1.ii.a: $x_{1}'=x_{3}'$}.  We have  $x_{1}''=x_{1}' $ because $x'_{1} = x'_{3}={\boldf{t}}\of{3}<{\boldf{t}}\of{1}$ and  $({\boldf{t}}\of{1},x'_{1},x''_{1})\in \mathrm{X}$. 	But then  $x_{1}''={\boldf{t}}\of{3}\leq {\boldf{t}}\of{4}$, a contradiction. 

\textit{Subcase 1.ii.b: $x_{2}'=x_{3}'$}. We have $x_{2}''=x_{2}' $ because $x'_{2} = x'_{3}={\boldf{t}}\of{3}<{\boldf{t}}\of{2}$ and  $({\boldf{t}}\of{2},x'_{2},x''_{2})\in \mathrm{X}$. But then $x_{2}''={\boldf{t}}\of{3}\leq {\boldf{t}}\of{4}$, a contradiction. 

\textit{Case 2: ${\boldf{t}}\of{1}={\boldf{t}}\of{2}>{\boldf{t}}\of{3}$}. We obtain a contradiction similarly as in the previous case.

\textit{Case 3: ${\boldf{t}}\of{1}={\boldf{t}}\of{2}={\boldf{t}}\of{3}$}. We must have $x_{3}'={\boldf{t}}\of{3}$, $x_{2}'={\boldf{t}}\of{2}$ and $x_{1}'={\boldf{t}}\of{1}$ because $\Min(x_{1}'',x_{2}'',x_{3}'')>{\boldf{t}}\of{4}\geq {\boldf{t}}\of{1}= {\boldf{t}}\of{2}={\boldf{t}}\of{3}.$ But then $(x_{1}',x_{2}',x_{3}')\notin \mathrm{X}$, a contradiction.

In all three cases above, we get a contradiction which means that there is no tuple ${\boldf{t}}\notin R^{\mx}_{[3],4}$ that satisfies $\phi^{\mx}_{[3],4}({\boldf{t}}\of{1},{\boldf{t}}\of{2},{\boldf{t}}\of{3},{\boldf{t}}\of{4})$.
\end{proof}

It is easy to see that the pp-formula 
\begin{align*}
\phi^{\mx}_{[2],3}(x_{1},x_{2},x_{3})\coloneqq \exists h   \big( \phi^{\mx}_{[3],4}(x_{1},x_{2},h,x_{3}) \wedge (h>x_{1}) \big)
\end{align*}  is equivalent to $(x_{1}>x_{3}\wedge x_{2}>x_{3})\vee x_{1}=x_{2}$. 

\begin{claim}
The following pp formula defines   $R^{\mx}_{\{1\},3}$.
\begin{align*} \phi^{\mx}_{\!\{1\},3}(x_{1},x_{2},x_{3}) \ \coloneqq   \exists h_{2},h_{3}   \big(  \phi^{\mx}_{[2],3}(x_{1},h_{2},x_{3})  \wedge (h_{2}>x_{2})   \wedge \phi^{\mx}_{[2],3}(x_{1},h_{3},x_{2}) \wedge (h_{3}>x_{3})\big)
\end{align*}
\end{claim} 
\begin{proof}
``$\Rightarrow$'':
Suppose that $\phi^{\mx}_{\!\{1\},3}({\boldf{t}})$ is true for some ${\boldf{t}}\in \mathbb{Q}^{3}$ . If ${\boldf{t}}\of{1}\leq {\boldf{t}}\of{2}$ and ${\boldf{t}}\of{1}\leq {\boldf{t}}\of{3}$, then $h_{2}={\boldf{t}}\of{1}$ and $h_{3}={\boldf{t}}\of{1}$ which contradicts $h_{2}> {\boldf{t}}\of{2}$ and $h_{3}>{\boldf{t}}\of{3}$. Thus ${\boldf{t}}\in R^{\mx}_{\!\{1\},3}$.

``$\Leftarrow$'':  Suppose that ${\boldf{t}}\in R^{\mx}_{\!\{1\},3}$ for some ${\boldf{t}}\in \mathbb{Q}^{3}$. Without loss of generality, ${\boldf{t}}\of{1}>{\boldf{t}}\of{2}$.  
Then $\phi^{\mx}_{\!\{1\},3}({\boldf{t}})$ being true is witnessed by $h_{2}\coloneqq {\boldf{t}}\of{1}$ and any $h_{3}\in \mathbb{Q}$ that satisfies $h_{3} > \max(\boldf{t}\of{2},\boldf{t}\of{3})$.
\end{proof}

Since we already have a pp-definition $\phi^{\mx}_{\!\{1\},3}$ for $R^{\mx}_{\!\{1\},3}$, we can 
obtain a pp-definition $\phi^{\mx}_{\!\{1\},n+1}$  of $R^{\mx}_{\!\{1\},n+1}$ inductively as in \red{the proof of} Lemma~\ref{lemma:RelBaseMin}. 
The challenging part is showing the pp-definability of $R^{\mx}_{[k],k+1}$. \blue{Note that we have already covered the cases where $k\in [3]$.}% in the beginning of the proof. 
\begin{claim}
For $k\geq 4$, the relation $R^{\mx}_{[k],k+1}$ can be pp-defined by
\begin{align*} \phi^{\mx}_{[k],k+1}(x_{1},\dots,x_{k},y) \coloneqq     \exists h_2,\dots,h_{k-2}   \big ( & \phi^{\mx}_{[3],4}(x_{1},x_{2},h_{2},y) \wedge \phi^{\mx}_{[3],4}(h_{k-2},x_{k-1},x_{k},y)
\\  
&  \wedge   \bigwedge_{i=3}^{k-2} \phi^{\mx}_{[3],4}(h_{i-1},x_{i},h_{i},y)      \big ).
\end{align*}   
\end{claim}

\begin{proof}
Suppose that $  \boldf{t} \in {\mathbb Q}^{k+1}$ satisfies
$\phi_{[k],k+1}^{\mx}$. 
%We  
Let $h_{2},\dots,h_{k-2}\in \mathbb{Q}$ 
be witnesses of the fact that $\boldf{t}$ satisfies 
$\phi_{[k],k+1}^{\mx}$. 
If $\boldf{t}\of{y} < m \coloneqq \min(\boldf{t}\of{x_1},\dots,\boldf{t}\of{x_k})$ then 
$\boldf{t} \in R^{\mx}_{[k],k+1}$ and we are done,
so suppose that $\boldf{t}\of{y}\geq m$. 
Define $h_2',\dots,h_{k-2}' \in \{0,1\}$ 
by $h'_i \coloneqq 1$ if $h_i = m$ and $h'_i \coloneqq 0$ otherwise. 
Note that if $m < \min(\boldf{t}\of{x_1},\boldf{t}\of{x_2},h_2)$
then $\chi(\boldf{t})\of{x_1} = \chi(\boldf{t})\of{x_2} = h_2' = 0$. Otherwise, if $m = \min(\boldf{t}\of{x_1},\boldf{t}\of{x_2},h_2)$, then exactly two out
of 
$\chi(\boldf{t})\of{x_1},\chi(\boldf{t})\of{x_2},h_2'$ take the value $1$. 
The same holds for each conjunct of
$\phi^{\mx}_{[k],k+1}$, 
so they imply
\begin{align*}
\chi(\boldf{t})\of{x_1}+\chi(\boldf{t})\of{x_2}+h'_2 & = 0 \mod 2,  \\
h'_{k-2} + \chi(\boldf{t})\of{x_{k-1}} + \chi(\boldf{t})\of{x_k} & = 0 \mod 2,  && \text{and} \\
h'_{i-1} + \chi(\boldf{t})\of{x_i} + h'_i & = 0 \mod 2 &&
\text{for every } i \in \{3,\dots,k-2\}. 
\end{align*} 
Summing all these equations we deduce that 
$\sum_{i=1}^k \chi(\boldf{t})\of{x_i} = 0 \bmod 2$ 
and hence $\boldf{t} \in R^{\mx}_{[k],k+1}$.

Conversely, suppose that $\boldf{t} \in R^{\mx}_{[k],k+1}$. We have to show that $\boldf{t}$ satisfies $\phi_{[k],k+1}^{\mx}(x_1,\dots,x_k,y)$. If $\boldf{t}\of{y} < m \coloneqq \min(\boldf{t}\of{x_1},\dots,\boldf{t}\of{x_k})$ then we set all of 
$h_2,\dots,h_{k-2}$ to $m$ and all conjuncts of 
\red{$\phi_{[k],k+1}^{\mx}$} are satisfied. We may therefore suppose in the following that $\boldf{t}\of{y} \geq m$. 
Then it must be the case that 	$\sum_{i=1}^k \chi(\boldf{t})\of{x_i} = 0 \bmod 2$.
\red{Arbitrarily choose $s>\max(\boldf{t}\of{x_1},\dots,\boldf{t}\of{x_{k}})$.}
\blue{Without loss of generality, we may assume that $\boldf{t}\of{x_k}\leq \cdots \leq \boldf{t}\of{x_1}$; otherwise, we simply rename the variables to achieve the desired order.} 
Define 
\begin{align*}
h_2 & \coloneqq \begin{cases} s & \text{ if } \boldf{t}\of{x_1} = \boldf{t}\of{x_2}, \\
\min(\boldf{t}\of{x_1},\boldf{t}\of{x_2}) & \text{ otherwise},
\end{cases} 
\end{align*}
and, for $i \in \{3,\dots,k-2\}$, define
\begin{align*}
h_i & \coloneqq \begin{cases} s & \text{ if } h_{i-1} = \boldf{t}\of{x_i}, \\
\min(h_{i-1},\boldf{t}\of{x_i}) & \text{ otherwise.} 
\end{cases}
\end{align*}	
This clearly satisfies all conjuncts of 
$\phi_{[k],k+1}^{\mx}$ except for possibly the second. 
\blue{We show that our assignment also satisfies the second conjunct.
\red{Suppose, on the contrary, that the second conjunct is not satisfied.}
Since $\boldf{t}\of{x_k}\leq \cdots \leq \boldf{t}\of{x_1}$, by the definition of our assignment, we have $\boldf{t}\of{x_{k-1}}\leq h_{k-2}$.
Since $\sum_{i=1}^k \chi(\boldf{t})\of{x_i} = 0 \bmod 2$, we also have $m=\boldf{t}\of{x_k}=\boldf{t}\of{x_{k-1}}$.
By our assumption that the second conjunct is not satisfied, it follows that $m=\boldf{t}\of{x_k}=\boldf{t}\of{x_{k-1}}=  h_{k-2} $.
Moreover,
\begin{enumerate}
%			\item \label{item:counter_X1}  $m=\boldf{t}\of{x_k}=\boldf{t}\of{x_{k-1}}\leq h_{k-2} $,  
\item \label{item:counter_X2} for every $i\in \{3,\dots, k-2\}$, either  
%$h_{i-1}> \boldf{t}\of{x_{i}} =h_{i} \geq m$ 
\red{
$m \leq h_{i} =  \boldf{t}\of{x_{i}} < h_{i-1}$} 
or $m\leq h_{i-1}=\boldf{t}\of{x_{i}} <h_{i}$,  
\item \label{item:counter_X3} both options in the above item alternate for successive indices within $\{3,\dots, k-2\}$.
\end{enumerate}
%
%
%Then, by item~\eqref{item:counter_X1}, we have $h_{k-2}=m$. Unnoetig, haben wir doch gerade eben schon gesagt. 
%
Clearly, by  item~\eqref{item:counter_X2}, if $h_i= m$ holds for some $i\in \{3,\dots, k-2\}$, then $\boldf{t}\of{x_i}=m$.
We claim that this is also true whenever $h_i\neq m$, i.e., that $\boldf{t}\of{x_i}=m$ for every $i\in \{3,\dots, k\}$.
The claim can be proved by a simple induction on $i$.
Suppose that $h_i\neq m$ for some $i\in \{3,\dots, k-2\}$ such that  $\boldf{t}\of{x_{i'}}=m$ holds for every $i' \in \{i+1,\dots,k\}$.
Then it follows from item~\eqref{item:counter_X2} and item~\eqref{item:counter_X3} together with $h_{k-2}=m$ and the induction hypothesis that $k-i$ is odd.
Also, by the induction hypothesis, $\sum_{j=i+1}^k \chi(\boldf{t})\of{x_j}$ is odd.
Since $\sum_{j=1}^k \chi(\boldf{t})\of{x_j}$ is even and $\boldf{t}\of{x_k}\leq \cdots \leq \boldf{t}\of{x_1}$, it must be the case that $\boldf{t}\of{x_i}=m$.
This finishes the proof of the claim. 
Since $h_{k-2}=m$, it follows from item~\eqref{item:counter_X2}, item~\eqref{item:counter_X3}, and our claim that $h_2 = m$ if $k$ is even and $h_2 \neq m$ if $k$ is odd.
If $k$ is even, then $\sum_{j=3}^k \chi(\boldf{t})\of{x_j}$ is even.
Since our assignment satisfies the first conjunct  and $h_2 = m$, we must have either $\boldf{t}\of{x_{1}}=m$ or $\boldf{t}\of{x_{2}}=m$.
But then $\sum_{j=1}^k \chi(\boldf{t})\of{x_j}$ is odd, a contradiction to $\boldf{t} \in R^{\mx}_{[k],k+1}$.
If $k$ is odd, then $\sum_{j=3}^k \chi(\boldf{t})\of{x_j}$ is odd.
Since our assignment satisfies the first conjunct and $h_2 = m$, we must have either $\boldf{t}\of{x_{1}}=\boldf{t}\of{x_{2}}< h_2$ or $h_2\leq \min(\boldf{t}\of{x_{1}},\boldf{t}\of{x_{2}})$. 
But then $\sum_{j=1}^k \chi(\boldf{t})\of{x_j}$ is odd, a contradiction to $\boldf{t} \in R^{\mx}_{[k],k+1}$.
Hence, also the second conjunct is satisfied by our assignment.}
%
%
%Since $\sum_{i=1}^k \chi(\boldf{t})\of{x_i} = 0 \bmod 2$ and $\boldf{t}\of{1}\geq \cdots \geq \boldf{t}\of{k}$, we must have $\boldf{t}\of{x_{k-1}}=\boldf{t}\of{x_{k}}$.
%
%We also clearly have  $h_{k-2} \geq  \boldf{t}\of{x_{k-1}}$.
%
%			 
%		 
%		
%		
%		Note that 
%		$h_i$ is set to $m$ if and only if
%		$\sum_{j=1}^i \chi(\boldf{t})\of{x_j}$ is odd. 
%		In particular, 
%		$h_{k-2}$ is set to $m$ if and only if
%		$\sum_{j=1}^{k-2} \chi(\boldf{t})\of{x_j}$ is odd. 
%		If this is the case, then exactly one
%		of $\boldf{t}\of{x_{k-1}},\boldf{t}\of{x_k}$ equals $m$
%		because $\sum_{j=1}^{k} \chi(\boldf{t})\of{x_j}$ is even. 
%		If $h_{k-2}$ is set to $s$ then 
%		for the same reason we must have
%		$\boldf{t}\of{x_{k-1}} = \boldf{t}\of{x_k}$. 
%		In both cases the second conjunct of 
%		$\phi_{[k],k+1}^{\mx}$ is satisfied as well. 
\end{proof}

For the general case, let $k \coloneqq |I|$. Without loss of generality we may assume that $I = [k]$. 
\begin{claim} \label{claim:kn} 
The following pp-formula defines $R^{\mx}_{[k],n}$. 
\begin{align*}  \phi^{\mx}_{[k],n}(x_{1},\dots,x_{n})\coloneqq \exists h   \big( \phi^{\mx}_{[k],k+1}(x_{1},\dots,x_{k},h)   \wedge \phi^{\mx}_{\!\{1\},n+1}(h,x_{1},\dots,x_{n}) \big)
\end{align*}
\end{claim}

\begin{proof} ``$\Rightarrow$'': \red{Let ${\boldf{t}} \in R^{\mx}_{[k],n}$.}  
If  $\sum_{i=1}^{k}\chi({\boldf{t}})\of{i}=0 \bmod 2$, then   $\phi^{\mx}_{[k],k+1}({\boldf{t}}\of{1},\dots,{\boldf{t}}\of{k},h)$ holds for every $h\in \mathbb{Q}$ and   $\phi^{\mx}_{\!\{1\},n+1}(h,{\boldf{t}}\of{1},\dots,{\boldf{t}}\of{n})$ holds for every $h\in \mathbb{Q}$ with  $h>\Min({\boldf{t}}\of{k+1},\dots,{\boldf{t}}\of{n})$.  	
\red{Otherwise, $\Min({\boldf{t}}\of{1},\dots,{\boldf{t}}\of{k})>\Min({\boldf{t}}\of{k+1},\dots,{\boldf{t}}\of{n})$. 
Let $h\in \mathbb{Q}$ be such that $\Min({\boldf{t}}\of{k+1},\dots,{\boldf{t}}\of{n}) < h < \Min({\boldf{t}}\of{1},\dots,{\boldf{t}}\of{k})$. 
Then $h$ is a witness that shows that
${\boldf{t}} $ satisfies $\phi^{\mx}_{[k],n}$.}
%Then $\phi^{\mx}_{\!\{1\},n+1}(h,{\boldf{t}}\of{1},\dots,{\boldf{t}}\of{n})$ is true for every %\blue{Every such $h$ clearly also makes $\phi^{\mx}_{[k],k+1}(\boldf{t}\of{1},\dots,\boldf{t}\of{k},h)$ true.}

``$\Leftarrow$'': Let ${\boldf{t}} $ be an arbitrary $n$-tuple over $\mathbb{Q}$ not contained in $R^{\mx}_{[k],n}$. Then $\sum_{i=1}^{k}\chi({\boldf{t}})\of{i}\neq 0 \bmod 2$, and $\Min({\boldf{t}}\of{1},\dots,{\boldf{t}}\of{k})\leq \Min({\boldf{t}}\of{k+1},\dots,{\boldf{t}}\of{n})$. For every witness $h\in \mathbb{Q}$  such that $\phi^{\mx}_{[k],k+1}({\boldf{t}}\of{1},\dots,{\boldf{t}}\of{k},h)$ is true,  we have   $\Min({\boldf{t}}\of{1},\dots,{\boldf{t}}\of{k})>h$. But then no such $h$ can witness $\phi^{\mx}_{\!\{1\},n+1}(h,{\boldf{t}}\of{1},\dots,{\boldf{t}}\of{n})$ being true. 
Thus, ${\boldf{t}} $ does not satisfy $\phi^{\mx}_{[k],n}$.
\end{proof}

%Claim~\ref{claim:kn} therefore implies that for all $k \leq n$
%the relation 
%$R^{\mx}_{[k],n}$ is pp-definable in $({\mathbb Q};X)$,
%	which 
This completes the proof of Lemma~\ref{lemma:RelationalBaseMx}.
\end{proof} 
%
% The polynomial-time algorithms  from \cite{bodirsky2010complexity} for CSPs of temporal structures preserved by $\mx$ and for CSPs of temporal structures preserved by $\mi$ only differ in the way how one determines whether a variable is contained in a free set in the current projection. 
%
%Thus
The expressibility of $\CSP(\mathbb{Q};\mathrm{X})$ in $\FPR_2$ can be shown using the same approach as in the first part of Section~\ref{section_tcsps_in_fp} via Proposition~\ref{blockedSetProjection} if the suitable procedure from \cite{bodirsky2010complexity} for finding free sets can be implemented in $\FPR_2$.
This is possible by encoding systems of \blue{mod-2} equations in $\FPR_2$ similarly as in the case of symmetric reachability in directed graphs in the paragraph above Corollary~III.2. in \cite{dawar2009logics}.
As usual, a solution to a homogeneous system of \blue{mod-2} equations is called \emph{trivial} if all variables take value $0$, and \emph{non-trivial} otherwise. 

\begin{proposition} \label{LFPmxCorrectnessSoundness}  
$\CSP(\mathbb{Q};\mathrm{X})$ is expressible in $\FPR_{2}$.  
\end{proposition} 
\begin{proof}  
Recall that $\struct{B}\coloneqq (\mathbb{Q};\mathrm{X})$ is preserved by $\mx$ and hence also by $\pp$.
\blue{Since all proper projections of the relations of $\struct{B}$ are trivial, $\struct{B}$ satisfies the prerequisites of Corollary~\ref{fp_pp}.}
Our aim is to construct a formula $\phi(x)$ satisfying the requirements of Corollary~\ref{fp_pp} by rewriting the algorithm in Figure~\ref{algo:FreeSetsMX} in the syntax of $\FPR_2$. 
In the computation of the algorithm in Figure~\ref{algo:FreeSetsMX}, each 
constraint is of the form $\mathrm{X}(x,y,z)$, and hence contributes a single equation to $E$, namely  $x+y+z=0$.
The algorithm subsequently isolates those variables which denote the value $1$ in some non-trivial solution for $E$. 
Write $E$ as $M\boldf{x}=\boldf{v}$. 
We define two numeric terms $f_{M}$ and $f_{\boldf{v}}$ which encode the matrix and the vector, respectively, of this system.
\begin{align*}
f_{M}(x_1,x_2,x_3,y_1,x)\coloneqq \ & \big(\mathrm{X}(x_1,x_2,x_3) \wedge U(x_1)\wedge U(x_2) \wedge U(x_3)  \wedge (y_1=x_1 \vee y_1=x_2 \vee y_1=x_3)\big) \\  & \vee (x_1=x_2=x_3=y_1=x)  
\\[0.5\baselineskip]
f_{\boldf{v}}(x_1,x_2,x_3,y_1,x)\coloneqq \ &  (x_1=x_2=x_3=y_1=x)  
\end{align*}
Let $\struct{A}$ be an instance of $\CSP(\struct{B})$ and $U\subseteq A$ arbitrary.
For every $x\in U$, the matrix 
$ \Mat^\struct{A}_2\llbracket  f_{M}(\cdot,\cdot,x) \rrbracket \in  \{0,1\}^{A^{3}\times A}$
contains
\begin{enumerate}
\item 	for each constraint $\mathrm{X}(x_1,x_2,x_3)$ of $\struct{A}$, where $x_1,x_2,x_3\in U$, three 1s in the  $(x_1,x_2,x_3)$-th row: namely, in the $x_1$-th, $x_2$-th, and $x_3$-th column, and
\item  \label{vecctr} a single 1 in the $(x,x,x)$-th row: namely, in the $x$-th column.
\end{enumerate}
We can test the solvability of  $M\boldf{x} = \boldf{v}$ in $\FPR_2$ by comparing the rank of $M$ with the rank of $(M|\boldf{v})$:  the system is satisfiable if and only if they have the same rank.
The case that $\struct{A}$ contains a constraint of the form $\mathrm{X}(y,y,y)$ is treated specially; in this case, $\struct{A}$ does not have a solution (note that our encoding of $M\boldf{x} = \boldf{v}$ is incorrect whenever $\struct{A}$ contains such a constraint).  
The formula $\phi(x)$ can be defined as follows.
\begin{align*}
\phi(x)   \coloneqq \ &  \exists y \big(\mathrm{X}(y,y,y)\vee \neg  \big([\rk_{(x_1,x_2,x_3),y_1} 	f_{M}(x_1,x_2,x_3,y_1,x) \bmod 2 ]  = \\ & [\rk_{(x_1,x_2,x_3),(y_1,y_2)}\ (y_2\neq y)\cdot 	f_{M}(x_1,x_2,x_3,y_1,x)+(y_2=y)  \cdot f_{\boldf{b}}(x_1,x_2,x_3,y_1,x) \bmod 2]      \big)   \big)
\end{align*} 
Now the statement of the proposition  follows from Corollary~\ref{fp_pp}.
\end{proof}  

\subsection{\blue{A proof of inexpressibility in FPC}}
Interestingly, the inexpressibility of $\CSP(\mathbb{Q};\mathrm{X})$ in $\FPC$ cannot be shown by giving a pp-construction of systems of \blue{mod-2} equations and utilizing the inexpressibility result of Atserias, Bulatov, and Dawar \cite{atserias2009affine} (see Corollary~\ref{nointerpretations}).
For this reason we resort to the strategy of showing that $\CSP(\mathbb{Q};\mathrm{X})$ has unbounded counting width and then applying Theorem~\ref{separation} \cite{dawar2017definability}.
\blue{In Proposition~\ref{prop:Xequiv}, we show that  $\CSP(\mathbb{Q};\mathrm{X})$ can be reformulated as a particular decision problem for systems of equations over $\mathbb{Z}_{2}$, where each constraint $\mathrm{X}(x,y,z)$ is viewed as the mod-2 equation $x+y+z=0$.} 

\begin{proposition}\label{prop:Xequiv}
The problem \textup{3-Ord-Xor-Sat} (defined in the introduction) and $\CSP(\mathbb{Q};\mathrm{X})$ are the same computational problem. 
\end{proposition}
\begin{proof} The structure $(\mathbb{Q};\mathrm{X})$ is preserved by $\mx$.
Thus, the algorithm in Figure~\ref{solvestar} jointly with the algorithm in Figure~\ref{algo:FreeSetsMX} for computing free sets is sound and complete for $\CSP(\mathbb{Q};\mathrm{X})$.  

\blue{Suppose that the equations in an instance $\struct{A}$ of $\CSP(\mathbb{Q};\mathrm{X})$ form a positive instance of \textup{3-Ord-Xor-Sat}.
We show that then every projection of $\struct{A}$ has a free set.
Note that all proper projections of the relation $\mathrm{X}$ are trivial.
Therefore, we can ignore projections of constraints in our \red{argument}.
Let $S\subseteq A$ be arbitrary, and let $E$ be the set of all equations $x+y+z=0$ for $x,y,z\in A\setminus S$ such that $\struct{A}$ has the constraint $\mathrm{X}(x,y,z)$.
By our assumption, $E$ has a solution over $\mathbb{Z}_2$ where at least one variable $x$ denotes the value $1$.
This means that, given  $\pr_{A\setminus S}(\struct{A})$ as an input, the algorithm in Figure~\ref{algo:FreeSetsMX} returns a non-empty set $F$ containing $x$. 
Since $F$ is the union of all free sets for $\pr_{A\setminus S}(\struct{A})$, we conclude that $\pr_{A\setminus S}(\struct{A})$ has a free set.
This means that, given $\struct{A}$ as an input, the algorithm in Figure~\ref{solvestar} jointly with the algorithm in Figure~\ref{algo:FreeSetsMX} finds a free set in every step and accepts $\struct{A}$.}
Since our algorithm is correct for 
$\CSP(\mathbb{Q};\mathrm{X})$, we conclude that 
$\struct{A} \rightarrow (\mathbb{Q};\mathrm{X})$.

Conversely, suppose that $\struct{A} \rightarrow (\mathbb{Q};\mathrm{X})$. Then the algorithm described above produces a sequence $A_1,\dots,A_\ell$ of subsets of $A$ where $A_1 \coloneqq A$ and, for every $i<\ell$ the set
$A_{i+1}$ is the subset of $A_{i}$
where we remove all elements which are contained in a free set of the substructure of $\struct{A}$ with domain $A_i$. Moreover, $\struct{A}$ contains no \blue{mod-2} equations 
on variables from $\struct{A}_{\ell}$. 
Let $E$ be a non-empty subset of the equations from $\struct{A}$
and let $B$ be the variables that appear in the equations from $E$. Let $i$ be maximal such that $B \subseteq A_i$. Then mapping all variables in $B \cap A_{i+1}$ to $0$
and all variables in $B \setminus A_{i+1}$ to $1$ is a non-trivial solution to $E$:
\begin{itemize} 
\item an even number of variables of each constraint 
is in $B \setminus A_{i+1}$, by the definition of free sets; 
\item $B$ cannot be fully contained in $B \setminus A_{i+1}$ because $E$ is non-empty; 
\item $B$ cannot be fully contained 
in $B \cap A_{i+1}$ by the maximal choice of $i$.  \qedhere
\end{itemize} 
\end{proof}

\blue{The satisfiability problem for systems of equations over a fixed finite Abelian group, where the \red{number} of variables per equation is bounded by a constant, can be formulated as a finite-domain CSP.
In the present article, we only need to encode equations of the form $x_1+\cdots+ x_j = a$. For this purpose, we can use the following mixture of definitions from~\cite{atserias2009affine} and \cite{atserias2019definable}.
\begin{definition} \label{def:lineq} Let $\mathscr{G}$ be a finite Abelian group and $k$ a natural number.
Then we define $\struct{E}_{\mathscr{G},k}$ as the relational structure over the domain $G$ of $\mathscr{G}$ with the relations 
$\{{\boldf{t}}\in G^{\,j} \mid \textstyle\sum_{i\in [j]}{\boldf{t}}\of{i}=a \}$
for every $j \in [k]$ and $a\in G$. 
Let $e$ be the neutral element in $\mathscr{G}$, and let $\struct{A}$ be an instance of $\CSP(\struct{E}_{\mathscr{G},k})$ for some $k$.
The \emph{homogeneous companion} of  $\struct{A}$  is obtained by moving the tuples from each $j$-ary relation of $\struct{A}$, $j\in [k]$, to the unique $j$-ary relation $R^{\struct{A}}$  such that $R^{\struct{E}_{\mathscr{G},k}} = \{{\boldf{t}}\in G^{\,j} \mid \textstyle\sum_{i\in [j]}{\boldf{t}}\of{i}=e \}$.
\end{definition}}
\red{Every system of equations over $\mathscr{G}$ of the form $x_1+\cdots + x_{j} = a$ with $j \in [k]$ and $a\in G$ gives rise to a
structure in the signature of $\struct{E}_{\mathscr{G},k}$ whose domain consists of the variables and whose relations are described by the equations. Clearly, 
the system 
is satisfiable if and only if this structure has a homomorphism to $\struct{E}_{\mathscr{G},k}$.} 
%
%According to the discussion above Definition~\ref{def:lineq}, we may view $\CSP(\mathbb{Q};\mathrm{X})$ as a proper subset of  $\CSP(\struct{E}_{\mathbb{Z}_{2},3})$. 
%
We use the probabilistic construction of multipedes from \cite{gurevich1996finite,blass2002polynomial} as a black box for extracting certain homogeneous systems of \blue{mod-2} equations that represent instances of $\CSP(\mathbb{Q};\mathrm{X})$ \blue{via Proposition~\ref{prop:Xequiv}.}
More specifically, we use the reduction from the proof of Theorem~23 in \cite{blass2002polynomial} of the isomorphism problem for multipedes to the satisfiability of a system of equations over $\mathbb{Z}_{2}$ with $3$ variables per equation.  
The following concepts were introduced in \cite{gurevich1996finite}; we mostly follow the terminology in \cite{blass2002polynomial}. 
\begin{definition} 
A \emph{multipede} is a finite relational structure $\struct{M}$ with the signature $\{<,E,H\}$, where $<,E$ are binary symbols and $H$ is a ternary relation symbol, such that $\struct{M}$ satisfies the following axioms. 
The domain  has a partition   into \emph{segments} $\mathrm{SG}(\struct{M})$  and \emph{feet} $\mathrm{FT}(\struct{M})$  such that ${<}^{\struct{M}}$ is a linear order on $\mathrm{SG}(\struct{M})$, and $E^{\struct{M}}$ is the graph of a surjective function $\mathrm{sg}\colon \mathrm{FT}(\struct{M}) \rightarrow \mathrm{SG}(\struct{M})$ with $|\mathrm{sg}^{-1}(x)|=2$ for every $x\in \mathrm{SG}(\struct{M})$.
\blue{For every $\boldf{t} \in H^{\struct{M}}$, either the entries of $\boldf{t}$ are contained in $\mathrm{SG}(\struct{M})$ and ${\boldf{t}}$ is called a \emph{hyperedge}, or they are contained in $\mathrm{FT}(\struct{M})$ and ${\boldf{t}}$ is called a \emph{positive triple}. 
We require that 
\begin{itemize}
\item 	$H^{\struct{M}}$ only contains triples with pairwise distinct entries and is closed under adding triples obtainable by permuting the entries of an already present triple;
\item For every positive triple $\boldf{t}$, the triple $\mathrm{sg}(\boldf{t})$ is a hyperedge (here $\mathrm{sg}$ acts component-wise);
\item If $\boldf{s}\in H^{\struct{M}}$ is a hyperedge, then exactly $4$ triples $\boldf{t}$ with $\mathrm{sg}(\boldf{t}) = \boldf{s}$ are positive triples;
\item For all positive triples $\boldf{t}_1,\boldf{t}_2 \in \mathrm{sg}^{-1}(\boldf{s})$, the number of entries where $\boldf{t}_1$ and $\boldf{t}_2$ differ is even.
\end{itemize}}
\end{definition}
%
% \blue{
% 	Note that the set of all positive triples $\boldf{t}$ with $\mathrm{sg}(\boldf{t})=\boldf{s}$ for a fixed hyperedge $\boldf{s}$ is preserved under transposition of the feet at an even number of entries in $\boldf{s}$; transposing the feet at an odd number of entries in $\boldf{s}$ yields the complementary set of the four triples which are not positive.}

\blue{A multipede $\struct{M}$ is \emph{odd} if for each $\emptyset \subsetneq X\subseteq \mathrm{SG}(\struct{M})$ there is a hyperedge ${\boldf{t}}\in H^{\struct{M}}$ such the number of entries in $\boldf{t}$ containing an element from $X$ is odd. 
A multipede $\struct{M}$ is $k$-\emph{meager} if for each $\emptyset \subsetneq X\subseteq \mathrm{SG}(\struct{M})$ of size  at most $2k$ we have $|X|>  |H^{\struct{M}}\cap X^{3}|/3.$}

\begin{remark}
\red{The relation $H^{\struct{M}}$ might as well be encoded using $3$-element sets.
And indeed, this was the case in \cite{gurevich1996finite,blass2002polynomial}. We have adapted the definition to our setting where relations may only contain ordered tuples.
For this reason, our definition of $k$-meagerness differs from the original by a factor of $6$ because we must take the multiple occurrences of each hyperedge into account.
These are the only deviations from the original definition. In particular, all results from \cite{gurevich1996finite,blass2002polynomial} concerning multipedes remain true after our modifications.}
\end{remark} 
The following four statements (Proposition~\ref{rigid}, Lemma~\ref{countingmultipedes}, Proposition~\ref{existmultipedes}, and Lemma~\ref{conversionmult}) are crucial for our application of multipedes in the context of $\CSP(\mathbb{Q};\mathrm{X})$.  %An automorphism is called \emph{trivial} if it is an identity map, and \emph{non-trivial} otherwise.
\begin{proposition}[\cite{blass2002polynomial}, Proposition~17] \label{rigid}
Let $\struct{M}$ be an odd multipede. Then $\Aut(\struct{M}) = \{\mathrm{id}\}$.
\end{proposition} 	 
\begin{lemma}[\cite{gurevich1996finite}, Lemma~4.5] \label{countingmultipedes} For any $k\in \mathbb{N}_{>0}$, let $\struct{M}$ be a $2k$-meager multipede. 
Let $\struct{M}_1$ and $\struct{M}_2$ be two expansions of $\struct{M}$ obtained by placing a constant on the two different feet of one particular segment, respectively.	
Then $\struct{M}_{1} \equiv_{\mathcal{C}^{k}}\struct{M}_{2}$. 
The statement even holds for expansions of $\struct{M}_1$ and $\struct{M}_2$  by constants for all segments. 
\end{lemma}   
\blue{The above lemma is stated in \cite{gurevich1996finite} using the $\mathcal{C}_{\infty\omega}^{k}$-equivalence instead.
However, it is well-known that for finite $\tau$-structures $\struct{A}$ and $\struct{B}$, we have $\struct{A} \equiv_{\mathcal{C}_{\infty\omega}^{k}} \struct{B}$ if and only if $\struct{A} \equiv_{\mathcal{C}^{k}} \struct{B}$ \cite{gradel1992inductive}.}
\begin{proposition}[\cite{blass2002polynomial}, Proposition~18] \label{existmultipedes} For   \red{any integer $k > 0$}, there exists an odd $k$-meager multipede.  
\end{proposition}    
\blue{Let $\struct{M}$ be a multipede and let $ A$ be  the incidence matrix of the hyperedges on the segments, i.e., the value of $A$ at the coordinate $(\boldf{t},s) \in  (H^{\struct{M}}\cap \mathrm{SG}(\struct{M})^{3})\times \mathrm{SG}(\struct{M})$ equals  $1$ if  $s$ is one of the entries in $\boldf{t}$ and $0$ otherwise.
Note that $A$ has exactly three non-zero entries per row. 
Let $\struct{A}$ be the system $A \boldf{x}  = \boldf{0}$ viewed as an instance of $\CSP(\struct{E}_{\mathbb{Z}_{2},3})$.
%
%Fix an arbitrary bijection $f_Y\colon M \rightarrow M$ that preserves ${<}^{\struct{M}}\cap Y^2$ and $E^{\struct{M}}\cap (\mathrm{sg}^{-1}(Y)\times Y)$.
%
%Let $\boldf{b}_Y\in \{0,1\}^{H^{\struct{M}}\cap Y^{3}}$ be the tuple such that $\boldf{b}'\of{\boldf{t}}=0$ if and only if $f_Y$ preserves positive triples of $\struct{M}$ at $\boldf{t}$. 
%
For all $X\subseteq Y \subseteq \mathrm{SG}(\struct{M})$, we define the maps $f_{X,Y}\colon Y\cup \mathrm{sg}^{-1}(Y) \rightarrow M$ and $\tilde{f}_{X,Y}\colon Y \rightarrow \{0,1\}$ as follows:
\[ f_{X,Y}(x) \coloneqq 
\begin{cases}
y & \text{if } \mathrm{sg}^{-1}(s)=\{x,y\} \text{ for some } s\in X, \\
x & \text{otherwise},
\end{cases}  
\qquad \text{and} \qquad    \tilde{f}_{X,Y}(x) \coloneqq \begin{cases}  1 & \text{if } s\in X, \\
0 & \text{otherwise}.
\end{cases}
\]   
%
%Clearly, every partial isomorphism from $\struct{M}$ to $\struct{M}$ w.r.t.\ the substructure on $Y\cup Z$ is of the form $f_{X,Y}$ for some $f$ and 
%
%
%Every automorphism of the substructure of $\struct{M}$ on $Y\cup Z$ is of the form $f_{X,Y}$ for some $X\subseteq \mathrm{SG}(\struct{M})$ (see the second-but-last paragraph in the proof of Theorem~23 of~\cite{blass2002polynomial}).
%
%For every $X\subseteq Y$, let $\boldf{x}_{X}\in \{0,1\}^{Y}$ be the tuple defined by $\boldf{x}_{X}\of{s}\coloneqq 1$ if and only if $s \in X$. The following lemma is a simple consequence of the definition of a multipede.
% 
\begin{lemma}[cf.\ \cite{blass2002polynomial}, the proof of Theorem 23] \label{conversionmult} 
For every $X\subseteq Y \subseteq \mathrm{SG}(\struct{M})$, the following are \red{equivalent}:
\begin{enumerate}
\item \label{item:systems_1} $f_{X,Y}$ is a partial isomorphism from $\struct{M}$ to $\struct{M}$;
\item \label{item:systems_2} $\tilde{f}_{X,Y}$ is a partial homomorphism from $\struct{A}$ to $\struct{E}_{\mathbb{Z}_{2},3}$.
\end{enumerate} 
\end{lemma} 
\begin{proof} Clearly,  $f_{X,Y}$  \red{preserves} ${<}^{\struct{M}}\cap Y^2$ and also the set of all hyperedges whose entries are in $Y$.
Hence, \eqref{item:systems_1} holds iff $f_{X,Y}$ preserves the set of all positive triples whose entries are in $\mathrm{sg}^{-1}(Y)$.
Note that, for a hyperedge $\boldf{s}$, $f_{X,Y}$ preserves the set of all positive triples whose entries are in $\mathrm{sg}^{-1}(\boldf{s}\of{1},\boldf{s}\of{2},\boldf{s}\of{3})$  iff  the number of entries of $\boldf{s}$ contained in $X$ is even, i.e., iff $\tilde{f}_{X,Y}(\boldf{s}\of{1})+\tilde{f}_{X,Y}(\boldf{s}\of{2})+ \tilde{f}_{X,Y}(\boldf{s}\of{3})=0 \bmod 2$.
This is true for every hyperedge with entries in $Y$ if and only if \eqref{item:systems_2} holds. 
\end{proof}}

\begin{example}
We now describe the multipede $\struct{M}$ from Figure~\ref{multipede} in detail.
We have that $\mathrm{SG}(\struct{M})= \mathbb{Z}_{9}$, $\mathrm{FT}(\struct{M})=\mathbb{Z}_{9}\times \mathbb{Z}_{2}$,   ${<}^{\struct{M}}$ is the linear order $0<\cdots < 8$, and 
$E^{\struct{M}} = \{({\boldf{t}},s)\in  (\mathbb{Z}_{9}\times \mathbb{Z}_{2})\times \mathbb{Z}_{9} \mid {\boldf{t}}\of{1}=s\}$.
Moreover, we have the following set of hyperedges: 
\blue{$$  H^{\struct{M}}\cap \mathrm{SG}(\struct{M})^3 = \{ \boldf{s} \in  \mathbb{Z}_{9}^3 \mid \text{there are } i,j,k \in [3] \text{ such that }  \boldf{s}\of{i}=\boldf{s}\of{j}+2  \text{ and }
\boldf{s}\of{j}={\boldf{s}}\of{k}+3\bmod 9  \},$$ }
and the following set of positive triples:
\begin{align*} H^{\struct{M}}\cap \mathrm{FT}(\struct{M})^3  = \{ ({\boldf{t}}_1,{\boldf{t}}_2,{\boldf{t}}_3) \in (\mathbb{Z}_{9}\times \mathbb{Z}_{2})^{3} \mid & \ ({\boldf{t}}_1\of{1},{\boldf{t}}_2\of{1},{\boldf{t}}_3\of{1}) \in H^{\struct{M}}\cap \mathrm{SG}(\struct{M})^3  \\  
& \text{ and }  {\boldf{t}}_1\of{2}+	{\boldf{t}}_2\of{2}+	{\boldf{t}}_3\of{2} = 1 \bmod 2 \}.
\end{align*}  

Note that the hyperedges do not overlap on more than one segment, because the minimal distances between two entries of an hyperedge are $2,3$, or $4 \bmod  9$. This directly implies that both multipedes are $2$-meager. 
Using Gaussian elimination, one can check that the  system of \teal{mod-$2$} equations $A\boldf{x}=\boldf{0}$, where $A$ is the incidence matrix of the hyperedges on the segments, only admits the trivial solution. We claim that from this fact it follows that $\struct{M}$ is odd.  Otherwise, suppose that there exists a non-empty subset $X$ of the hyperedges witnessing that this is not the case. Then $A\boldf{x}=\boldf{0}$ is satisfied by the non-trivial assignment that maps $\boldf{x}\of{s}$ to $1$ if and only if $s\in X$, which yields a contradiction.
Thus, by Proposition~\ref{rigid} and Lemma~\ref{conversionmult}, the expansions $\struct{M}_1$ and $\struct{M}_2$ of $\struct{M}$ obtained by placing a constant on the two different feet of the segment $0$ are not isomorphic. 
\end{example}

Keeping the construction above Lemma~\ref{conversionmult} in mind, we can derive the following statement about systems of \blue{mod-2} equations. 
\begin{proposition} \label{frstclaim} For every $k\geq 3$, there exist instances $\struct{A}_1$ and $\struct{A}_2$ of $\CSP(\struct{E}_{\mathbb{Z}_{2},3})$ such that 
\begin{enumerate}
\item $\struct{A}_1$ and $\struct{A}_2$ have the same homogeneous companion which only has the trivial solution,
\item $\struct{A}_1$ has no solution and $\struct{A}_2$ has a solution,
\item $\struct{A}_1\equiv_{\mathcal{C}^{k}} \struct{A}_2$.
\end{enumerate}
\end{proposition}   
Our proof strategy for Proposition~\ref{frstclaim} is as follows. 
We first use multipedes to construct instances $\struct{A}'_1$ and $\struct{A}'_2$ of $\CSP(\struct{E}_{\mathbb{Z}_{2},3})$ that satisfy
item (1) and (2) of the statement. 
%
%More precisely, $\struct{A}'_1$ is obtained by adding a single inhomogeneous equation to the homogeneous variant of the system described below Proposition~\ref{existmultipedes}, and $\struct{A}'_2$ is its homogeneous companion.
%
Then we use the following construction of Atserias and Dawar~\cite{atserias2019definable} to transform them into instances $\struct{A}_1$ and $\struct{A}_2$ that additionally satisfy
item (3) of the statement. 
For an instance $\struct{A}$ of $\CSP(\struct{E}_{\mathbb{Z}_{2},3})$,  
let $G(\struct{A})$ be the system that contains for each equation 
$x_1+ \cdots + x_j= b$ of $\struct{A}$ and all $a_{1},\dots,a_{j}\in \{0,1\}$ the equation \[\blue{x_{1,a_{1}}+\dots+x_{j,a_j}=b+a_{1}+\cdots+a_{j}.}\]

%
%		In order for the last step to work, we have to
%		ensure that $\struct{A}'_1\Rightarrow_{\exists^{+}\!\!\mathcal{L}^{k}} \struct{E}_{\mathbb{Z}_{2},3}$.  
%We also use the following lemma, which is essentially from Atserias and Dawar~\cite{atserias2019definable}. 

\begin{lemma}[Atserias and Dawar~\cite{atserias2019definable}]\label{lem:AD}
Let $\struct{A}$ be an instance of $\CSP(\struct{E}_{\mathbb{Z}_{2},3})$ and $k \geq 3$ such that $\struct{A}\Rightarrow_{\exists^{+}\!\!\mathcal{L}^{k}} \struct{E}_{\mathbb{Z}_{2},3}$. Then 
$G(\struct{A}) \equiv_{\mathcal{C}^{k}} G(\struct{A}_0)$
where $\struct{A}_0$ is the homogeneous companion of $\struct{A}$.  
\end{lemma}
\begin{proof}
The statement is almost Lemma~2 in Atserias and Dawar~\cite{atserias2019definable} with the only difference that they additionally assume that all constraints in the instance %$\struct{A}$ 
are imposed on three distinct variables; however, their winning strategy for Duplicator also works in the more general setting. 
\end{proof}

\begin{proof}[Proof of~\ref{frstclaim}.]
For a given $k\geq 3$, let $\struct{M}$ be an odd $6k$-meager multipede whose existence follows from Proposition~\ref{existmultipedes}. 
\blue{Let $A \boldf{x}=\boldf{0}$ be the system of \teal{mod-$2$} equations derived from  $\struct{M}$ using the construction described in the paragraph above Lemma~\ref{conversionmult}. 
It is easy to see that every automorphism of $\struct{M}$ is of the form $f_{X,Y}$ for $Y\coloneqq M$ and some $X\subseteq  Y$.
Since $\struct{M}$ is odd, by Proposition~\ref{rigid}, the only automorphism of $\struct{M}$ is the identity.
Therefore, by Lemma~\ref{conversionmult}, $A\boldf{x}=\boldf{0}$ only has the trivial solution.}
This means that the  inhomogeneous system obtained
from $A\boldf{x}=\boldf{0}$ by adding the equation $\boldf{x}\of{s}=1$, where $s$ is the first segment, has no solution.
%  
%On the other hand, 	since $A$ has one or three non-zero entries per row, the inhomogeneous system obtained from $A\boldf{x}=\boldf{0}$ by adding the equation $\boldf{x}\of{s}=1$ and replacing $\boldf{0}$ with $\boldf{1}$ is satisfiable by setting all variables to $1$.
%
We refer to this system by $\struct{A}'_1$ and to its homogeneous companion by $\struct{A}'_2$. 
\blue{We clearly have  $\struct{A}'_2 \Rightarrow_{\exists^{+}\!\!\mathcal{L}^{k}} \struct{E}_{\mathbb{Z}_{2},3}$ because Duplicator has the trivial winning strategy of placing all pebbles on $0$ in the existential $k$-pebble game played on $\struct{A}'_2$ and $\struct{E}_{\mathbb{Z}_{2},3}$.

We claim that also $\struct{A}'_1 \Rightarrow_{\exists^{+}\!\!\mathcal{L}^{k}} \struct{E}_{\mathbb{Z}_{2},3}$. 
We may assume that $\struct{M}$ has its signature expanded by constant symbols for every segment (Lemma~\ref{countingmultipedes}). For convenience, we fix an arbitrary linear order on $M$ which coincides with ${<}^{\struct{M}}$ on $\mathrm{SG}(\struct{M})$, and say that $x$ is a \emph{left foot} and $y$ a \emph{right foot} of a segment $s$ with $\mathrm{sg}^{-1}(s)=\{x,y\}$ if $x$ is less than $y$ w.r.t.\ this order. 
Let $\struct{M}_1$ and $\struct{M}_2$ be the expansions of $\struct{M}$ by a constant for the left and the right foot of the first segment, respectively.
By Lemma~\ref{countingmultipedes}, we know that Duplicator has a winning strategy in the bijective $3k$-pebble game played on $\struct{M}_1$ and $\struct{M}_2$. We use it to construct a winning strategy for Duplicator in the existential $k$-pebble game played on $\struct{A}'_1 $ and $\struct{E}_{\mathbb{Z}_{2},3}$. 
Suppose that Spoiler chooses $i\in [k]$ and places the pebble $a_i$ on some  $s \in \mathrm{SG}(\struct{M})$.
Then we consider the situation in the bijective $3k$-pebble game played on $\struct{M}_1$ and $\struct{M}_2$ where Spoiler places, in three succeeding rounds, the pebble $a_i$ on the same segment $s$, the pebble $a_{i+k}$ on its left foot, and the pebble $a_{i+2k}$ on its right foot.
Since Duplicator has a winning strategy in this game, she can always react by selecting a bijection whose restriction to the pebbled elements is  a partial isomorphism.
Let $f$ be the last bijection selected by Duplicator during a winning play.
Since the signature contains constant symbols for every segment, it must be the case that $f(s)=s$. Consequently,  $f(\mathrm{sg}^{-1}(s))=\mathrm{sg}^{-1}(s)$.
Now, if $f$ is the identity on $\mathrm{sg}^{-1}(s)$, then Duplicator places \teal{$b_{i}$} on $0$ in the existential $k$-pebble game, otherwise on $1$. 
Note that, if $s$ is the first segment, then $f$ cannot be the identity on $\mathrm{sg}^{-1}(s)$ due to the presence of the additional constant.
Also note that the restriction of $f$ to $Y\cup \mathrm{sg}^{-1}(Y)$, where $Y$ is the set of pebbled segments, is of the form $f_{X,Y}$ for some $X\subseteq Y$.
Since $f_{X,Y}$ is a partial isomorphism and the function $\tilde{f}$ specified by the pebbles placed in the existential $k$-pebble game is of the form $\tilde{f}_{X,Y}$, by Lemma~\ref{conversionmult}, $\tilde{f}$ is a partial homomorphism.}

%		Now we use the following construction from \cite{atserias2019definable}. 
%  
%		Let $\struct{A}_1$ be the system that contains for each equation 
%
%		$x_1+ \cdots + x_j= b$ of $\struct{A}'_1$ and all $a_{1},\dots,a_{j}\in \{0,1\}$ the equation  $$x^{a_{1}}_{1}+\dots+x^{a_{j}}_{j}=b+a_{1}+\cdots+a_{j}.$$ 
%
%	Analogously we obtain $\struct{A}_2$ from $\struct{A}'_2$.

For $i \in \{1,2\}$, let $\struct{A}_i$ be $G(\struct{A}_i')$. 
Note that the homogeneous companion $\struct{A}$ of $\struct{A}_1$ and $\struct{A}_2$ is identical and contains a copy of $\struct{A}'_2$ with variables $x_{i,a}$ for both upper indices $a\in \{0,1\}$. Thus,  $\struct{A}$ only admits the trivial solution,  which proves item (1). 
Also note that $\struct{A}_2$ is satisfiable by setting every variable $x_{i,a}$ to $a$, and that the unsatisfiability of
$\struct{A}'_1$ implies 
the unsatisfiability of 
$\struct{A}_1$, 
because the variables of the form \teal{$x_{i,0}$} induce a copy of $\struct{A}'_1$ 
in $\struct{A}_1$.  This proves item (2).
%
%	Finally, it follows from the proof of Lemma~2 in  \cite{atserias2019definable} that $ \struct{A}_1 \equiv_{\mathcal{C}^{k}}   \struct{A}_2$,  which proves item (3).
It follows from Lemma~\ref{lem:AD} that 
$\struct{A}_1 \equiv_{\mathcal{C}^{k}}   \struct{A}_2$,  which proves item (3).
\end{proof}

\begin{theorem} \label{FPxorsat} $\CSP(\mathbb{Q};\mathrm{X})$ is inexpressible in $\FPC$.
\end{theorem}
\begin{proof} Our  strategy is to pp-define  in $(\mathbb{Q};\mathrm{X})$ a temporal structure $\struct{B}$ such that from each pair $\struct{A}_1$ and $\struct{A}_2$  as in Proposition~\ref{frstclaim} we can obtain instances $\struct{A}'_1$ and $\struct{A}'_2$ of $\CSP(\struct{B})$ with
\begin{enumerate}  	 
\item $\struct{A}'_1\equiv_{\mathcal{C}^{k}} \struct{A}'_2$,
\item $\struct{A}'_1 \centernot{\rightarrow} \struct{B}$, and $\struct{A}'_2 \rightarrow \struct{B}$.
\end{enumerate}
The signature of $\struct{B}$ is $\{R_2,\dots, R_{5}\}$, and we set $R^{\struct{B}}_{i}\coloneqq R^{\mx}_{[i],i}$ for $i\in \{2,\dots, 5\}$ (see Definition~\ref{ordxordefinable}).
By Lemma~\ref{lemma:RelationalBaseMx} together with Theorem~\ref{partialbasismx},
we have that $\struct{B}$ is pp-definable in $(\mathbb{Q};\mathrm{X})$.
We now uniformly construct $\struct{A}'_i$ from $\struct{A}_i$ for both $i\in \{1,2\}$.
The domain of $\struct{A}'_i$ is the domain of $\struct{A}_i$ extended by a new element $z$, and the relations of $\struct{A}'_i$ are given as follows: for every $(x_1, \dots, x_{j}) \in R^{\struct{A}_i}$, 
\blue{\begin{itemize}
\item if $R^{\struct{E}_{\mathbb{Z}_{2},3}}= \{{\boldf{t}}\in \{0,1\}^{\,j} \mid \textstyle\sum_{i\in [j]}{\boldf{t}}\of{i}=1   \}$, then $R^{\struct{A}'_{i}}_{j+1}$ contains the tuple $(x_1,\dots, x_{j},z)$, and
\item if $R^{\struct{E}_{\mathbb{Z}_{2},3}}= \{{\boldf{t}}\in \{0,1\}^{\,j} \mid \textstyle\sum_{i\in [j]}{\boldf{t}}\of{i}=0   \}$, then $R^{\struct{A}'_{i}}_{j+2}$ contains the tuple $(x_1,\dots, x_{j},z,z)$.
\end{itemize}}

We have \blue{$\struct{A}'_1 \equiv_{\mathcal{C}^{k}} \struct{A}'_2$ by taking the extension of the winning strategy for Duplicator in the bijective $k$-pebble game} played on $\struct{A}_1$ and $\struct{A}_2$ where the new variable $z$ of $\struct{A}'_1$ is always mapped to its counterpart in $\struct{A}'_2$.  This proves item~(1).

We already know from Proposition~\ref{prop:Xequiv} that $ \CSP(\mathbb{Q};R^{\mx}_{[3],3}) =\CSP(\mathbb{Q};\mathrm{X})$ can be reformulated as a certain decision problem for  \blue{mod-2} equations which we call \textup{\textup{3-Ord-Xor-Sat}}.
Note that double occurrences of variables, such as the occurrence of $z$ above, do matter for 
\textup{\textup{3-Ord-Xor-Sat}} in contrast to plain satisfiability of \blue{mod-2} equations.
Also $\CSP(\struct{B})$ has a reformulation as a decision problem for \blue{mod-2} equations  where each constraint $R_{j}(x_1,\dots, x_j)$ for $j\in \{2,\dots, 5\}$ is interpreted as the homogeneous \blue{mod-2} equation $x_1+\cdots +x_{j}=0$.
The reformulation is as follows and can be obtained as in the proof of Proposition~\ref{prop:Xequiv}:

\smallskip
\noindent  INPUT:  A finite homogeneous system of  \blue{mod-2} equations of length $\ell \in \{2,\dots, 5\}$.

\noindent QUESTION:  Does every non-empty subset $S$ of the equations have a solution where at least one variable  in an equation from $S$ denotes the value $1$?
\medskip

Note that every solution of $\struct{A}_2$ viewed as an instance of $\CSP(\struct{E}_{\mathbb{Z}_{2},3})$ extended by setting $z$ to $1$ restricts to a non-trivial solution to every subset $S$ of the equations of $\struct{A}'_2$ with respect to the variables that appear in $S$, because $z$ occurs in every equation of $S$.   
We claim that the system $\struct{A}'_1$ only admits the trivial solution. If $z$ assumes the value $0$ in a solution of $\struct{A}'_1$, then this case reduces to the homogeneous companion of $\struct{A}_1$ which has only the trivial solution. If $z$ assumes the value $1$ in a solution of $\struct{A}'_1$, then this case reduces to $\struct{A}_1$ which has no solution at all. 
This proves item~(2).

It now follows from Theorem~\ref{separation} that $\CSP(\struct{B})$ is inexpressible in FPC.
Since $\struct{B}$ has a pp-definition in $(\mathbb{Q};\mathrm{X})$, by Theorem~\ref{REDUCTION} and Theorem~\ref{REDUCTION}, also $\CSP(\mathbb{Q};\mathrm{X})$ is inexpressible in FPC.
\end{proof}

\section{\label{section_classification} Classification of TCSPs in FP}
In this section we classify CSPs of temporal structures with respect to expressibility in fixed-point logic. 
We start with the case of a temporal structure $\struct{B}$ that is not preserved by any operation mentioned in Theorem~\ref{TCSPdichot}. 
In general, it is not known whether the NP-completeness of $\CSP(\struct{B})$ is sufficient for obtaining inexpressibility in FP. 
What is sufficient is the fact that $\struct{B}$ pp-constructs  $(\{0,1\};1\textup{IN}3)$.
\begin{lemma}\label{FPNPcomplete} Let $\struct{B}$ be a relational structure that
pp-constructs $(\{0,1\};1\textup{IN}3)$. Then  $\CSP(\struct{B})$ is inexpressible in FPC. 
\end{lemma}  
\begin{proof} It is well-known that $(\{0,1\};1\textup{IN}3)$ pp-constructs all finite structures. By the transitivity of pp-constructability,  $\struct{B}$ pp-construct the structure $\struct{E}_{\mathbb{Z}_{2},3}$ whose CSP is inexpressible in FPC by Theorem 10 in \cite{atserias2009affine}.  Thus, $\CSP(\struct{B})$ is inexpressible in FPC by 
\red{Theorem~\ref{REDUCTION}}.
%Theorem~\ref{theorem:reduction_pp_constructible}. 
\end{proof}  
We show in Theorem~\ref{mixedmx} that the temporal structures preserved by $\mx$ for which we know that their CSP is expressible in FP by the results in Section~\ref{section_tcsps_in_fp} are precisely the ones unable to pp-define the relation $\mathrm{X}$ which we have studied in Section~\ref{section_inexpressibility}.   
\begin{theorem}
\label{mixedmx}
Let  $\struct{B}$ be a temporal structure preserved by $\mx$. Then either $\struct{B}$ admits a pp-definition of $\mathrm{X}$, or one of the following is true:
\begin{enumerate}
\item $\struct{B}$ is preserved by a constant operation,
\item  $\struct{B}$ is preserved by $\Min$.
\end{enumerate} 
\end{theorem}   
\begin{proof} \label{proof_mixedmx} 
If (1) or (2) holds, then   $\mathrm{X}$ cannot have a pp-definition in $\struct{B}$ by Proposition~\ref{InvAutPol},  because $\mathrm{X}$ is neither preserved by a constant operation nor by $\Min$. 

Suppose that neither (1) nor (2) holds for $\struct{B}$, that is, $\struct{B}$ contains a relation that is
not preserved by any constant operation, and a relation  that is not preserved by $\Min$.
Our goal is to show that $\mathrm{X}$ has a pp-definition in $\struct{B}$.
The proof strategy is as follows. 
We first show that temporal relations which are preserved by $\mx$ and not preserved by a constant operation admit a pp-definition of $<$. 	
Then we analyse the behaviour of projections of temporal relations which are preserved by $\mx$ and not preserved by $\Min$ and use the pp-definability of $<$ to pp-define $X$.

We need to introduce some additional notation.  
Recall the definition of $\MS(R)$ for a temporal relation $R$ (Definition~\ref{ordxordefinable}).
For every $I\subseteq [n]$, we fix an arbitrary homogeneous system $M_{I}(R)  {\boldf{x}}={\boldf{0}}$ of \teal{mod-$2$} equations with solution set $\MS(\pr_{I}(R))$, where the matrix $M_{I}(R)$ 
%has the unique 
is in 
\blue{\emph{reduced row echelon form  without zero rows}:
\begin{itemize}
\item each row contains a non-zero entry, %non-zero rows are above all rows that only contain zeros,
\item the leading coefficient of each row is always strictly to the right of the leading coefficient of the row above it,
\item every leading coefficient is the only non-zero entry in its column.
\end{itemize}}
We reorder the columns of $M_{I}(R)$  such that it takes the form  
\begin{align}  
\left(\begin{array}{cc}   {U}_{m}      & * 
%\\  0 & 0 
\end{array}\right)  
\label{eq:system} 
\end{align}
where ${U}_{m}$ is the $m\times m$ unit matrix for some $m \leq n$; we also write $m_I(R)$ for $m$.
Without loss of generality, we may also assume that $I$ consists of the first $|I|$ elements of $[n]$. 
Finally, we define 
\[\supp_{I,i}(R)\coloneqq \{j \in [|I|] \mid M_{I}(R)\of{i,j}=1  \}.\]  

\begin{claim} \label{cl2} Let $R$ be a non-empty temporal relation preserved by $\mx$. If $R$ is not preserved by a constant operation, then ${<}$ has a pp-definition in $(\mathbb{Q};R)$.
\end{claim}
\begin{proof}   Let $n$ be the arity of $R$ and let $m \coloneqq m_{[n]}(R)$. Since $R$ is not preserved by a constant operation, we have ${\boldf{1}}\notin \chi(R)$. This means that $|\supp_{[n],i}(R)|$ is odd for some $i \leq n$ which is fixed for the remainder of the proof. 
Let $R'$ be the contraction of $R$ given by the pp-definition \[ R(x_{1},\dots,x_{n}) \wedge  \bigwedge_{p,q \in [n]\setminus \{1,\dots,m\}} x_p=x_q. \]
Note that $R'$ is non-empty since $R$ contains a tuple ${\boldf{t}}$ which satisfies 
$\chi({\boldf{t}})\of{x_j}=1$ 
if and only if 
\begin{itemize}
\item $j>m$, or 
\item $j\leq m$ and $|\supp_{[n],j}(R)|$ is even. 
\end{itemize}
We claim that every ${\boldf{t}}\in R'$ is of this form. If $\chi({\boldf{t}})\of{x_j}=0$ for some $j>m$, then $\chi({\boldf{t}})\of{x_\ell}=0$ for every $\ell>m$ by the definition of $R'$, which implies that $\chi({\boldf{t}})\of{x_\ell}=0$ for every $\ell\leq m$ by a parity argument with the equations of $M_{[n]}(R) {\boldf{x}}={\boldf{0}}$. But then no entry can be minimal in ${\boldf{t}}$, a contradiction. Hence, $\chi({\boldf{t}})\of{x_\ell}=1$ for every $\ell>m$. 

For every $\ell\leq m$ we have $\chi({\boldf{t}})\of{\ell}=1$ if and only if $|\supp_{[n],\ell}(R)|$ is even. Since $R$ is non-empty, there exists an index  $k \in \{m+1,\dots,n\}$.
We have ${\boldf{t}}\of{k}<{\boldf{t}}\of{i}$ for every ${\boldf{t}}\in R'$ due to our previous argumentation. Hence, the relation $<$ coincides with $\pr_{\{k,i\}}(R')$ and therefore has a pp-definition in $(\mathbb{Q};R)$.  
\end{proof}

\begin{claim} \label{projectionsmin} 
Let $R$ be an $n$-ary temporal relation preserved by $\mx$
such that  for every $ I\subseteq [n]$, the set $\MS(\pr_{I}(R))$ is the solution set of a homogeneous system of \blue{mod-2} equations with at most two variables per equation. Then $R$ is preserved by $\Min$. 
\end{claim}

\begin{proof}
We show the claim by induction on $n$. 
For $n=0$, there is nothing to show. 
Suppose that the statement holds for all relations with arity less than $n$. 
For every $I\subseteq [n]$ we fix an arbitrary homogeneous system $M_{I}(R) {\boldf{x}}={\boldf{0}}$ of \blue{mod-2} equations with solution set $\MS(\pr_{I}(R)) $ that has at most two variables per equation. Note that $\MS(\pr_{I}(R)) $ is preserved by the \blue{mod-2} maximum operation $\Max$.
Moreover, for all ${\boldf{s}},{\boldf{s}}'\in \{0,1\}^{|I|}$ we have ${\boldf{0}}=\Max({\boldf{s}},{\boldf{s}}')$ if and only if ${\boldf{s}}={\boldf{s}}'={\boldf{0}}$, which means that $\chi(\pr_{I}(R))$ itself is preserved by $\Max$. 
Now for every pair ${\boldf{t}},{\boldf{t}}'\in R$ we want to show that $\Min({\boldf{t}},{\boldf{t}}')\in R$. If $\Min({\boldf{t}})=\Min({\boldf{t}}')$, then $ \chi(\Min({\boldf{t}},{\boldf{t}}'))=\Max(\chi({\boldf{t}}),\chi({\boldf{t}}'))\in \chi(R).$ 
If $\Min({\boldf{t}})\neq \Min({\boldf{t}}')$, then $\chi(\Min({\boldf{t}},{\boldf{t}}'))\in \{\chi({\boldf{t}}),\chi({\boldf{t}}')\}\subseteq \chi(R)$. 
Thus, there exists a tuple ${\boldf{c}}\in R$ with $\chi({\boldf{c}})=\chi(\Min({\boldf{t}},{\boldf{t}}'))$.  
We set $I\coloneqq \Minset({\boldf{c}})$. Since the statement holds for $ \pr_{n\setminus I}(R) $ by induction hypothesis and $\pr_{[n]\setminus I}(\Min({\boldf{t}},{\boldf{t}}'))=\Min(\pr_{[n]\setminus I}({\boldf{t}}),\pr_{[n]\setminus I}({\boldf{t}}'))  \in \pr_{[n]\setminus I}(R),$ there exists ${\boldf{r}}\in R$ with $\pr_{[n]\setminus I}(\Min({\boldf{t}},{\boldf{t}}'))=\pr_{[n]\setminus I}({\boldf{r}})$. 
We can apply an automorphism to ${\boldf{r}}$ to obtain a tuple ${\boldf{r}}'\in R$ where all entries are positive. We can also apply an automorphism to obtain a tuple ${\boldf{c}}'\in R$ so that its minimal entries $i\in I$ are equal $0$ and for every other entry $i\in [n]\setminus I$ it holds that ${\boldf{c}}'\of{i}>{\boldf{r}}'\of{i}$. Then $\mx({\boldf{c}}',{\boldf{r}}')$ yields a tuple in $R$ which in the same orbit as $\Min({\boldf{t}},{\boldf{t}}')$. Hence, $R$ is preserved by $\Min$.
\end{proof}	

%	In Claim~\ref{cl1} we show that every temporal structure which contains the strict linear order $<$ and a relation that satisfies the requirements of Claim~\ref{projectionsmin} already admits a pp-definition of $\mathrm{X}$.
%
\begin{claim} \label{cl1} Let $R$ be a temporal relation preserved by $\mx$. If $R$ is not preserved by $\Min$, then $\mathrm{X}$ has a pp-definition in $(\mathbb{Q}; <,R)$.
\end{claim}

\begin{proof} Let $n$ be the arity of $R$.
Since $R$ is not preserved by $\Min$, 
Claim~\ref{projectionsmin} implies that there exists  $I\subseteq [n]$ such that $\MS(\pr_{I}(R))$ is not the solution set of any homogeneous system of \blue{mod-2} equations with at most two variables per equation.
Recall that  $\MS(\pr_{I}(R))$ is the solution set of a system  $M_{I}(R) \boldf{x}=\boldf{0}$ of \teal{mod-2} equations where $M_{I}(R)$ is as in \eqref{eq:system}.
Let $m \coloneqq m_{I}(R)$ and fix an arbitrary index $i\in [m]$ with $|\supp_{I,i}(R)|\geq 3$.
We also fix an arbitrary pair of  distinct indices $k,\ell \in \supp_{I,i}(R)\setminus \{i\}$. 
Note that $k, \ell \in \{m+1,\dots,|I|\}$  by the shape of the matrix $M_{I}(R)$.
We claim that the formula
$\phi(x_{i},x_{k},x_{\ell})$ obtained from the formula 
\begin{align} 
R(x_1,\dots,x_n) 
\wedge    \bigwedge_{j \in I\setminus \{k,\ell,1,\dots,m\}} x_i < x_{j} 
\label{eq:Xdef} 
\end{align}
by existentially quantifying all variables except for $x_i,x_k,x_{\ell}$ 
is a pp-definition of $\mathrm{X}$. 

``$\Rightarrow$'':
Let $\boldf{t} \in \mathrm{X}$.  We have to prove that $\bar t$ satisfies $\phi(x_i,x_k,x_{\ell})$. 
First, suppose that $\boldf{t}\of{x_i}=\boldf{t}\of{x_k}<\boldf{t}\of{x_{\ell}}$. 
Note that $M_I(R) \bar x = \bar 0$ has a solution where 
\begin{itemize}
\item $x_k$ takes value $1$, 
\item all variables $x_j$ such that
the $j$-th equation contains $x_k$ are also set to $1$, and 
\item all other variables are set to $0$. 
\end{itemize}
The reason is that in this way in each equation that contains $x_k$ exactly two variables are set to $1$, and in each equation that does not contain $k$ no variable is set to $1$.
Hence, $R$ contains a tuple
$\bar s'$ such that $\chi(\pr_I(\bar s'))$ corresponds to this solution. Note that $\bar s'$ also  satisfies \eqref{eq:Xdef}, and that there exists $\alpha' \in \Aut(\mathbb{Q};<)$ such that $\boldf{t} = (\alpha' \boldf{s}'\of{x_i},\alpha' \boldf{s}'\of{x_k},\alpha' \boldf{s}'\of{x_\ell})$.

The \blue{second} case  $ \boldf{t}\of{x_\ell}< \boldf{t}\of{x_i}=\boldf{t}\of{x_k}  $  can be treated analogously to the first one, using a tuple $\bar s'' \in R$ such that 
for all $j \in I$
\[
\chi(\pr_I(\boldf{s}''))\of{x_j} = 
\begin{cases} 1 & \text{if } \ell=j, \\
1&  \text{if } \ell \in \supp_{I,j}(R),  \\
0 & \text{otherwise},
\end{cases} 
\]
and $\alpha'' \in \Aut(\mathbb{Q};<)$ such that $ \boldf{t}  = (\alpha'' \boldf{s}''\of{x_i},\alpha'' \boldf{s}''\of{x_k},\alpha'' \boldf{s}''\of{x_\ell})$.

Finally, suppose that  $\boldf{t}\of{x_k}=\boldf{t}\of{x_{\ell}}<\boldf{t}\of{x_i}$.
\blue{Let $\boldf{t}'$ and $\boldf{t}''$ be the two distinct tuples obtainable from $\boldf{t}$ through a non-trivial permutation of entries.
Note that $\boldf{t}'$ and $\boldf{t}''$ fall into the first and the second case, respectively, or vice versa.
Let $\boldf{s}',\boldf{s}''$ and $\alpha',\alpha''$ be any auxiliary tuples and automorphisms of $(\mathbb{Q};<)$ obtained in the previous two cases for $\boldf{t}'$ and $\boldf{t}''$.}
Then $\boldf{s} \coloneqq \mx(\alpha' \boldf{s}', \alpha'' \boldf{s}'') \in R$ is a tuple that satisfies the quantifier-free part of \eqref{eq:Xdef}, and there exists 
$\alpha \in \Aut(\mathbb{Q};<)$ such that $\boldf{t} = (\alpha \boldf{s}\of{x_i},\alpha \boldf{s}\of{x_k},\alpha \boldf{s}\of{x_\ell})$.

``$\Leftarrow$'':
Suppose $\boldf{s}\in R$ satisfies \eqref{eq:Xdef}. We must show 
$(\pr_I(\boldf{s})\of{x_i},\pr_I(\boldf{s})\of{x_k},\pr_I(\boldf{s})\of{x_{\ell}}) \in \mathrm{X}$. By the final conjuncts of \eqref{eq:Xdef} 
all indices of minimal entries in $\pr_I(\boldf{s})$
must be from
$x_k,x_{\ell},x_1,\dots,x_m$.
Let $j\in I$ be the index of a minimal entry in $\pr_I(\boldf{s})$.
First consider the case $j\in \{i,k,\ell\}$. 
The shape of $M_{I}(R)$ implies that the variables of 
the $i$-th equation of 
$M_{I}(R) \boldf{x}=\boldf{0}$ 
must come from $x_i,x_{m+1},\dots,m_{|I|}$.
As mentioned, none of these variables can denote a minimal entry in $\pr_I(\boldf{s})$ except for $x_i$, $x_k$, and $x_\ell$. Hence, 
the $i$-th equation implies that 
$\pr_I(\boldf{s})$ takes a minimal value at exactly two of the indices
$\{i,k,\ell\}$. 
So we conclude that 
$(\pr_I(\boldf{s})\of{x_i},\pr_I(\boldf{s})\of{x_k},\pr_I(\boldf{s})\of{x_{\ell}}) \in \mathrm{X}$. 

Otherwise, $j\in \{1,\dots, m\}\setminus \{i\}$.
The shape of $M_{I}(R)$ implies that 
the variables of 
the $j$-th equation 
of $M_{I}(R) \boldf{x}=\boldf{0}$ 
must come from $x_j,x_{m+1},\dots,m_{|I|}$. 
None of these variables can denote a minimal entry in $\pr_I(\boldf{s})$ except for $x_j$, $x_k$, and $x_\ell$. 
Hence, the $j$-th equation of $M_{I}(R) \boldf{x}=\boldf{0}$ implies that $\pr_I(\boldf{s})$ takes a minimal value at exactly two of the indices
$\{j,k,\ell\}$. We have thus reduced the situation to the first case. 
\end{proof}   
The statement of Theorem~\ref{mixedmx} follows from Claim~\ref{cl2} and Claim~\ref{cl1}.
\end{proof}

We are now ready for the proof of our characterisation of temporal CSPs in $\FP$ and $\FPC$.
\begin{proof}[Proof of Theorem~\ref{mainresult}]
Let $\struct{B}$ be a temporal structure.

``(1)$\Rightarrow$(2)'': Trivial because $\FP$ is a fragment of $\FPC$.

``(2)$\Rightarrow$(3)'':  Lemma~\ref{FPNPcomplete} implies that $\struct{B}$ does not pp-construct $(\{0,1\};1\textup{IN}3)$;
Theorem~\ref{FPxorsat} and 
\red{Theorem~\ref{REDUCTION}} %Theorem~\ref{theorem:reduction_pp_constructible} 
show that $\struct{B}$ does not pp-construct $(\mathbb{Q};\mathrm{X})$. 

``(3)$\Rightarrow$(4)'':  Since $\struct{B}$ does not pp-construct $(\{0,1\};1\textup{IN}3)$,  by Theorem~\ref{TCSPdichot}, $\struct{B}$ is preserved by $\Min$, $\mi$, $\mx$, $\elel$, the dual of one of these operations, or by a constant operation.
If $\struct{B}$ is  preserved by $\mx$ but neither by $\Min$ nor by a constant operation, then  $\struct{B}$ pp-defines $\mathrm{X}$ by Theorem~\ref{mixedmx}, a contradiction to (3). 
If $\struct{B}$ is  preserved by $\dual\, \mx$ but neither by $\Max$ nor by a constant operation, then 
$\struct{B}$ pp-defines $-\mathrm{X}$ by the dual version of Theorem~\ref{mixedmx}. Since $(\mathbb{Q};\mathrm{X})$ and $(\mathbb{Q};-\mathrm{X})$ are homomorphically equivalent, we get a contradiction to (3) in this case as well.
Thus (4) must hold for $\struct{B}$.

``(4)$\Rightarrow$(1)'':   	If $\struct{B}$ has a constant polymorphism, then $\CSP(\struct{B})$ is trivial and thus expressible in FP.
If $\struct{B}$ is preserved by $\Min$, $\mi$, or $\elel$, then every relation of  $\struct{B}$ is pp-definable in \mbox{$(\mathbb{Q};<,\mathrm{R}^{\leq}_{\Min})$} by Lemma~\ref{lemma:RelBaseMin}, or in
$(\mathbb{Q};\mathrm{R}_{\mi},\mathrm{S}_{\mi},{\neq})$ by Lemma~\ref{RelationalBaseMi}, or in
$(\mathbb{Q};\mathrm{R}_{\elel},\mathrm{S}_{\elel},\neq)$ by Lemma~\ref{RelationalBaseLl}.
Thus, $\text{CSP}(\struct{B})$ is expressible in FP by Proposition~\ref{LFPminCorrectnessSoundness}, Proposition~\ref{LFPmiCorrectnessSoundness}, or Proposition~\ref{LFPllCorrectnessSoundness} combined with \red{Theorem~\ref{REDUCTION}}. %Theorem~\ref{theorem:reduction_pp_constructible}.   
Each of the previous statements can be dualized to obtain expressibility of $\CSP(\struct{B})$ in FP if $\struct{B}$ is preserved by $\Max$, $\dual\,\mi, $ or $\dual \,\elel$.
\end{proof}

We finally prove our characterisation of the temporal CSPs in $\FPR_2$.  
\begin{proof}[Proof of Theorem~\ref{FPRclassification}]
If $\struct{B}$ pp-constructs all finite structures, then $\struct{B}$ pp-constructs in particular the structure $\struct{E}_{\mathbb{Z}_{3},3}$. 
It follows from work of Gr\"adel and Pakusa~\cite{gradel2019rank} that $\CSP(\struct{E}_{\mathbb{Z}_{3},3})$ is inexpressible in $\FPR_2$ (see the comments after Theorem 6.8 in~\cite{GradelGPP19}).
%Theorem~\ref{theorem:reduction_pp_constructible}
\red{Theorem~\ref{REDUCTION}}
then implies that $\CSP(\struct{B})$ is inexpressible in $\FPR_2$ as well by. 

For the backward direction suppose that $\struct{B}$ does not pp-construct all finite structures. 
Then $\struct{B}$ is preserved by one of the operations listed in Theorem~\ref{TCSPdichot}.
If $\struct{B}$ is preserved by $\Min$, $\mi$, $\elel$, the dual of one of these operations, or by a constant operation, then $\struct{B}$ is expressible in $\FP$ by Theorem~\ref{mainresult} and thus in $\FPR_2$.
If $\struct{B}$ has $\mx$ as a polymorphism, then every relation of  $\struct{B}$ is pp-definable in the structure $(\mathbb{Q};\mathrm{X})$ 
by Lemma~\ref{lemma:RelationalBaseMx}.  
Thus, $\text{CSP}(\struct{B})$ is expressible in $\FPR_2$ by Proposition~\ref{LFPmxCorrectnessSoundness} combined with 
\red{Theorem~\ref{REDUCTION}}. %Theorem~\ref{theorem:reduction_pp_constructible}.  
Dually, if $\struct{B}$ has the polymorphism dual $\mx$, then $\CSP(\struct{B})$ is expressible in $\FPR_2$ as well. 
\end{proof}   

\section{Classification of TCSPs in Datalog\label{section_datalog}}
In this section, we classify  temporal CSPs with respect to expressibility in Datalog.
In some of our syntactic arguments in this section it will be convenient to work with formulas over the structure $({\mathbb Q};\leq,\neq)$ instead of the structure $({\mathbb Q};<)$. 
A $\{\leq,\neq\}$-formula is called \emph{Ord-Horn if it is a conjunction of clauses of the form 
$ x_{1} \neq y_{1} \vee \cdots \vee x_{m}\neq y_{m} \vee x \leq y$} where the last disjunct is optional~\cite{bodirsky2021complexity}. 
Nebel and B\"urckert~\cite{nebel1995reasoning} showed that satisfiability of Ord-Horn formulas can be decided in polynomial time.
Their algorithm shows that if a all relations of a template $\struct{B}$ are definable by Ord-Horn formulas, then  $\CSP(\struct{B})$ can be solved by a $\Datalog$ program. 
\blue{In \cite{nebel1995reasoning}, Ord-Horn was introduced as the greatest tractable subclass of Allen's interval algebra containing all basic relations.
Among temporal structures, the Ord-Horn fragment is not even maximal w.r.t.\ tractability of the CSP as it is properly contained in the tractable class of temporal structures preserved by $\elel$~\cite{bodirsky2010fast}. 
However, it is the greatest element w.r.t.\ expressibility of the CSP in Datalog apart from temporal structures preserved by a constant operation, as we show in Theorem~\ref{datalogmainresult}.}
We first prove in Proposition~\ref{ordhorn} that Ord-Horn definability of temporal relations can be characterized in terms of admitting certain polymorphisms. 
\blue{The condition in Proposition~\ref{ordhorn} can be simplified to preservation by $\elel$ and $\dual\elel$, a characterisation we use in  Theorem~\ref{datalogmainresult}.}
Later we will prove that there is no characterisation of expressibility in Datalog in terms of identities for polymorphism clones (see Proposition~\ref{temporal-structureS_datalog}).  
\begin{proposition}
\label{ordhorn} A  temporal relation is definable by an Ord-Horn formula if and only if it is preserved by every binary injective operation on $\mathbb{Q}$ that preserves $\leq$.
\end{proposition}
Proposition~\ref{ordhorn} is proved using the syntactic normal form for temporal relations preserved by $\pp$ from \cite{bodirsky2014tractability} and the syntactic normal form for temporal relations preserved by $\elel$ from \cite{bodirsky2021complexity}. 
\begin{proposition}[\cite{bodirsky2014tractability}]  \label{ppsynt}	A temporal relation is preserved by $\pp$ if and only if it can be defined by a conjunction of formulas of the form $z_{1} \circ_{} z \vee \cdots \vee  z_{n} \circ_{n} z$ where $\circ_{i}\in \{\leq,\neq\}$ for each $i \in \{1,\dots,n\}$.  
\end{proposition}  
\begin{lemma} \label{weirdnormalform} Every temporal relation preserved by $\pp$ or $\elel$ can be defined by a conjunction of formulas of the form 
$x_{1} \neq y_{1}\vee \dots\vee x_{m}\neq y_{m} \vee z_{1} \leq z \vee \cdots \vee  z_{\ell} \leq z$.
% where $\circ_{i}\in \{\leq,\neq\}$ for each $i \in \{1,\dots,\ell\}$}.
\end{lemma}
\begin{proof} 
If $R$ is a temporal relation preserved by $\pp$ then the statement follows from Proposition~\ref{ppsynt}.
Let $R$ be a temporal relation preserved by $\elel$. 
Then $R$ is definable by a conjunction of clauses $\phi\coloneqq\bigwedge_{i}\phi_{i}$ where each clause  $\phi_{i}$ is as in Proposition~\ref{llSynt}. For an index $i$, let $\psi_{i}$ be  obtained from $\phi_{i}$ by dropping the inequality disjuncts of the form $x \neq y$. So $\psi_{i}$ is of the form
\begin{enumerate} 
\item $z_{1} < z \vee \dots \vee z_{\ell} < z$, or of the form
\item  $z_{1} < z\vee \dots \vee z_{\ell}<z  \vee   (z=z_{1}=\dots=z_{\ell})$. 
\end{enumerate}
If $\psi_{i}$ is of the form (1), then $\psi_i$ is a formula preserved by $\Min$ (Proposition~\ref{minsynt}). 
%$\psi_{i}$ clearly defines a relation preserved by $\Min$. 
%
If $\psi_{i}$ is of the form (2), then it is easy to see that $\psi_{i}$ is equivalent to 
\[  \bigwedge_{j \in [\ell]}  \left (z_{j} \leq z \vee  \bigvee_{k\in  [\ell]\setminus \{j\}} \blue{z_{k}} < z \right ).   \]
which is a formula preserved by $\Min$ as well (Proposition~\ref{minsynt}). 
In particular, in both cases $\psi_{i}$ is preserved by $\pp$. 
\blue{Since $\psi_i$ is preserved by $\pp$, it is equivalent to a conjunction $\psi'_i$ of clauses as in Proposition~\ref{ppsynt}.
We replace in each $\phi_{i}$ the disjunct $\psi_{i}$ by $\psi'_i$.} By use of distributivity of $\vee$ and $\wedge$, we can then rewrite $\phi_{i}$  	into a definition of $R$ that has the desired form. 
\end{proof}
A $\{\leq,\neq\}$-formula $\phi$ is said to be in \emph{conjuctive normal form (CNF)} if it is a conjunction of \emph{clauses}; a \emph{clause} if a disjunction of 
\emph{literals}, i.e., atomic $\{\leq,\neq\}$-formulas or negations of atomic $\{\leq,\neq\}$-formulas. 
We say that $\phi$ is \emph{reduced} if for any literal of $\phi$, the formula obtained by removing a literal from $\phi$ is 
not equivalent to $\phi$ over $({\mathbb Q};\leq,\neq)$. 
The next lemma is a straightforward but useful observation.
%
%\begin{lemma} 	\label{pruning} 
%	 Every $\{\leq,\neq\}$-formula is equivalent to a formula in reduced CNF. 
%		If $\phi$ is in reduced CNF with a clause $\gamma$ and a literal $\alpha$ of $\gamma$, then there exists a tuple which satisfies \blue{$\phi \wedge \alpha$} but does not satisfy all other literals of $\gamma$. 
%\end{lemma}
%
\begin{definition} \label{definition:lex_k} The $k$-ary operation \emph{lex}  on $\mathbb
{Q}$ is defined by  
\[\lex_{k}({\boldf{t}})\coloneqq  \lex\!\big({\boldf{t}}\of{1},\lex\!\big({\boldf{t}}\of{2},\dots\lex({\boldf{t}}\of{k-1},{\boldf{t}}\of{k})\dots\big)\big).\]
\end{definition} 
\begin{proof}[Proof of Proposition~\ref{ordhorn}]  
The forward direction is straightforward:  
every clause of an Ord-Horn formula is preserved by every injective operation on $\mathbb{Q}$ that preserves $\leq$.
For the backward direction,  let $R$ be a temporal relation preserved by every binary injective operation on $\mathbb{Q}$ that preserves $\leq$.   
In particular, $R$ is preserved by $\elel$ and by  $f\colon \mathbb{Q}^{2} \rightarrow \mathbb{Q}$ defined by $f(x,y)\coloneqq \lex_{3}(\Max(x,y),x,y)$. 
Let $\phi$ be a definition of $R$ provided by Lemma~\ref{weirdnormalform}.  
\blue{Note that if we remove literals from $\phi$, then the resulting formula is still of the same syntactic form, so we may assume that $\phi$ is a reduced CNF definition.} 
Let $\psi = (x_{1} \neq y_{1}\vee \dots\vee x_{m}\neq y_{m} \vee z_{1} \leq z \vee \cdots \vee  z_{\ell} \leq z)$  be a conjunct of $\phi$. 
We claim that $\ell \leq 1$. 
\blue{Otherwise, since $\phi$ is in reduced CNF,} there exist tuples ${\boldf{t}}_{1},{\boldf{t}}_{2}\in R$ 
such that ${\boldf{t}}_{1}$ does not satisfy
all disjuncts of $\psi$ except for $z_1 \leq z$ and ${\boldf{t}}_{2}$ does not satisfy
all disjuncts of $\psi$ except for $z_2 \leq z$. 
Without loss of generality, we may assume that ${\boldf{t}}_{1}\of{z}={\boldf{t}}_{2}\of{z}$, because otherwise we may replace $ {\boldf{t}}_{1}$ with $\alpha {\boldf{t}}_{1}$ for some $\alpha\in\Aut(\mathbb{Q};<)$ that maps ${\boldf{t}}_{1}\of{z}$ to ${\boldf{t}}_{2}\of{z}$.
Note that $f({\boldf{t}}_{1},{\boldf{t}}_{2})$
does not satisfy $z_i \leq z$ for every $i \in \{3,\dots,\ell\}$ because $f$ preserves $<$. 
\blue{Also note that 
%	${\boldf{t}}_{1}\of{z_1} \leq {\boldf{t}}_{1}\of{z}$, 
${\boldf{t}}_{1}\of{z}<{\boldf{t}}_{1}\of{z_2}$ and 
${\boldf{t}}_{2}\of{z}<{\boldf{t}}_{2}\of{z_1}$.}  
%	and ${\boldf{t}}_{2}\of{z_2} \leq {\boldf{t}}_{2}\of{z}$.
%
Since ${\boldf{t}}_{1}\of{z}={\boldf{t}}_{2}\of{z}$, 
we have $f({\boldf{t}}_{1},{\boldf{t}}_{2})\of{z} < f({\boldf{t}}_{1},{\boldf{t}}_{2})\of{z_1}$ 
and $f({\boldf{t}}_{1},{\boldf{t}}_{2})\of{z} < f({\boldf{t}}_{1},{\boldf{t}}_{2})\of{z_2}$ by the definition of $f$. 
But then $f({\boldf{t}}_{1},{\boldf{t}}_{2})$ does not satisfy $\psi$, a contradiction to $f$ being a polymorphism of $R$. Hence $\ell \leq 1$.
Since $\psi$ was chosen arbitrarily, we conclude that  $\phi$ is Ord-Horn.  
\end{proof}

Let $\mathrm{R}_{\Min}$ be the temporal relation defined by $y<x\vee z<x$ that was already mentioned in the introduction. 
Recall that $\CSP(\mathbb{Q};\mathrm{R}_{\Min})$ is inexpressible in $\Datalog$ \cite{bodirsky2010fast}. The reason for inexpressibility is not unbounded counting width, but the combination of the two facts that $\CSP(\mathbb{Q};\mathrm{R}_{\Min})$ admits unsatisfiable instances of arbitrarily high girth, and that  all proper projections of $\mathrm{R}_{\Min}$ are trivial. 
The counting width of $\CSP(\mathbb{Q};\mathrm{R}_{\Min})$ is bounded because $\textup{co-CSP}(\mathbb{Q};\mathrm{R}_{\Min})$ is definable using the FP sentence $ \exists x [ \dfp_{U,x} \exists y,z( U(y) \wedge U(z) \wedge  \mathrm{R}_{\Min}(x,y,z) )  ](x)$, see the paragraph below Theorem~\ref{separation}.
We show in Theorem~\ref{minordhorn} that the inability of a temporal structure with a polynomial-time tractable CSP to \red{pp-define} $\mathrm{R}_{\Min}$ can be characterised in terms of being preserved by a constant operation, or by the operations from Proposition~\ref{ordhorn} which witness Ord-Horn definability.  

\begin{theorem}
\label{minordhorn}
Let $\struct{B}$ be a temporal structure that admits a pp-definition of $<$. 
Then exactly one of the following two statements is true:
\begin{enumerate} 
\item \label{rmin1} $\struct{B}$ admits a pp-definition of the relation $\mathrm{R}_{\Min}$ or of the relation $-\mathrm{R}_{\Min}$,
\item \label{rmin2} $\struct{B}$ is preserved by every binary injective operation on $\mathbb{Q}$ that preserves $\leq$.    
\end{enumerate} 
\end{theorem} 
\begin{proof}%[Proof of Theorem~\ref{minordhorn}] \label{proof_minordhorn}
``\eqref{rmin2}$\Rightarrow$\eqref{rmin1}'': If \eqref{rmin2} holds, then $\struct{B}$ is in particular preserved by 
$\elel$ and $\dual \elel$.
But then, by Proposition~\ref{InvAutPol}, neither $\mathrm{R}_{\Min}$ nor $-\mathrm{R}_{\Min}$ has a pp-definition in $\struct{B}$ because $\mathrm{R}_{\Min}$ is not preserved by $\dual \elel$ and $-\mathrm{R}_{\Min}$ is not preserved by  
$\elel$ \cite{bodirsky2010fast}. 

``\eqref{rmin1}$\Rightarrow$\eqref{rmin2}'': Suppose that \eqref{rmin2} does not hold.
The \emph{Betweenness} problem mentioned in the introduction can be formulated as the CSP of a temporal structure $(\mathbb{Q};\mathrm{Betw})$ where 
\[
\mathrm{Betw} \coloneqq \{ (x,y,z)  \in \mathbb{Q}^3 \mid x<y<z \blue{\text{ or }} z< y <x \}.
\]

\begin{claim} \label{betw} $(\mathbb{Q};\mathrm{Betw},<)$  pp-defines  both relations $\mathrm{R}_{\Min}$ and $-\mathrm{R}_{\Min}$.
\end{claim}
\begin{proof} We show that $\phi(x,y,z) \coloneqq \exists a,b\big(\mathrm{Betw}(a,x,b) \wedge y<a \wedge  z<b \big)$ is a pp-definition of $\mathrm{R}_{\Min}$. Then it will be clear that $-\mathrm{R}_{\Min}$
has the pp-definition $\exists a,b\big(\mathrm{Betw}(a,x,b) \wedge (a<y) \wedge  (b<z) \big)$. 

``$\Rightarrow$'':	Let $\boldf{t}\in \mathrm{R}_{\Min}$ be arbitrary. We may assume with loss of generality that $\boldf{t}\of{3}< \boldf{t}\of{1} $.
Then any $a,b \in {\mathbb Q}$ such that $ \boldf{t}\of{3} < a < \boldf{t}\of{1}$ and $\Max(\boldf{t}\of{1},\boldf{t}\of{2})< b$ witness that $\boldf{t}$ satisfies $\phi$. 

``$\Leftarrow$'': Let $\boldf{t} \in \mathbb{Q}^3 \setminus \mathrm{R}_{\Min}$ be arbitrary.  
Then $\boldf{t}\of{1}\leq \boldf{t}\of{2}$ and $\boldf{t}\of{1}\leq \boldf{t}\of{3}$. Suppose that there exist $a,b \in \mathbb{Q}$ witnessing that $\boldf{t}$  satisfies $\phi$. Then $\boldf{t}\of{1}\leq \boldf{t}\of{2} < a$ and $\boldf{t}\of{1}\leq \boldf{t}\of{3} < b$, a contradiction to $\mathrm{Betw}(a,\boldf{t}\of{1},b)$. Thus, $\boldf{t}$ does not satisfy $\phi$.
\end{proof}

Now suppose that $\struct{B}$ does not pp-define $\mathrm{Betw}$.
Then by Lemma~10 and Lemma~49 in~\cite{bodirsky2010complexity},   $\struct{B}$ is preserved by some operation $g\in \{\pp,\elel,\dual\pp,\dual\elel\}$. 
If $\struct{B}$ is preserved by $\pp$ or $\elel$ then we can apply  Lemma~\ref{weirdnormalform}. The case where $\struct{B}$ is preserved by $\dual\pp$ or $\dual\elel$ can be shown analogously using a dual version of Lemma~\ref{weirdnormalform}.
The proof strategy is as follows. 
We fix any relation $R$ of $\struct{B}$ which is not preserved by some binary injective operation $f$ on $\mathbb{Q}$ that preserves $\leq$.
Lemma~\ref{weirdnormalform} implies that
$R$ has a definition of a particular form.  
It turns out that the projection onto a particular set of three entries in $R$ behaves like $\mathrm{R}_{\Min}$ modulo imposing some additional constraints onto the remaining variables.
These additional constraints rely on the pp-definability of $<$. 

Let $R$ be a relation of arity $n$ with a \red{pp-definition} in $\struct{B}$ such that 
$R$ is not preserved by a binary operation $f$ on $\mathbb{Q}$ preserving $\leq$.
\blue{Let $\phi(u_1,\dots,u_n)$ be a definition of $R$ 
of the form as described in Lemma~\ref{weirdnormalform}; since the removal of literals from $\phi$ preserves the form in Lemma~\ref{weirdnormalform}, we may \red{additionally} assume that $\phi$ is in reduced CNF.}
Then $\phi$ must have a conjunct $\psi$
of the form 
\[x_{1} \neq y_{1}\vee \dots\vee x_{m}\neq y_{m} \vee z_{1} \leq z \vee \cdots \vee  z_{\ell} \leq z\] 
that is not preserved by $f$. 
Since $f$ is injective and preserves $\leq$, it preserves all Ord-Horn formulas.
Hence, $\ell \geq 2$. Since $\phi$ is reduced, there are tuples 
${\boldf{t}}_1$ and ${\boldf{t}}_2$ \blue{satisfying $\phi$} such that for $i \in \{1,2\}$ 
\begin{itemize}
\item $\boldf{t}_i$ satisfies the disjunct $z_i \leq z$ of $\psi$; 
\item ${\boldf{t}}_i$ does not satisfy all other disjuncts of $\psi$.
\end{itemize} 
Let $\psi_{\mathrm{R}_{\Min}}(z,v_1,v_2)$ be the formula obtained \blue{by} existentially quantifying all variables except for
$z$, $v_1$, and $v_2$ in the following formula 
\begin{align}  v_{1} < z_1 \wedge v_{2}< z_2   \wedge  
R(u_1,\dots,u_n) 
\wedge  \bigwedge_{j\in \{3,\dots,\ell\}} z< z_{j} 
\wedge  \bigwedge_{i \in [m]} x_{i} = y_{i}. \label{eq:rmindef}
\end{align}
We claim that $\psi_{\mathrm{R}_{\Min}}$
is a pp-definition of $R_{\Min}$ over $({\mathbb Q};<,R)$. For the forward direction let ${\boldf{t}}\in \mathrm{R}_{\Min}$.  First suppose that ${\boldf{t}}\of{v_1} < {\boldf{t}}\of{z}$.
\blue{Let $\alpha$ be any automorphism of $(\mathbb{Q};<)$ that sends 
\begin{itemize}
\item ${\boldf{t}}_{1}\of{z}$ to ${\boldf{t}}\of{z}$, and
\item ${\boldf{t}}_{1}\of{z_1}$ to some rational number $q$ with ${\boldf{t}}\of{v_1}< q \leq {\boldf{t}}\of{z}$. 
\end{itemize}
Then 
$\alpha({\boldf{t}}_{1})$ provides witnesses for the variables $u_1,\dots,u_n$}
in 
$(\ref{eq:rmindef})$
%\in \pr_{I}(R)$ 
showing that ${\boldf{t}}$ satisfies $\psi_{\mathrm{R}_{\Min}}$. The case where  
${\boldf{t}}\of{z}\leq {\boldf{t}}\of{v_1}$ and 
${\boldf{t}}\of{z}> {\boldf{t}}\of{v_2}$ can be treated analogously, using \blue{$\boldf{t}_2$ instead of $\boldf{t}_1$.}

For the backward direction, suppose that $\boldf{s} \in {\mathbb Q}^{n+2}$ satisfies 
$(\ref{eq:rmindef})$.  
In particular, ${\boldf{s}}\of{z} < {\boldf{s}}\of{z_j}$
for every $j\in \{3,\dots,\ell\}$ and ${\boldf{s}}\of{x_i} = {\boldf{s}}\of{y_i}$ for every $i \in \{1,\dots,m\}$, 
and hence  
${\boldf{s}}\of{z_1} \leq {\boldf{s}}\of{z}$ or ${\boldf{s}}\of{z_2} \leq {\boldf{s}}\of{z}$
because $\boldf{s}$ satisfies $\psi$. 
If ${\boldf{s}}\of{z_1} \leq {\boldf{s}}\of{z}$ then ${\boldf{s}}\of{v_1} < {\boldf{s}}\of{z_1} \leq {\boldf{s}}\of{z}$ and hence
$({\boldf{s}}\of{z},{\boldf{s}}\of{z_1}, {\boldf{s}}\of{z_2}) \in \mathrm{R}_{\Min}$. 
Similarly, if ${\boldf{s}}\of{z_2} \leq {\boldf{s}}\of{z}$ then ${\boldf{s}}\of{v_2} < {\boldf{s}}\of{z_2} \leq {\boldf{s}}\of{z}$ 
and again $({\boldf{s}}\of{z},{\boldf{s}}\of{z_1}, {\boldf{s}}\of{z_2}) \in \mathrm{R}_{\Min}$. 
\end{proof}
We are ready for the proof of our second classification result; it combines 
\red{Theorem~\ref{REDUCTION}}, 
%Theorem~\ref{theorem:reduction_pp_constructible}, 
Proposition~\ref{ordhorn}, Theorem~\ref{minordhorn}, and results from previous sections. 
\begin{proof}[Proof of Theorem~\ref{datalogmainresult}]
Let $\struct{B}$ be a temporal structure. 

``(1)$\Rightarrow$(2)'':  If $\CSP(\struct{B})$ is expressible in $\Datalog$, then $\struct{B}$ does not pp-construct $(\{0,1\};1\textup{IN}3)$; otherwise we get a contradiction to the expressibility of $\CSP(\struct{B})$ in $\Datalog$ by Theorem~\ref{mainresult}, because $\Datalog$ is a fragment of $\FP$. Moreover, $\struct{B}$  does not pp-construct $(\mathbb{Q},\mathrm{R}_{\Min})$; otherwise we get a contradiction to the inexpressibility of $\CSP(\mathbb{Q},\mathrm{R}_{\Min})$ in $\Datalog$  (Theorem~5.2 in \cite{bodirsky2010fast}) through 
\red{Theorem~\ref{REDUCTION}}. %Theorem~\ref{theorem:reduction_pp_constructible}.

``(2)$\Rightarrow$(3)'':  Since $\struct{B}$ does not pp-construct $(\{0,1\};1\textup{IN}3)$, Theorem~\ref{TCSPdichot} implies that $\struct{B}$ is preserved by $\Min$, $\mi$, $\mx$, $\elel$, the dual of one of these operations, or by a constant operation. In the case where $\struct{B}$ is preserved by a constant operation we are done, so suppose that $\struct{B}$  is not preserved by a constant operation.
% but also does not admit a pp-definition of $<$.
%
First consider the case that $<$ is 
pp-definable in $\struct{B}$. Since $\mathrm{R}_{\Min}$ and $-\mathrm{R}_{\Min}$ are not 
pp-definable in $\struct{B}$, Theorem~\ref{minordhorn} shows that $\struct{B}$ is preserved by every binary injective operation on $\mathbb{Q}$ preserving $\leq$. 
In particular, $\struct{B}$ is preserved by $\elel$ and $\dual\elel$. 

Now consider the case that $<$ 
is not pp-definable. Since none of the temporal relations $\mathrm{Cycl}$, $\mathrm{Betw}$, $\mathrm{Sep}$ listed in \blue{Theorem~12.3.1 in \cite{bodirsky2021complexity}} is preserved by any of the operations $\Min,\mi,\mx, \elel$, or their duals, the theorem implies that $\Aut(\struct{B})$ contains all permutations of $\mathbb{Q}$. 
This means that $\struct{B}$ is an \emph{equality constraint language} as defined in \cite{bodirsky2008complexity}.
The structure $\struct{B}$ has a polymorphism which depends on two arguments but it does not have a constant polymorphism. Therefore, $\struct{B}$ has a binary injective polymorphism, by Theorem~4 in \cite{bodirsky2008complexity}. 
Since $\struct{B}$ is an equality constraint language with a binary injective polymorphism, by Lemma~2 in \cite{bodirsky2008complexity}, $\struct{B}$ is preserved by every binary injection on $\mathbb{Q}$.
In particular, $\struct{B}$ is preserved by $\elel$ and $\dual\elel$. 

``(3)$\Rightarrow$(1)'': If $\struct{B}$ has a constant polymorphism, then $\CSP(\struct{B})$ is trivial and thus expressible in $\Datalog$.
Otherwise, $\struct{B}$ is preserved by both $\elel$ and $\dual\elel$.  Then also the expansion $\struct{B}'$ of $\struct{B}$ by $<$ 
is preserved by both $\elel$ and $\dual\elel$. The latter implies that $\struct{B}'$  cannot pp-define $\mathrm{R}_{\Min}$ \blue{because $R_{\Min}$ is not preserved by $\dual\elel$ \cite{bodirsky2010fast}}. 
Thus, Theorem~\ref{minordhorn} implies that $\struct{B}'$ is preserved by every binary injective operation on $\mathbb{Q}$ that preserves $\leq$.
Then Proposition~\ref{ordhorn} then shows that all relations of $\struct{B}'$ and in particular of $\struct{B}$ are Ord-Horn definable. 
Therefore, $\CSP(\struct{B})$ is expressible in $\Datalog$ by Theorem~22 in \cite{nebel1995reasoning}.
\end{proof}  

\section{Algebraic conditions for temporal CSPs\label{section_algebraic_conditions}}
In this section, we consider several candidates for general algebraic criteria for expressibility of CSPs in FP and Datalog stemming from the well-developed theory of finite-domain CSPs.  
\blue{These criteria have already been displayed in Theorem~\ref{finitedomainsituation}.}%; several other conditions that are equivalent over finite structures will be discussed in this section.

Our results imply that none of them can be used to characterise expressibility of temporal CSPs in FP or in Datalog.
However, we also present a new simple algebraic condition which characterises 
expressibility of both finite-domain and
temporal CSPs in FP,  proving Theorem~\ref{alternativefp}. 
We assume basic knowledge of universal algebra; see, e.g., the textbook of Burris and Sankappanavar~\cite{BS}. 

\begin{definition} An \emph{identity} is a formal expression $s(x_1, \dots , x_n) \approx t(y_1, \dots, y_m)$ where $s$ and $t$ are  terms built from function symbols and the variables $x_1,\dots,x_n$ and $y_1,\dots,y_m$, respectively. An \emph{(equational)} \emph{condition} is a set of identities. Let $\mathscr{A}$ be a set of operations on a fixed set $A$. 
For a set $F\subseteq A$, a condition $\mathcal{E}$ is \emph{satisfied in $\mathscr{A}$ on $F$} if the function symbols of $\mathcal{E}$ can be assigned functions in  $\mathscr{A}$ in such a way that all identities of $\mathcal{E}$ become true for all possible values of their variables in $F$. If $F=A$, then we simply  say that \emph{$\mathcal{E}$ is satisfied in $\mathscr{A}$}.
\end{definition}

\subsection{\blue{Failures of known equational conditions}}
If we add to the assumptions of Theorem~\ref{finitedomainsituation} that all polymorphisms $f$ of $\struct{B}$ are \emph{idempotent}, i.e., satisfy $f(x,\dots,x) \approx x$, then the list of equivalent items can be prolonged further. 
In this setting, 
a prominent condition is
%The property 
that the variety of $\Pol(\struct{B})$ is \emph{congruence meet-semidistributive}, short SD($\wedge$), %. Congruence meet-semidistributivity 
which can also be studied over infinite domains. 
By Theorem~1.7 in \cite{olvsak2021maltsev},
in general SD$(\wedge)$ is equivalent to the existence of so-called $(3+n)$-polymorphisms for some $n$; these are idempotent operations $f,g_1,g_2$ where \blue{$g_1$ is $m$-ary,  $g_2$ is $n$-ary}, and 
$f$ is $(m  + n)$-ary, that satisfy 
\begin{align*}
f(x,\dots,x,\underset{i}{y},x,\dots,x) & \approx g_1(x,\dots,x,\underset{i}{y},x,\dots,x) && \text{for every } i \leq m, \\
f(x,\dots,x,\!\! \underset{n+i}{y} \!  ,x,\dots,x) & \approx g_2(x,\dots,x,\underset{i}{y},x,\dots,x) && \text{for every } i \leq n.  
\end{align*}
Proposition~\ref{olsaktermsordhorn} below implies that 
the correspondence between SD$(\wedge)$ 
and expressibility in Datalog / FP / FPC fails for temporal CSPs. 
A set of identities $\mathcal{E}$ is called 
\begin{itemize}
\item \emph{idempotent} if, for each operation symbol $f$
appearing in the condition, $f(x,\dots, x) \approx x$
is a consequence of $\mathcal{E}$, and \item \emph{trivial} if $\mathcal{E}$ can be satisfied by projections  over a set $A$ with $|A|\geq 2$, and \emph{non-trivial} otherwise. 
\end{itemize}
\blue{An $n$-ary operation $f$ 
\begin{itemize}
\item 	\emph{depends on the $i$-th argument} if there exist $a_1,\dots, a_n,a$ with $a_i \neq a$ and \[f(a_1,\dots, a_n) \neq f(a_1,\dots,a_{i-1} ,a,a_{i+1},\dots, a_n) ;\]
\item  is \emph{injective in the $i$-th argument} if the above inequality holds for all 
$a_1,\dots, a_n,a$ with $a_i \neq a$.
\end{itemize}	 
Let $I_f$ be the set of all indices $i\in[n]$ such that $f$ depends on the $i$-th argument.
Then $I_{f}= \{i_{1},\dots, i_{m}\}$ for some $i_1<\cdots < i_m$. 
We define the \emph{essential part} of an operation $f$ as the map  $(x_{1},\dots,x_{m}) \mapsto f(x_{\mu(1)},\dots,x_{\mu(n)})$ where $\mu \colon [n] \rightarrow [m]$ is any map that satisfies $\mu(i_{j})=j$ for each $j \in [m]$.
\red{This} is well-defined because $f$ does not depend on any argument from $[n]\setminus I_{f}$.
Note that an operation is a projection if and only if its essential part is the identity map.}
\begin{proposition}	\label{olsaktermsordhorn}  The polymorphism clone of $(\mathbb{Q};\neq,\mathrm{S}_{\elel})$ does not satisfy any non-trivial idempotent condition 
%which means that its variety is not $\mathrm{SD}(\wedge)$, 
(but $\CSP(\mathbb{Q},\neq,S_{\elel})$ is expressible in $\FP$).
\end{proposition} 
\begin{proof} 
\red{We first prove that for every idempotent $f
\in \Pol(\mathbb{Q};\mathrm{S}_{\elel},\neq)$
the operation $f^{\textit{ess}}$ is unary. 
Clearly, $f^{\textit{ess}}$ is idempotent and at least unary because $\neq$ has no constant polymorphism.
Note that the relation $I_4 \coloneqq \{\boldf{t}\in \mathbb{Q}^{4} \mid \boldf{t}\of{1} = \boldf{t}\of{2} \Rightarrow \boldf{t}\of{3} = \boldf{t}\of{4}\}$ has the pp-definition $S_{\elel}(x_1,x_2,x_3,x_4) \wedge S_{\elel}(x_1,x_2,x_4,x_3)$ in  $(\mathbb{Q};\neq,S_{\elel})$.
It is easy to see that if a polymorphism of $I_4$ depends on the $i$-th argument, then it is already injective in the $i$-th argument; see, e.g., the proof of  Proposition 6.1.4 in \cite{bodirsky2021complexity}.
Thus, $f^{\textit{ess}}$ is injective.
We claim that $f^{\textit{ess}}$ is unary.
Suppose for contradiction that $f^{\textit{ess}}$ has more than one argument. 
Let $c\coloneqq f^{\textit{ess}}(0,\dots, 0,1)$. Note that $f^{\textit{ess}}(c,\dots,c)=c$ by the idempotence of $f^{\textit{ess}}$, and hence $f^{\textit{ess}}(0,\dots, 0,1) = f^{\textit{ess}}(c,\dots,c)$,
contradicting the injectivity of $f^{\textit{ess}}$.
Thus, $f^{\textit{ess}}$ must be unary. 
Since  $f^{\textit{ess}}$ is idempotent and unary, it is the identity map.
But then each \blue{idempotent $f \in \Pol(\mathbb{Q};\mathrm{S}_{\elel},\neq)$} 
is a projection.
Hence, $\Pol(\mathbb{Q};\mathrm{S}_{\elel},\neq)$ does not satisfy a non-trivial condition $\mathcal{E}$ witnessed by some idempotent operations.}
\end{proof} 

Simply dropping idempotence in the definition of $(3+n)$-terms does not provide a characterisation of FP either, as  Proposition~\ref{olsakterms} shows.
In the proof of Proposition~\ref{olsakterms} we need the following result. 

\blue{\begin{lemma}[Lemma~4.4 in \cite{barto2019equations}, see also Lemma~3 in \cite{pinsker2021canonical}]\label{lem:construct}
Let $\struct{B}$ be an $\omega$-categorical structure and $f_1,g_1,\dots,f_n,g_n \in \Pol(\struct{B})$ where $f_i$ and $g_i$ have the same arity. If for every $i \in \{1,\dots,n\}$ and every finite $F \subseteq B$ there exists $\alpha_i \in \Aut(\struct{B})$ such that  $\alpha_i \circ f_i(\boldf{t}) = g_i(\boldf{t})$ for all $\boldf{t}$ over $F$, then there are $e,e_1,\dots,e_n \in \End(\struct{B})$ witnessing that 
$\Pol(\struct{B})$ satisfies
\[e_i\circ f_i(x_1,\dots,x_{k_i}) \approx e \circ g_i(x_1,\dots,x_{k_i}).\]
Moreover, if $\alpha_i $ and $\alpha_j$ can always be chosen to be equal for some $i,j\in [n]$, then additionally  $e_i = e_j$.
\end{lemma}}

%\textemdash   Proposition~\ref{olsakterms} shows that, when working with TCSPs, we also lose the correspondence of expressibility of  CSPs in FP to satisfiability of an arbitrary equational condition that is unsatisfiable by affine combinations over any field \cite{olvsak2019generalizing}.
%   
\begin{proposition}	\label{olsakterms}     
The polymorphism clone of $ (\mathbb{Q};\mathrm{X})$ contains not necessarily idempotent $(3+3)$-operations (but  $\CSP({\mathbb Q};\mathrm{X})$ is not expressible in FP). 
\end{proposition} 
\begin{proof} \label{proof_olsakterms} 
Consider the operations
\begin{align*} 
\tilde{g}_{2}(x_{1},x_{2},x_{3})=\tilde{g}_{1}(x_{1},x_{2},x_{3})\coloneqq \ & \mx(\mx(x_{1},x_{2}),\mx(x_{2},x_{3}))  \\
\tilde{f}(x_{1},x_{2},x_{3},x_{4},x_{5},x_{6})\coloneqq \ & \mx\big(\tilde{g}_{1}(x_{1},x_{2},x_{3}),\tilde{g}_{1}(x_{4},x_{5},x_{6})\big). 
\end{align*} 
\blue{Let $S$ be a finite subset of $\mathbb{Q}$, and let $\struct{B}_1$ and $\struct{B}_2$ be the substructures of $(\mathbb{Q};<)$ on the sets 
\begin{align*} 
& \{ \tilde{g}_{1}(y,x,x) , \tilde{g}_{1}(x,y,x), \tilde{g}_{1}(x,x,y)  \mid x,y\in S  \} \\
\text{and}\quad & \{\tilde{f}(y,x,x,x,x,x) , \tilde{f}(x,y,x,x,x,x), \tilde{f}(x,x,y,x,x,x)  \mid x,y\in S  \},
\end{align*}
respectively.  
Let $\alpha$ and $\beta$ be the operations from the definition of $\mx$.  
Consider the map
\[
h(b)\coloneqq \left\{ \begin{array}{ll} \beta(b) & \text{if } b=\tilde{g}_{1}(x,x,x) \text{ for some }x \in \mathbb{Q}, \\ 
\alpha(b) & \text{otherwise.}
\end{array}
\right.
\]
We claim that $h$ is an isomorphism from $\struct{B}_1$ to $\struct{B}_2$.
This is easy to show using Lemma~\ref{claim:comparisson} once we have made the following observation. 
For every $x\in \mathbb{Q}$, we have $\tilde{g}_{1}(x,x,x) = \beta^2(x)$ and $\tilde{f} (x,x,x,x,x,x)= \beta \circ \tilde{g}_{1}(x,x,x) =\beta^3(x)$.
Moreover, for all distinct $x,y\in \mathbb{Q}$, we have 
\[ \begin{array}{l}    \tilde{g}_{1}(y,x,x) = \alpha^2(\min(x,y)  ) \\
\tilde{g}_{1}(x,y,x) = \beta\circ\alpha(\min(x,y)  )\\
\tilde{g}_{1}(x,x,y) = \alpha^2(\min(x,y)  ) 
\end{array}  \quad\text{and}\quad  \begin{array}{l}   \tilde{f}  (y,x,x,x, x,x) =   \alpha (\tilde{g}_{1}(y,x,x) )  =  \alpha^3(\min(x,y)  ) \\
\tilde{f} (x, y,x,x, x,x) =     \alpha (\tilde{g}_{1}(x,y,x) )  = \alpha\circ \beta\circ\alpha (\min(x,y)  ) \\
\tilde{f} (x, x,y,x, x,x) =  \alpha (\tilde{g}_{1}(x,x,y)  ) = \alpha^3 (\min(x,y)  ).
\end{array}   \]  
By Lemma~\ref{claim:comparisson}, if $x<y$, then $\alpha(x)<\beta(y)$ which means that 
\[\hspace{-0.6em} \begin{array}{l} \tilde{g}_{1}(y,x,x)=  \mx(\alpha(x),\beta(x))=\alpha^{2}(x),  \\
\tilde{g}_{1}(x,y,x)=  \mx(\alpha(x),\alpha(x))=\beta(\alpha(x)) \\ 
\tilde{g}_{1}(x,x,y)= \mx(\beta(x),\alpha(x))=\alpha^{2}(x),   
\end{array}   \]
If  \red{$y < x$}, then $\alpha(y)<\beta(x)$ which means that 
\[ \hspace{-0.6em}\begin{array}{ll} \tilde{g}_{1}(y,x,x)=  \mx(\alpha(y),\beta(x))=\alpha^{2}(y),  \\
\tilde{g}_{1}(x,y,x)=  \mx(\alpha(y),\alpha(y))=\beta(\alpha(y)),  \\ 
\tilde{g}_{1}(x,x,y)= \mx(\beta(x),\alpha(y))=\alpha^{2}(y).   
\end{array}   \]
The statement about $\tilde{f}$ follows easily from its definition.
We prove that $h$ \red{preserves $<$}.
%; then it will be clear that it is well-defined as a map and injective.

Let $b_1,b_2$ be arbitrary elements of $B_1$.
Clearly, if $h(b_1)=\alpha(b_1)$ and $h(b_2)=\alpha(b_2)$, or $h(b_1)=\beta(b_1)$ and $h(b_2)=\beta(b_2)$, then $b_1<b_2$ implies $h(b_1) < h(b_2)$ because $\alpha$ and $\beta$ \red{preserve $<$}.
If $h(b_1)=\alpha(b_1)$ and $h(b_2)=\beta(b_2)$ or $h(b_1)=\beta(b_1)$ and $h(b_2)=\alpha(b_2)$, then $b_1<b_2$ implies $h(b_1) < h(b_2)$ by Lemma~\ref{claim:comparisson}.
%
%We claim that, if  $h(b_1)=\alpha(b_1)$ and $h(b_2)=\beta(b_2)$, then $b_1\leq b_2$ if and only if $b_1<b_2$. Suppose, on the contrary, that $b_1=b_2$. By definition, there exist $x,y,x',y' \in \mathbb{Q}$ with $x\neq y$ and $x'=y'$ such that $b_1\in \{\alpha^2(\min(x,y)), \beta \circ \alpha (\min(x,y))\}$ and $b_2 = \beta^2(x')$.
%%
%If $b_1 = \alpha^2(\min(x,y))$, then $b_1=b_2$ yields a direct contradiction to the fact that $\alpha$ and $\beta$ have disjoint images.
%%
%If $b_1 = \beta \circ \alpha(\min(x,y))$, then $b_1=b_2$ if and only if $\alpha(\min(x,y))=\beta(x')$ because $\beta$ is injective.
%%
%But now $\alpha(\min(x,y))=\beta(x')$ again yields a contradiction to the fact that $\alpha$ and $\beta$ have disjoint images.
%
\red{Thus $h$ preserves $<$ and it follows that $h$ is an isomorphism.} 
Since $(\mathbb{Q};<)$ is homogeneous, there exists $\eta \in \Aut(\mathbb{Q};<)$ extending $h$, i.e.,
\[ \begin{array}{c}  \eta \circ \tilde{g}_{1}(y,x,x) = \tilde{f}  (y, x,x,x, x,x)     \\
\eta \circ \tilde{g}_{1}(x,y,x) = \tilde{f} (x, y,x,x, x,x) \\
\eta \circ \tilde{g}_{1}(x,x,y) = \tilde{f} (x, x,y,x, x,x)
\end{array}    \] 
holds for all $x,y\in S$. 
By symmetry of the operation $\mx$, we also have that 
\[ \begin{array}{c}  \eta \circ \tilde{g}_{2}(y,x,x) = \tilde{f}  (x, x,x,y, x,x)     \\
\eta \circ \tilde{g}_{2}(x,y,x) = \tilde{f} (x, x,x,x, y,x) \\
\eta \circ \tilde{g}_{2}(x,x,y) = \tilde{f} (x, x,x,x, x,y)
\end{array}    \] 
holds for all $x,y\in S$. } 
Then Lemma~\ref{lem:construct} yields functions $f$ and $g_1=g_2$ which are $(3+3)$-operations in $\Pol(\mathbb{Q};\mathrm{X})$.   \end{proof}
The requirement of the existence of \blue{non-idempotent} $(m+n)$-operations is an example of a so-called  
\emph{minor condition}, which is a set of identities of the special form
$$ f_{1}(x^{1}_{1},\dots,x^{1}_{n_{1}})  \approx \cdots \approx  f_{k}(x^{1}_{k},\dots,x^{k}_{n_{k}}).$$ 
(such identities are sometimes also called \emph{height-one identities}~\cite{barto2018wonderland}\footnote{\emph{Linear identities} are defined similarly, but also allow that the terms in the identities consist of a single variable, which is more general. A finite minor condition therefore is a special case of what has been called a \emph{strong linear Maltsev condition} in the universal algebra literature).}). 
Another example of minor conditions can be found in
item (5) and (6) of Theorem~\ref{finitedomainsituation}:
an at least binary operation $f$ is called a \emph{weak near-unanimity} (WNU) if it satisfies  
$$f(y,x,\dots,x) \approx f(x,y,x,\dots,x) \approx \cdots \approx f(x,\dots,x,y).$$ 
\blue{Proposition~\ref{criterion} implies that another well-known characterisation of solvability of finite-domain CSPs in FP, namely the inability to express systems of equations over finite non-trivial Abelian groups,  fails for temporal CSPs (Corollary~\ref{nointerpretations}).
Recall the structures $\struct{E}_{\mathscr{G},k}$ from Definition~\ref{def:lineq}.
\begin{corollary}  \label{nointerpretations} $(\mathbb{Q};\mathrm{X})$ does not pp-construct $\struct{E}_{\mathscr{G},3}$ for any finite non-trivial Abelian group $\mathscr{G}$.
\end{corollary}
\begin{proof}
Suppose, on the contrary, that $(\mathbb{Q};\mathrm{X})$ pp-constructs $\struct{E}_{\mathscr{G},3}$ for a finite non-trivial Abelian group $\mathscr{G}$.
Then, by \blue{Lemma~\ref{pp_clone_minion}}, there exists a minion homomorphism $\xi\colon \Pol(\mathbb{Q};\mathrm{X}) \rightarrow \Pol(\struct{E}_{\mathscr{G},3})$. 
By Proposition~\ref{olsakterms}, $\Pol(\mathbb{Q};\mathrm{X})$ has (not necessarily idempotent) $(3+3)$-terms.
Since minion homomorphisms preserve satisfiability of minor conditions, it follows that $\Pol(\struct{E}_{\mathscr{G},3})$ also has such operations.
But, by definition, every polymorphism of $\struct{E}_{\mathscr{G},3}$ is idempotent.
This leads to a contradiction to Theorem~\ref{finitedomainsituation}.  Thus the statement of the corollary holds.
\end{proof}}

Despite their success in the  setting of finite-domain CSPs, finite minor conditions 
such as item (6) in Theorem~\ref{finitedomainsituation} 
are insufficient for classification purposes in the context of $\omega$-categorical CSPs.
\begin{proposition} \label{noh1finitelybounded} Let $\mathcal{L}$ be any logic at least as expressive as the existential positive fragment of $\FO$. Then there is no finite minor condition that would capture the expressibility of the CSPs of reducts of finitely bounded homogeneous structures in $\mathcal{L}$.
\end{proposition}

Proposition~\ref{noh1finitelybounded} is a consequence of the proof of Theorem~1.3 in \cite{bodirsky2019topology}. Both statements rely on the following result.

\begin{theorem}[\cite{cherlin1999universal,hubivcka2016homomorphism}]\label{thm:css}
For every finite set $\mathcal{F}$ of finite connected structures with a finite signature $\tau$, there exists a $\tau$-reduct $\mathrm{CSS}(\mathcal{F})$ of a finitely bounded homogeneous structure such that $\mathrm{CSS}(\mathcal{F})$ embeds precisely those finite $\tau$-structures which do not contain a homomorphic image of any member of $\mathcal{F}$.
\end{theorem}
For $\omega$-categorical structures, a statement that is stronger than in the conclusion of Proposition~\ref{noh1finitelybounded} has been shown in~\cite{gillibert2022symmetries}; however, the proof in~\cite{gillibert2022symmetries} does not apply to reducts of finitely bounded homogeneous structures in general. 

\begin{proof}[Proof of Proposition~\ref{noh1finitelybounded}] 
Suppose, on the contrary, that there exists such a condition $\mathcal{E}$. 
By the proof of Theorem~1.3 in \cite{bodirsky2019topology}, there exists a finite family $\mathcal{F}$ of finite connected structures with a finite signature $\tau$ such that $\Pol(\mathrm{CSS}(\mathcal{F}))$  does not satisfy  $\mathcal{E}$.
\blue{Recall the definition of the canonical conjunctive query $Q_\struct{A}$ from Section~\ref{section_CSPs}.}
The existential positive sentence $\phi_{\mathrm{CSS}(\mathcal{F})} \coloneqq \bigvee_{\struct{A}\in \mathcal{F}} Q_{\struct{A}}$ defines the complement of $\CSP(\mathrm{CSS}(\mathcal{F}))$. 
But then $\CSP(\mathrm{CSS}(\mathcal{F}))$ is expressible in $\mathcal{L}$, a contradiction. 
\end{proof}
The satisfiability of minor conditions in polymorphism clones is preserved under minion  homomorphisms, and
the satisfiability of sets of arbitrary identities in polymorphism clones is preserved under clone homomorphisms \cite{barto2018wonderland}. 
In Theorem~\ref{topology} we use the latter to show that, for $\Datalog$, Proposition~\ref{noh1finitelybounded} can be strengthened to sets of arbitrary identities.  
We hereby give a negative answer to a question from \cite{bodirsky2019topology} concerning the existence of a fixed set of identities that would capture $\Datalog$ expressibility for $\omega$-categorical CSPs.
This question was in fact already answered negatively in \cite{gillibert2022symmetries}, however,
our result also provably applies in the setting of finitely bounded homogeneous structures.
Recall the relation $S_{\elel}$ defined before  Lemma~\ref{RelationalBaseLl}. 
\begin{proposition}\label{temporal-structureS_datalog}  \hspace{0em} % hspace for hyperref to work
\begin{enumerate}  
\item \label{item:prop7.6_1}$(\mathbb{Q};\neq,\mathrm{S}_{\elel})$ does not pp-construct $(\mathbb{Q};\mathrm{R}_{\Min})$.
\item \label{item:prop7.6_2} There exists a uniformly continuous clone homomorphism  from $ \Pol(\mathbb{Q};\neq,\mathrm{S}_{\elel}) $ to $\Pol(\mathbb{Q};\mathrm{R}_{\Min}).  $   
\end{enumerate}
\end{proposition}
\begin{proof} 
For~\eqref{item:prop7.6_1}, suppose on the contrary that $(\mathbb{Q};\neq,\mathrm{S}_{\elel})$ pp-constructs $(\mathbb{Q};\mathrm{R}_{\Min})$. Since $\neq$ and $ \mathrm{S}_{\elel}$ are Ord-Horn definable,  $\CSP(\mathbb{Q};\neq,\mathrm{S}_{\elel})$ is expressible in $\Datalog$ by Theorem~\ref{datalogmainresult}.
Then, 
by \red{Theorem~\ref{REDUCTION}}, 
%by Theorem~\ref{theorem:reduction_pp_constructible}, 
$\CSP(\mathbb{Q};\mathrm{R}_{\Min})$ is expressible in $\Datalog$, which contradicts the fact that $\CSP(\mathbb{Q};\mathrm{R}_{\Min})$ is inexpressible in Datalog by Theorem~5.2 in \cite{bodirsky2010fast}.  Thus $(\mathbb{Q};\neq,\mathrm{S}_{\elel})$ does not pp-construct $(\mathbb{Q};\mathrm{R}_{\Min})$. 

\blue{For~\eqref{item:prop7.6_2},  we define $\xi\colon \Pol(\mathbb{Q};\neq,\mathrm{S}_{\elel}) \rightarrow \Pol(\mathbb{Q};\mathrm{R}_{\Min})$ as follows. 
Let $f\in \Pol(\mathbb{Q};\neq,\mathrm{S}_{\elel})$ be arbitrary and let $n$ be its arity.
As in the definition of essential parts, let $I_f $ be the set of all indices $i\in [n]$ such that $f$ depends on the $i$-th argument.
We define $\xi(f)$ as the map $(x_{1},\dots, x_{n}) \mapsto \Min \{x_{i}\mid i\in I_f\}$.
The set $I_f$ is non-empty because $\neq$ is not preserved by any constant operation.
Hence, $\xi$ is well-defined. 
We claim that $\xi$ is a clone homomorphism. 
Clearly, $\xi$ preserves arities and projections. 
Let $g_1,\dots, g_{n}$ be arbitrary $m$-ary operations from $\Pol(\mathbb{Q};\neq,\mathrm{S}_{\elel})$.
To show that $\xi(f(g_{1},\dots,g_{n}))=\xi(f)(\xi(g_{1}),\dots , \xi(g_{n}))$, we must show that 
\[
\min \{x_i \mid i\in I_{f(g_1,\dots, g_n)}\} = \min\{  \min\{x_i \mid i\in I_{g_j} \}\mid j\in I_{f} \}.
\]
Note that the right-hand side equals $\min\{ x_i \mid i\in \bigcup_{j\in I_f}  I_{g_j} \}$.
We show that $\bigcup_{j\in I_f}  I_{g_j} = I_{f(g_1,\dots, g_n)}$.
The backward direction of the set inclusion is trivial: if $i\notin I_{g_j}$ for every $j\in  I_f$, then clearly $i\notin I_{f(g_1,\dots, g_n)}$.
So suppose that $i\in I_{g_j}$ for some $j\in I_f$, i.e., $g_j$ depends on the $i$-th argument and $f$ depends on the $j$-th argument.
Recall from the proof of Proposition~\ref{olsaktermsordhorn} that, since $f$ depends on the $j$-th argument and preserves $S_{\elel}$, it is injective in the $j$-th argument.
Since $g_j$ depends on the $i$-th argument and $f$ is injective in the $j$-th argument, it follows that $f(g_1,\dots, g_n)$ depends on the $i$-th argument, i.e., $i\in I_{f(g_1,\dots, g_n)}$.
Finally, we show that $\xi$ is uniformly continuous. For every finite $B'\subseteq \mathbb{Q}$, we choose $A'\coloneqq B'$.
If $|B'|<2$, then clearly $\xi(f)$ and $\xi(g)$ agree on $B'$.
Otherwise, $f(x_1,\dots, x_{n}) = g(x_1,\dots, x_n)$ for all $x_1,\dots, x_{n} \in A'$ implies $I_{f}=I_g$ since an operation from $\Pol(\mathbb{Q};\neq,\mathrm{S}_{\elel})$
depends on the $i$-th argument iff it is injective in the $i$-th argument.
Thus, in this case, $\xi(f)$ and $\xi(g)$ also agree on $B'$.
This concludes the proof of~\eqref{item:prop7.6_2}.}
\end{proof}   
\begin{proof}[Proof of Theorem~\ref{topology}] 
Suppose, on the contrary, that there is a condition $\mathcal{E}$ which is preserved by uniformly continuous clone homomorphisms and captures expressibility of temporal CSPs in Datalog.
Since $\CSP(\mathbb{Q};\neq,\mathrm{S}_{\elel})$ is expressible in $\Datalog$, $\Pol(\mathbb{Q};\neq,\mathrm{S}_{\elel})$ must satisfy $\mathcal{E}$.
By Item~(2) in Proposition~\ref{temporal-structureS_datalog}, $\Pol(\mathbb{Q};\mathrm{R}_{\Min})$ satisfies $\mathcal{E}$ as well.
By assumption, $\CSP(\mathbb{Q};\mathrm{R}_{\Min})$ must be expressible in $\Datalog$, a contradiction to Theorem~5.2 in \cite{bodirsky2010fast}.  
\end{proof}

\subsection{\blue{New minor conditions}}
The expressibility of temporal and finite-domain CSPs in $\FP$ / $\FPC$ can be characterised by a family of minor conditions (Theorem~\ref{alternativefp}).
Inspired by~\cite{barto2019equations,barto2021algebraic}, we introduce general terminology to conveniently reason with minor conditions.
\blue{The following definition yields the same minor conditions as the paragraph above Example~3.13 in \cite{barto2021algebraic} up to the addition of implied equalities between terms and subsequent removal of the auxiliary terms on the left hand side.}
\begin{definition} \label{generalized}  Let $\struct{A}_1,\struct{A}_2$ be relational structures with a finite signature $\tau$. 
We define the minor condition $\mathcal{E}(\struct{A}_1,\struct{A}_2)$ as follows.
For every $R\in \tau$ and $\boldf{r}\in R^{\struct{A}_2}$, we introduce a unique $|R^{\struct{A}_1}|$-ary function symbol $g^{R}_{\boldf{r}}$.
Also, for every $R\in \tau$, we arbitrarily fix an enumeration $\boldf{x}_1,\dots, \boldf{x}_m$ of $R^{\struct{A}_1}$.
The elements of $A_1$ will be used as names for variables in the following.  
\blue{For every $a\in A_2$, if there exist $R,S\in\tau$ and  $\boldf{r}\in R^{\struct{A}_2}$,  $\boldf{s}\in S^{\struct{A}_2}$ such that 
$\boldf{r}\of{i} =a=\boldf{s}\of{j}$, then $\mathcal{E}(\struct{A}_1,\struct{A}_2)$ contains the \red{identity}  
\[ g^{R}_{\boldf{r}}(\boldf{x}_1\of{i},\dots, \boldf{x}_m\of{i}) \thickapprox g^{S}_{\boldf{s}}(\boldf{y}_1\of{j},\dots, \boldf{y}_n\of{j})\] 
where $\boldf{x}_1,\dots, \boldf{x}_m$ and $\boldf{y}_1,\dots, \boldf{y}_n$ are the fixed enumerations of $R^{\struct{A}_1}$ and $S^{\struct{A}_1}$, respectively. 
There are no other identities in $\mathcal{E}(\struct{A}_1,\struct{A}_2)$.
If  $\boldf{r}$ only appears in a single relation $R^{\struct{A}_2}$, then we set $g_{\boldf{r}}\coloneqq g^{R}_{\boldf{r}}$.}
\end{definition} 

\blue{The relation $1\textup{IN}3$ defined in the introduction can be generalised to  $$1\textup{IN}k\coloneqq \{ \boldf{t} \in \{0,1\}^k \mid \boldf{t}\of{i}=1 \text{ for exactly one } i\in [k] \}.$$
\begin{example} The existence of a $k$-ary WNU operation equals $\mathcal{E}(\struct{A}_1,\struct{A}_2)$ for 
\begin{align*}
\struct{A}_1 & \coloneqq (\{0,1\}; 1\textup{IN}k) 
\quad \text{and}\quad  \struct{A}_2   \coloneqq (\{a\}; \{(a,\dots, a)\}). \end{align*} 
\end{example}

The following proposition is not essential to the present article but demonstrates the magnitude of coverage of Definition~\ref{generalized}. 
We include it here since it was not mentioned in \cite{barto2021algebraic}.}

\blue{\begin{proposition}  For every finite minor condition $\mathcal{E}$, there exists a pair $\struct{A}_1,\struct{A}_2$ of finite structures in a finite signature $\tau$ such that $\mathcal{E}$ and $\mathcal{E}(\struct{A}_1,\struct{A}_2)$ are equivalent with respect to satisfiability in minions.  
\end{proposition}
\begin{proof} Let $\mathcal{E}$ be an arbitrary finite minor condition.
We define $A_1$ as the set of all variables occurring in $\mathcal{E}$.
Fix an arbitrary function symbol $f$ occurring in $\mathcal{E}$. 
Let $m$ be the arity of $f$.
We implicitly define the $k$-tuples $\boldf{x}_1,\dots, \boldf{x}_m$ by listing all the $f$-terms occurring in $\mathcal{E}$ in an arbitrary but fixed order: 
$  f(\boldf{x}_1\of{1},\dots,\boldf{x}_m\of{1}), \dots, 
f(\boldf{x}_1\of{k},\dots,\boldf{x}_m\of{k}).$
Without loss of generality, we may assume that $\boldf{x}_1,\dots,\boldf{x}_m$ are pairwise distinct, otherwise we can reduce the arity of $f$ without changing the satisfiability of $\mathcal{E}$ in minions.  
Now, for every such $f$, we require that $\tau$ contains a $k$-ary relation symbol $R_{f}$ which interprets in $\struct{A}_1$ as  $\{\boldf{x}_1,\dots,\boldf{x}_m\}.$
Let $\sim$ be the smallest equivalence relation on the terms which occur in $\mathcal{E}$ given by the identities in $\mathcal{E}$.
We define $A_{2}$ as the set of all equivalence classes of $\sim$.
For every function symbol $f$ which occurs in $\mathcal{E}$, the relation $ R_{f}^{\struct{A}_2}$ consists of a single tuple $\boldf{t}_{f}$ of  equivalence classes of $\sim$ of all $f$-terms which occur in $\mathcal{E}$, these equivalence classes appearing in $\boldf{t}_{f}$ in the fixed order from above.
It is easy to check that $\mathcal{E}(\struct{A}_1,\struct{A}_2)$ and $\mathcal{E}$ are identical up to reduction of arities through removal of non-essential arguments and adding additional identities which are implied by $\mathcal{E}$.
Thus, $\mathcal{E}$ and $\mathcal{E}(\struct{A}_1,\struct{A}_2)$ are equivalent. 
\end{proof}}

\blue{Recall from Corollary~\ref{nointerpretations} that $(\mathbb{Q};\mathrm{X})$ cannot pp-construct equations over any non-trivial finite Abelian group.
However, by Proposition~\ref{prop:Xequiv}, $\CSP(\mathbb{Q};\mathrm{X})$ can be reformulated as a decision problem for systems of \teal{mod-$2$} equations.
In particular, every homogeneous system of \teal{mod-$2$} equations of length $3$ without a non-trivial solution represents an unsatisfiable instance of $\CSP(\mathbb{Q};\mathrm{X})$.
A similar statement can also be made about some relations which are pp-definable in $(\mathbb{Q};\mathrm{X})$ (see the proof of Theorem~\ref{FPxorsat}).
%
%We will see that, 
Using this fact and the theory developed in \cite{barto2021algebraic}, it is possible to obtain many non-trivial minor conditions which are unsatisfiable in $\CSP(\mathbb{Q};\mathrm{X})$.
The question is then whether some of these conditions are satisfied in all temporal structures whose CSP is expressible in $\FP$.
The following theorem is a generalization of several observations from \cite{barto2021algebraic}. It allows us to reformulate the satisfiability of a minor condition in a given $\omega$-categorical structure as a statement about the existence of homomorphisms into pp-powers of the structure.}

\blue{\begin{theorem}\label{testing}  Let $\struct{B}$ be a countable $\omega$-categorical structure. 
For every pair $\struct{A}_1,\struct{A}_2$ of finite structures in a finite signature $\tau$, the following are equivalent:
\begin{enumerate}
\item\label{ctr_pp_powers_1} $\Pol(\struct{B}) \models \mathcal{E}(\struct{A}_1,\struct{A}_2)$.
\item \label{ctr_pp_powers_2} For every pp-power $\struct{C}$ of $\struct{B}$, if $\struct{A}_1 \rightarrow \struct{C}$, then $\struct{A}_2\rightarrow \struct{C}$.
\end{enumerate}  
\end{theorem}} 
\begin{definition} 
\label{def:dwnu}
We write $\mathcal{E}_{k,n}$ for 
$\mathcal{E}(\struct{A}_1,\struct{A}_2)$  if 
\begin{align*}
\struct{A}_1 & \coloneqq (\{0,1\}; 1\textup{IN}k) 
\quad \text{and}\quad  \struct{A}_2   \coloneqq ([n]; \{{\boldf{t}}\in [n]^{k}\mid {\boldf{t}}\of{1}<\cdots < {\boldf{t}}\of{k}\}). \end{align*}
The minor condition $\mathcal{E}_{k,n}$ is properly contained in a minor condition called~\emph{dissected WNUs} in~\cite{gillibert2022symmetries}. 
\end{definition}

\begin{example}
The minor condition $\mathcal{E}_{3,4}$  equals 
\[ \begin{array}{c} 
g_{(1,3,4)}(y,x,x)\approx g_{(1,2,4)}(y,x,x)\approx g_{(1,2,3)}(y,x,x), \\
g_{(2,3,4)} (y,x,x)\approx g_{(1,2,4)}(x,y,x)\approx g_{(1,2,3)}(x,y,x), \\
g_{(2,3,4)}(x,y,x)\approx g_{(1,3,4)}(x,y,x)\approx g_{(1,2,3)}(x,x,y), \\
g_{(2,3,4)}(x,x,y)\approx g_{(1,3,4)}(x,x,y)\approx g_{(1,2,4)}(x,x,y). \\
\end{array}  \]
By Theorem~\ref{testing}, $\Pol(\mathbb{Q};\mathrm{X})$ does not satisfy $\mathcal{E}_{3,4}$ because \[(\{0,1\}; 1\textup{IN}3)  \rightarrow (\mathbb{Q};\mathrm{X}) \quad\text{while}\quad  ([4]; \{(1,2,3),(1,2,4),(1,3,4),(2,3,4)\})\centernot{\rightarrow} (\mathbb{Q};\mathrm{X}).\]
\end{example} 	
Note that $\mathcal{E}_{k,n}$ is implied by the existence of a single $k$-ary WNU operation.
Also note that $\mathcal{E}_{k,n}$ implies $\mathcal{E}_{k,k+1}$ for all $n>k>1$. 
\blue{We first give a short proof of Theorem~\ref{testing}; 
then we restrict our attention to the family $(\mathcal{E}_{k,n})$.
%
%	Our aim is to make the proof of Theorem~\ref{alternativefp} as transparent as possible.
%
To prove the equivalence of \eqref{ctr_pp_powers_1} and \eqref{ctr_pp_powers_2} in Theorem~\ref{testing}, we need the following lemma.}
\begin{lemma} \label{triviality2}  \hspace{0em}
\begin{enumerate} 
\item   \label{ctr_triviality1} For any structure $\struct{D}$, if $\Pol(\struct{D}) $ satisfies $ \mathcal{E}(\struct{A}_1,\struct{A}_2)$ on the image of a homomorphism from $\struct{A}_1$ to $\struct{D}$, then there exists a homomorphism from $\struct{A}_2$ to $\struct{D}$.   
\item  	\label{ctr_triviality2} For every countable $\omega$-categorical structure $\struct{B}$, every finite $F\subseteq B$, and every finite  structure $\struct{A}_1$, there exists an $|F|^{|A_1|}$-dimensional pp-power $\struct{B}_F(\struct{A}_1)$ of $\struct{B}$ such that $\struct{A}_1 \rightarrow \struct{B}_F(\struct{A}_1)$ and 
\[
\struct{A}_2 \rightarrow \struct{B}_F(\struct{A}_1) \quad \mbox{ iff }\quad  \Pol(\struct{B}) \models \mathcal{E}(\struct{A}_1,\struct{A}_2)\mbox{ on } F. 
\]
\end{enumerate} 
\end{lemma}
The proof of Lemma~\ref{triviality2}(1) is similar to the proof of  Lemma~3.14(2) in \cite{barto2021algebraic}, and the proof of  Lemma~\ref{triviality2}(2)
is similar to the proof of Theorem~3.12(1) in \cite{barto2021algebraic}.
\blue{The notion of a \emph{free structure} plays a central role in~\cite{barto2021algebraic}.
The connection to our work is that the structure $\struct{B}_{F}(\struct{A}_1)$ in Lemma~\ref{triviality2}(2) is homomorphically equivalent to the free structure of the ``minion of local functions defined on $F$'' generated by $\struct{A}_1$.
However, since Lemma~\ref{triviality2}(2) has an elementary proof, it is not necessary to introduce the extra terminology from \cite{barto2021algebraic}.} 

\blue{\begin{proof}   

For \eqref{ctr_triviality1},  let $f\colon \struct{A}_1 \rightarrow \struct{D}$ be a homomorphism such that $\Pol(\struct{D})\models \mathcal{E}(\struct{A}_1,\struct{A}_2)$ on $f(A_1)$.  Consider the map $h \colon A_2 \to D$ defined as follows. If $a\in A_2$ does not appear in any tuple from a relation of $\struct{A}_2$, then we set $h(a)$ to be an arbitrary element of $D$.
\blue{If there exists $\boldf{r}\in R^{\struct{A}_2}$ such that $a=\boldf{r}\of{i}$,  then we take the operation $g^{R}_{\boldf{r}} \in \Pol(\struct{D})$ witnessing $\mathcal{E}(\struct{A}_1,\struct{A}_2)$ on $f(A_1)$ and set $ h(a)\coloneqq g^{R}_{\boldf{r}}(f(\boldf{x}_1),\dots, f(\boldf{x}_m))\of{i}$ where $\boldf{x}_1,\dots, \boldf{x}_m$ is the fixed enumeration of $R^{\struct{A}_1}$ from Definition~\ref{generalized}.
The map $h$ is well-defined: if $\boldf{r}\of{i} = a =\boldf{s}\of{j}$ for 
some $\boldf{r}\in R^{\struct{A}_2}$ and $\boldf{s}\in S^{\struct{A}_2}$, then 
\[g^{R}_{\boldf{r}}(f(\boldf{x}_1),\dots, f(\boldf{x}_m))\of{i} = g^{S}_{\boldf{s}}(f(\boldf{y}_1),\dots, f(\boldf{y}_n))\of{j}\] by the definition of $\mathcal{E}(\struct{A}_1,\struct{A}_2)$.
It remains to show that $h$ is a homomorphism.
Let $R\in \tau$ and $\boldf{r}\in R^{\struct{A}_2}$ be arbitrary. 
Since $f$ is a homomorphism, we have $f(\boldf{x}_i) \in R^{\struct{D}}$ for every $i\in [m]$.
Since $g^{R}_{\boldf{r}}$ is a polymorphism of $\struct{D}$, we have
$ h(\boldf{r}) = g^{R}_{\boldf{r}}(f(\boldf{x}_1),\dots, f(\boldf{x}_m)) \in R^{\struct{D}}$.}

For \eqref{ctr_triviality2}, we fix a finite $F\subseteq B$ and a finite $\tau$-structure $\struct{A}_1$.  
Let $f_1,\dots, f_{d}$ be an arbitrary fixed enumeration of all possible maps from $A_1$ to $F$.
Let $R$ be an arbitrary symbol from $\tau$,  and let $k\coloneqq \ar(R)$.
We fix an  arbitrary enumeration $\boldf{x}_1,\dots, \boldf{x}_m$ of $R^{\struct{A}_1}$. 
The domain of $\struct{B}_F(\struct{A}_1)$ is $B^{d}$, and, for every $R\in \tau$ with $k=\ar(R)$,  the relation $R^{\struct{B}_F(\struct{A}_1)}$ consists of all tuples $(\boldf{t}_1,\dots,\boldf{t}_k)\in (B^d)^k $ for which there exists $m$-ary $g\in \Pol(\struct{B})$ such that $(\boldf{t}_1\of{i},\dots,\boldf{t}_{k}\of{i}) =  g(f_i(\boldf{x}_1),\dots, f_i(\boldf{x}_m))$ for every $i\in [d]$.
Since $\Pol(\struct{B})$ is closed under taking compositions of operations, the relation 
$$ \{(\boldf{t}_1\of{1},\dots, \boldf{t}_1\of{d},\dots,\boldf{t}_k\of{1},\dots, \boldf{t}_k\of{d}) \mid (\boldf{t}_1,\dots,\boldf{t}_k) \in R^{\struct{B}_F(\struct{A}_1)}\}$$  is preserved by every polymorphism of $\struct{B}$. 
Hence, by Theorem~4 in \cite{bodirsky2006constraint}, it has a pp-definition in $\struct{B}$. 
Consequently, $\struct{B}_F(\struct{A}_1)$ is a $d$-dimensional pp-power of $\struct{B}$. 

Next we show that the map $h\colon A_1 \rightarrow B^d$ defined by $h(x)\coloneqq (f_1(x),\dots, f_d(x))$ is a homomorphism from $\struct{A}_1$ to $\struct{B}_F(\struct{A}_1)$.
Let $\boldf{t}$ be an arbitrary tuple from $R^{\struct{A}_1} $  for some $R\in \tau$, and let $k$ be the arity of $R$.
Then there exists $j\in [m]$ such that $\boldf{t}=\boldf{x}_j$ where $\boldf{x}_1,\dots, \boldf{x}_m$ is the fixed enumeration of $R^{\struct{A}_1}$ from the definition of $R^{\struct{B}_F(\struct{A}_1)}$.
Hence,  $h(\boldf{t})$ is of the form $(\boldf{t}_1,\dots ,\boldf{t}_k) \in (B^d)^k$ where $(\boldf{t}_1\of{i},\dots,\boldf{t}_{k}\of{i}) =  \pr_j(f_i(\boldf{x}_1),\dots, f_i(\boldf{x}_m))$ for every $i\in [d]$. It follows that $h(\boldf{t}) \in R^{\struct{B}_F(\struct{A}_1)}$.

It  remains to show that $\struct{A}_2 \rightarrow \struct{B}_F(\struct{A}_1)$  if and only if  $\Pol(\struct{B}) \models \mathcal{E}(\struct{A}_1,\struct{A}_2)$  on   $F$.  

``$\Leftarrow$'':  Suppose that $\Pol(\struct{B}) \models \mathcal{E}(\struct{A}_1,\struct{A}_2)$ on $F$. 
Consider the map $\xi$ which sends an $m$-ary operation  $f $ on $B$ to its component-wise action $(\boldf{t}_1,\dots, \boldf{t}_{m}) \mapsto f(\boldf{t}_1,\dots, \boldf{t}_{m})$ on $B^d$.
Since $\Pol(\struct{B}_F(\struct{A}_1))$ is a pp-power of $\struct{B}$, we have $\xi(f) \in \Pol(\struct{B}_F(\struct{A}_1))$ for every $f\in \Pol(\struct{B})$. 
Moreover, the images under $\xi$ of the operations witnessing $\Pol(\struct{B}) \models \mathcal{E}(\struct{A}_1,\struct{A}_2)$ on $F$ witness that  $\Pol(\struct{B}_F(\struct{A}_1)) \models \mathcal{E}(\struct{A}_1,\struct{A}_2)$ on $F^d$.
Note that $F^d$ contains the image of the homomorphism $h\colon \struct{A}_1 \rightarrow \struct{B}_F(\struct{A}_1), x \mapsto (f_1(x),\dots, f_d(x))$ from the previous paragraph. 
This means that $\Pol(\struct{B}_F(\struct{A}_1)) \models \mathcal{E}(\struct{A}_1,\struct{A}_2)$ on $h(A_1)$ and thus the requirements in item~\eqref{ctr_triviality1} are satisfied.
It now follows from item~\eqref{ctr_triviality1} that $\struct{A}_2 \rightarrow \struct{B}_F(\struct{A}_1)$. 

``$\Rightarrow$'': Suppose that there exists a homomorphism $h\colon \struct{A}_2 \rightarrow \struct{B}_F(\struct{A}_1)$.
Then, for every $R\in \tau$ and every $\boldf{t} \in R^{\struct{A}_2}$, by the definition of $R^{\struct{B}_F(\struct{A}_1)}$, there exists an $m$-ary $g_{\boldf{t}}^{R} \in \Pol(\struct{B})$ such that $h(\boldf{t})$ is of the form $(\boldf{t}_1,\dots ,\boldf{t}_k) \in (B^d)^k$ where $(\boldf{t}_1\of{i},\dots,\boldf{t}_{k}\of{i}) =  g_{\boldf{t}}^{R}(f_i(\boldf{x}_1),\dots, f_i(\boldf{x}_m))$ for every $i\in [d]$.
Now, for every $a\in A_{2}$ such that there exist $R,S\in\tau$ and  $\boldf{r}\in R^{\struct{A}_2}$,  $\boldf{s}\in S^{\struct{A}_2}$ with
$\boldf{r}\of{i} =a=\boldf{s}\of{j}$, we have $h(\boldf{r}\of{i})= h(\boldf{s}\of{j})$, where 
\begin{align*}
h(\boldf{r}\of{i}) = (g_{\boldf{r}}^{R}(f_1(\boldf{x}_1),\dots, f_1(\boldf{x}_m))\of{i},\dots, g_{\boldf{r}}^{R}(f_d(\boldf{x}_1),\dots, f_d(\boldf{x}_m))\of{i}), \\
h(\boldf{s}\of{j}) = (g_{\boldf{s}}^{S}(f_1(\boldf{x}_1),\dots, f_1(\boldf{x}_m))\of{j},\dots, g_{\boldf{s}}^{S}(f_d(\boldf{x}_1),\dots, f_d(\boldf{x}_m))\of{j}).
\end{align*} 
By the definition of $f_1,\dots, f_d$, the operations $g_{\boldf{t}}^{R}$ witness that $\Pol(\struct{B}) \models \mathcal{E}(\struct{A}_1,\struct{A}_2)$ on $F$. 
\end{proof}

\begin{proof}[Proof of Theorem~\ref{testing}]    
``\eqref{ctr_pp_powers_1}$\Rightarrow$\eqref{ctr_pp_powers_2}'':  
Suppose that $\Pol(\struct{B}) \models \mathcal{E}(\struct{A}_1,\struct{A}_2)$.  
If $\struct{C}$ is a pp-power of $\struct{B}$, then, by Proposition~\ref{InvAutPol}, each polymorphism of $\struct{B}$ represents a polymorphism of $\struct{C}$ through its component-wise action.
This means that  $\Pol(\struct{C}) \models \mathcal{E}(\struct{A}_1,\struct{A}_2)$. 
If additionally $\struct{A}_1 \rightarrow \struct{C}$, then it follows from Lemma~\ref{triviality2}\eqref{ctr_triviality1} that $\struct{A}_2 \rightarrow \struct{C}$.

``\eqref{ctr_pp_powers_2}$\Rightarrow$\eqref{ctr_pp_powers_1}'': For every finite $F\subseteq B$, the structure $\struct{B}_F(\struct{A}_1)$ from Lemma~\ref{triviality2}\eqref{ctr_triviality2} is a pp-power of $\struct{B}$.
Also, $\struct{A}_1 \rightarrow \struct{B}_F(\struct{A}_1)$.
By assumption, we have that $\struct{A}_2 \rightarrow \struct{B}_F(\struct{A}_1)$ for every finite $F\subseteq B$.
Using Lemma~\ref{triviality2}\eqref{ctr_triviality2}, we conclude that 
$\Pol(\struct{B}) \models \mathcal{E}(\struct{A}_1,\struct{A}_2)$ 
on $F$ for every finite $F\subseteq B$.
By a compactness argument, e.g.,  Lemma~9.6.10 in \cite{bodirsky2021complexity}, we have that $\Pol(\struct{B}) \models \mathcal{E}(\struct{A}_1,\struct{A}_2)$, which finishes the proof. 
\end{proof}}

\begin{example} \label{expl:non-triv} 
Let $\struct{A}_1$ and $\struct{A}_2$ be the structures from the definition of $\mathcal{E}_{k,n}$ (Definition~\ref{def:dwnu}).  Then Theorem~\ref{testing} implies that $\mathcal{E}_{k,n}$ is non-trivial: 
indeed, first note that there is no homomorphism from $\struct{A}_2$ to $\struct{A}_1$. 
Choose $\struct{C} \coloneqq \struct{B} \coloneqq \struct{A}_1$; then trivially $\struct{A}_1 \rightarrow \struct{C}$ 
but $\struct{A}_2 \not \to \struct{C}$, and hence $\Pol(\struct{B})$,  which only contains the projections,  does not satisfy $\mathcal{E}_{k,n}$. 
\end{example} 

\begin{example} 
We claim that the structures $(\mathbb{Q};\neq,\mathrm{S}_{\elel})$ and $(\mathbb{Q};\mathrm{R}_{\Min})$ are incomparable w.r.t.\ pp-constructibility.
We already know from Proposition~\ref{temporal-structureS_datalog}(1) that $(\mathbb{Q};\neq,\mathrm{S}_{\elel})$ does not pp-construct $(\mathbb{Q};\mathrm{R}_{\Min})$.
The reason there was that $\CSP(\mathbb{Q};\neq,\mathrm{S}_{\elel})$ is expressible in Datalog whereas $\CSP(\mathbb{Q};\mathrm{R}_{\Min})$ is not, and that pp-constructions preserve the expressibility of CSPs in Datalog.
This argument clearly cannot be used the other way around.
However, note that  $\Pol(\mathbb{Q};\mathrm{R}_{\Min})$ 
contains the ternary WNU operation $(x,y,z) \mapsto \min(x,y,z)$.
By Theorem~\ref{testing}, $\Pol(\mathbb{Q};\neq,\mathrm{S}_{\elel})$ does not contain a ternary WNU operation if and only if  there exists a pp-power  $\struct{C}$ of $(\mathbb{Q};\neq,\mathrm{S}_{\elel})$ such that $\struct{A}_1=(\{0,1\};1\mathrm{IN}3) \rightarrow \struct{C}$ and $\struct{A}_2 = ( \{a\};\{(a,a,a)\}) \centernot{\rightarrow} \struct{C}$.
And indeed,  such a pp-power exists:  the structure $\struct{C} \coloneqq (\mathbb{Q}; \{(x,y,z) \mid x\neq y \vee x < z\} )$ is even pp-definable in $\Pol(\mathbb{Q};\neq,\mathrm{S}_{\elel})$. 
Now if follows that  $(\mathbb{Q};\mathrm{R}_{\Min})$ does not pp-construct $(\mathbb{Q};\neq,\mathrm{S}_{\elel})$, otherwise \blue{Lemma~\ref{pp_clone_minion}} would yield a contradiction to the fact that $\Pol(\mathbb{Q};\neq,\mathrm{S}_{\elel})$ does not contain any ternary WNU operation.
\end{example}  
\begin{lemma} \label{noaffinecombinations}  For $n\geq 2$,
$\Pol(\struct{E}_{\mathbb{Z}_n,3})$ satisfies 	$\mathcal{E}_{k,k+1}$ if and only if  $\gcd(k,n)=1$. 
\end{lemma}
\begin{proof} 
First suppose that $\gcd(k,n)=1$. There exists $\lambda \in \mathbb{Z}_{n}$ such that $k\lambda = 1\bmod  n$. We write $g_{i}$ instead of $g_{{\boldf{t}}}$ for $\boldf{t} \in [k+1]^{k}$ with $ {\boldf{t}}\of{1}<\cdots < {\boldf{t}}\of{k}$ that omits $i$ as an entry. Then $\mathcal{E}_{k,k+1}$ is  witnessed by a set of $k$-ary WNU operations $g_1,\dots, g_{k+1}$ given by the affine combinations $g_{j}(x_1,\dots,x_k)\coloneqq \sum_{i=1}^{k} \lambda x_{i}$.  

Next, suppose that  $\Pol(\struct{E}_{\mathbb{Z}_n,3})$ satisfies 	$\mathcal{E}_{k,k+1}$.
It is well-known that, for every $k\geq 1$, the structure $\struct{E}_{\mathbb{Z}_n,k}$ has a pp-definition in $\struct{E}_{\mathbb{Z}_n,3}$.
In particular, the structure $\struct{C}\coloneqq (\mathbb{Z}_n;R)$ where $R\coloneqq \{\boldf{t} \in (\mathbb{Z}_n)^{k} \mid \sum_{i=1}^{k} \boldf{t}\of{i} = 1 \bmod n \}$ has a pp-definition in  $\struct{E}_{\mathbb{Z}_n,3}$.
Let $\struct{A}_1$ and $\struct{A}_2$ be as in Definition~\ref{def:dwnu}.
Clearly $\struct{A}_1 \rightarrow \struct{C}$. 
By Theorem~\ref{testing}, we have that $\struct{A}_2 \rightarrow \struct{C}$.
This means that the inhomogeneous system of $k+1$ \blue{mod-2} equations of the form $\sum_{j\in [k+1]\setminus \{i\}} x_{j}= 1 \bmod n$ has a solution.
By summing up the equations and subtracting $k$ on both sides, we get that 
$
k x_{1}+\cdots + k x_{k+1} -k =  k(x_1+\cdots + x_{k+1}-1)  = 1 \bmod  n. 
$
\blue{This can be the case only if  $\gcd(k,n)=1$.}
\end{proof}

Our next goal is the proof of Theorem~\ref{alternativefp}, which states that for temporal CSPs and finite-domain CSPs,  the condition $\mathcal{E}_{k,k+1}$ characterises expressibility in FP. 
For the proof of  we need to introduce some new polymorphisms of temporal structures. 
\blue{Recall the operation $\lex_k$ from Definition~\ref{definition:lex_k}.
\begin{definition}  \label{definition:PWNUs}	Let $k\in \mathbb{N}_{\geq 2}$.  The following definitions specify  $k$-ary operations on $\mathbb{Q}$: 
\begin{align*}     
\Min_{k}(\boldf{t}) \coloneqq & \, \Min\{\boldf{t}\of{1},\dots, \boldf{t}\of{k}  \}, \\ 
\med_{k}({\boldf{t}})\coloneqq & \,   \Min \{ \Max(\boldf{t}\of{i},\boldf{t}\of{j}) \mid i,j\in [k] \text{ and } i\neq j   \}, \\
\mi_{k}({\boldf{t}})\coloneqq  &  \lex_{k+2}(\Min_{k}({\boldf{t}}),\med_{k}(-\chi({\boldf{t}})),-\chi({\boldf{t}})), \\
\elel_{k}({\boldf{t}})\coloneqq & \lex_{k+2}(\Min_{k}({\boldf{t}}),\med_{k}({\boldf{t}}),{\boldf{t}}).
\end{align*}   
\end{definition}  
\begin{remark}Note that $\med_k(\boldf{t})$ equals the smallest value in $\boldf{t}$ if it appears in at least two entries, and otherwise the second smallest value in $\boldf{t}$.
Consequently, $\med_k$ is a \emph{near unanimity} (NU) operation, i.e., it satisfies the \red{identity} 
\[ f(y,x,\dots,x) \approx f(x,y,x,\dots,x) \approx \cdots \approx f(x,\dots,x,y)\approx x.\]
The involvement of a NU operation in the definitions of $\mi_k$ and $\elel_k$ is necessary for the proofs of Theorem~\ref{alternativefp} and  Proposition~\ref{PWNUs} to work, even though we do not mention this fact explicitly in the proofs.
One could also choose any other NU operation such that Proposition~\ref{newpoly} holds for the resulting operations $\mi_k$ and $\elel_k$.
\end{remark}}
\begin{proposition} \label{newpoly}   %\proofref{proof_PWNUs} 
Let $\struct{B}$ be a temporal structure and $k\in \mathbb{N}_{\geq 2}$.
\begin{enumerate}
\item \label{item:eq_1} If $\struct{B}$ is preserved by $\mi$, then also by $\mi_k$.
\item \label{item:eq_2} If $\struct{B}$ is preserved by $\elel$, then also by $\elel_k$.
\end{enumerate} 
\end{proposition}
\begin{proof}  \label{proof_newpoly}  
\blue{We first prove the following claim.
\begin{claim} \label{claim:comparisson_mi_elel} Let  $f\in \{\mi_k,\elel_k\}$, and let $\boldf{t}_1,\boldf{t}_2\in \mathbb{Q}^k$ be arbitrary.
\begin{itemize}
\item  	If $f(\boldf{t}_1)=f(\boldf{t}_2)$, then $\min_k(\boldf{t}_1)=\min_k(\boldf{t}_2)$ and $\chi(\boldf{t}_1)=\chi(\boldf{t}_2)$ if $f=\mi_k$, and $\boldf{t}_1=\boldf{t}_2$ if $f=\elel_k$.
\item If $f(\boldf{t}_1)<f(\boldf{t}_2)$, then $\min_k(\boldf{t}_1)\leq\min_k(\boldf{t}_2)$ and there exists $\ell \in [k]$ such that
\[\Min_k(\boldf{t}_1)=\boldf{t}_1\of{\ell}<\boldf{t}_2\of{\ell}.\]
\end{itemize}
\end{claim}
\begin{proof} The first part is a direct consequence of the following two facts:
\begin{itemize} 
\item $\mi_{k}$ compares $(\min_k(\boldf{t}_1),\chi(\boldf{t}_1))$ and $(\min_k(\boldf{t}_2),\chi(\boldf{t}_2))$  lexicographically, and 
\item $\elel_{k}$ compares $\boldf{t}_1$ and $\boldf{t}_2$  lexicographically.
\end{itemize}

We prove the second part by two separate case distinctions, starting with $\mi_k$.

\textit{Case 1: $\Min_{k}(\boldf{t}_{1})<\Min_{k}(\boldf{t}_{2})$.} 
Then we can choose any $\ell$ such that the $\ell$-th entry is minimal  in $\boldf{t}_{1}$. 

\textit{Case 2: $\Min_{k}(\boldf{t}_{1})=\Min_{k}(\boldf{t}_{2})$ and $ \med_{k}(-\chi(\boldf{t}_{1}))<\med_{k}(-\chi(\boldf{t}_{2})).$}
Then the smallest value in both tuples only appears in one entry of $\boldf{t}_{2}$ but in at least two entries of 
$\boldf{t}_{1}$.
Clearly, we can choose the index of one of these two entries as $\ell$.  

\textit{Case 3: $\Min_{k}(\boldf{t}_{1})=\Min_{k}(\boldf{t}_{2})$,  
$ \med_{k}(-\chi(\boldf{t}_{1}))=\med_{k}(-\chi(\boldf{t}_{2}))$, and $\lex_{k}(-\chi(\boldf{t}_{1}))<\lex_{k}(-\chi(\boldf{t}_{2}))$.}  
Then we define $\ell \in [k]$ as the leftmost index on which  $-\chi(\boldf{t}_{1})$ is pointwise smaller than $-\chi(\boldf{t}_{2})$.  

\textit{Case 1: $\Min_{k}(\boldf{t}_{1})<\Min_{k}(\boldf{t}_{2})$.} 
Then we can choose any $\ell$ such that the $\ell$-th entry is minimal  in $\boldf{t}_{1}$. 

\textit{Case 2: $\Min_{k}(\boldf{t}_{1})=\Min_{k}(\boldf{t}_{2})$ and $ \med_{k}(\boldf{t}_{1})<\med_{k}(\boldf{t}_{2}).$}
Then the smallest value in both tuples only appears in one entry of $\boldf{t}_{2}$ but in at least two entries of $\boldf{t}_{1}$.
Clearly, we can choose the index of one of these two entries as $\ell$. 

\textit{Case 3: $\Min_{k}(\boldf{t}_{1})=\Min_{k}(\boldf{t}_{2})$, $ \med_{k}(\boldf{t}_{1})=\med_{k}(\boldf{t}_{2})$, and $ \lex_{k}(\boldf{t}_{1})<\lex_{k}(\boldf{t}_{2}).$}
Then we define $\ell$ as the leftmost index on which  $\boldf{t}_{1}$ is pointwise smaller than $\boldf{t}_{3}$.   
\end{proof}

For~\eqref{item:eq_1}, by Lemma~\ref{RelationalBaseMi}, it suffices to prove that 
$\mi_{k}$ preserves $(\mathbb{Q};\mathrm{R}_{\mi},\mathrm{S}_{\mi},{\neq})$.

We start with the relation $\neq$.
Suppose that $\boldf{t}_1,\boldf{t}_2 \in \mathbb{Q}^k$ satisfy $\mi_{k}(\boldf{t}_{1} )=\mi_{k}(\boldf{t}_{2})$. 
Then, by Claim~\ref{claim:comparisson_mi_elel}, $\boldf{t}_{1}\of{\ell}=\boldf{t}_{2}\of{\ell}$ for every $\ell \in [k]$ such that the $\ell$-th entry of $\boldf{t}_1$ is minimal. We conclude that $\mi_{k}$ preserves $\neq$.

We continue with the relation $\mathrm{S}_{\mi}$.
Suppose that $\boldf{t}_1,\boldf{t}_2,\boldf{t}_3\in \mathbb{Q}^k$ satisfy
$(\mi_{k}(\boldf{t}_{1} ),\mi_{k}(\boldf{t}_{2} ),\mi_{k}(\boldf{t}_{3} )) \notin \mathrm{S}_{\mi}$.
Then $\mi_{k}(\boldf{t}_{1} )= \mi_{k}(\boldf{t}_{2})$ and $\mi_{k}(\boldf{t}_{1} )< \mi_{k}(\boldf{t}_{3})$.
By Claim~\ref{claim:comparisson_mi_elel}, $\min_k(\boldf{t}_1)=\min_k(\boldf{t}_2)$ and $\chi(\boldf{t}_1)=\chi(\boldf{t}_2)$, and there exists $\ell\in [k]$ such that 
$\min_k(\boldf{t}_1)=\boldf{t}_1\of{\ell}<\boldf{t}_3\of{\ell}$.
Then $(\boldf{t}_1\of{\ell}, \boldf{t}_2\of{\ell}, \boldf{t}_3\of{\ell})\notin \mathrm{S}_{\mi}$. 
We conclude that $\mi_k$ preserves $\mathrm{S}_{\mi}$.

Finally, consider the relation $\mathrm{R}_{\mi}$. 	 Suppose that $\boldf{t}_1,\boldf{t}_2,\boldf{t}_3\in \mathbb{Q}^k$ satisfy
$(\mi_{k}(\boldf{t}_{1} ),\mi_{k}(\boldf{t}_{2} ),\mi_{k}(\boldf{t}_{3} )) \notin \mathrm{R}_{\mi}$.
Then $\mi_{k}(\boldf{t}_{1} )\leq \mi_{k}(\boldf{t}_{2})$ and $\mi_{k}(\boldf{t}_{1} )< \mi_{k}(\boldf{t}_{3})$.
By Claim~\ref{claim:comparisson_mi_elel}, $\min_k(\boldf{t}_1)\leq\min_k(\boldf{t}_2)$ and there exists $\ell\in [k]$ such that  $\min_k(\boldf{t}_1)=\boldf{t}_1\of{\ell} <\boldf{t}_3\of{\ell}$.
Then $(\boldf{t}_1\of{\ell}, \boldf{t}_2\of{\ell}, \boldf{t}_3\of{\ell})\notin \mathrm{R}_{\mi}$. 
We conclude that $\mi_k$ preserves $\mathrm{R}_{\mi}$.

For \eqref{item:eq_2}, by Lemma~\ref{RelationalBaseLl}, it  suffices 
to show that $\elel_{k}$ preserves $(\mathbb{Q};\mathrm{R}_{\elel},\mathrm{S}_{\elel},\neq)$.

We start with the relation $\neq$.  Clearly, $\elel_k$ preserves $\neq$ because it is injective. 

We continue with the relation $\mathrm{S}_{\elel}$.
Suppose that $\boldf{t}_{1},\dots,\boldf{t}_4\in \mathbb{Q}^k$ satisfy  $(\elel_{k}(\boldf{t}_1),\dots,\elel_{k}(\boldf{t}_4)) \notin \mathrm{S}_{\elel}$.
Then $\elel_{k}(\boldf{t}_1) = \elel_{k}(\boldf{t}_2)$ and $\elel_{k}(\boldf{t}_4)<\elel_{k}(\boldf{t}_3)$.
By Claim~\ref{claim:comparisson_mi_elel}, we have $\boldf{t}_1=\boldf{t}_2$, and there exists $\ell \in [k]$ such that $\boldf{t}_4\of{\ell}<\boldf{t}_3\of{\ell}.$
Then $(\boldf{t}_1\of{\ell}, \dots, \boldf{t}_4\of{\ell})\notin \mathrm{S}_{\elel}$. 
We conclude that $\elel_k$ preserves $\mathrm{S}_{\elel}$.

Finally, consider the relation  $\mathrm{R}_{\elel}$.
Suppose that $\boldf{t}_{1},\boldf{t}_2,\boldf{t}_3\in \mathbb{Q}^k$ satisfy  $(\elel_{k}(\boldf{t}_1), \elel_{k}(\boldf{t}_2),\elel_{k}(\boldf{t}_3)) \notin \mathrm{R}_{\elel}$.
Then, without loss of generality, $\elel_{k}(\boldf{t}_1) \leq \elel_{k}(\boldf{t}_2)$ and $\elel_{k}(\boldf{t}_1)<\elel_{k}(\boldf{t}_3)$.
By Claim~\ref{claim:comparisson_mi_elel}, $\min_k(\boldf{t}_1)\leq\min_k(\boldf{t}_2)$ and there exists $\ell\in [k]$ such that  $\min_k(\boldf{t}_1)=\boldf{t}_1\of{\ell} <\boldf{t}_3\of{\ell}$.
Then $(\boldf{t}_1\of{\ell}, \boldf{t}_2\of{\ell}, \boldf{t}_3\of{\ell})\notin \mathrm{R}_{\elel}$. 
We conclude that $\elel_k$ preserves $\mathrm{R}_{\elel}$.} 
\end{proof}
\blue{Note that the proofs that $\mi_k$ preserves $\mathrm{R}_{\mi}$ and that $\elel_k$ preserves $\mathrm{R}_{\elel}$ are almost identical (but not entirely). The reason is that $R_{\elel}(x,y,z)$ is equivalent to $R_{\mi}(x,y,z) \wedge R_{\mi}(x,z,y)$, and  $\elel_k$ is essentially an injective version of $\mi_k$. By Proposition~\ref{InvAutPol}, $R_{\elel}$ is preserved by $\mi_k$.
However, it is not hard to see that $R_{\mi} $ is not preserved by any injective operation of arity $k\geq 2$, in particular not by $\elel_k$.}
\begin{proof}[Proof of Theorem~\ref{alternativefp}]
We start with the case where $\struct{B}$ is a temporal structure. 
Suppose that $\struct{B}$ is neither preserved by $\Min$, $\mi$, $\mx$, $\elel$, the dual of one of these operations, nor by a constant operation. Then, by Theorem~\ref{TCSPdichot}, $\struct{B}$ pp-constructs  $(\{0,1\}; 1\textup{IN}3)$. By \blue{Lemma~\ref{pp_clone_minion}}, there exists a minion homomorphism from $\Pol(\struct{B})$ to $\Pol(\{0,1\}; 1\textup{IN}3)$, the projection clone. By Theorem~\ref{testing},  for every $k\geq 2$, the condition $\mathcal{E}_{k,k+1}$ is non-trivial (see Example~\ref{expl:non-triv}).
Since minion homomorphisms preserve 
minor conditions such as $\mathcal{E}_{k,k+1}$  
it follows that $\Pol(\struct{B})$ cannot satisfy $\mathcal{E}_{k,k+1}$.  Next, we distinguish the subcases where $\struct{B}$ is a temporal structure preserved by one of the operations listed above.  

\textit{Case 1:	$\struct{B}$ is preserved by a constant operation.} Clearly,  $\mathcal{E}_{k,k+1}$ is witnessed by a set of $k$-ary constant operations for every $k \geq 1$.  

\textit{Case 2: $\struct{B}$ is preserved by $\Min$.} Then $\mathcal{E}_{k,k+1}$ is witnessed by a set of $k$-ary minimum operations for every $k\geq 2$.

\textit{Case 3: $\struct{B}$ is preserved by $\mx$.} By Theorem~\ref{mixedmx}, either $\struct{B}$
is preserved by $\Min$ or by a constant operation,
which are cases that we have already treated, or otherwise
$\struct{B}$ admits a pp-definition of $\mathrm{X}$.  
We claim that $\Pol(\mathbb{Q};\mathrm{X})$ does not satisfy $\mathcal{E}_{k,k+1}$ for every odd $k>1$. 
%
%Suppose, on the contrary, that  $\Pol(\mathbb{Q};\mathrm{X})$ satisfies $\mathcal{E}_{k,k+1}$ for some odd $k > 1$.
%
By Theorem~\ref{partialbasismx}, the temporal relation 
$  R^{\mx}_{[k],k} =  \{ {\boldf{t}}\in \mathbb{Q}^{k} \mid \sum_{\ell=1}^{k}\chi({\boldf{t}})\of{\ell} = 0 \bmod 2 \} $
is preserved by $\mx$.
By Lemma~\ref{lemma:RelationalBaseMx}, $ R^{\mx}_{[k],k}$ is pp-definable in $(\mathbb{Q};\mathrm{X})$.
\blue{Let $\struct{A}_1$ and $\struct{A}_2$ be as in Definition~\ref{def:dwnu}.}
Since $k$ is odd, there exists a homomorphism from $\struct{A}_1$ to $ (\mathbb{Q}; R^{\mx}_{[k],k})$.
However, there exists no homomorphism from $\struct{A}_2$ to $ (\mathbb{Q}; R^{\mx}_{[k],k}) $.
This is because the homogeneous system of $k+1$ \blue{mod-2} equations of the form $\sum_{j\in [k+1]\setminus \{i\}} x_{j}= 0 \bmod 2$ has no non-trivial solution, which means that $\struct{A}_2$ has no free set by Lemma~\ref{freeMX}.
Hence, Theorem~\ref{testing} implies that 
$\Pol(\mathbb{Q};\mathrm{X})$ does not satisfy $\mathcal{E}_{k,k+1} = \mathcal{E}(\struct{A}_1,\struct{A}_2)$.

\textit{Case 4: $\struct{B}$ has $\mi$ as a polymorphism.}  We proceed similarly as in the proof of Proposition~4.10 in \cite{barto2019equations}, but using  our Theorem~\ref{testing}. 
\blue{Let $\struct{A}_1$ and $\struct{A}_2$ be as in Definition~\ref{def:dwnu} for a fixed $k\geq 3$.} 
Let $\struct{C}$ be an arbitrary $d$-dimensional pp-power of $\struct{B}$ with the same signature as $\struct{A}_1$ for which there exists a homomorphism $h\colon \struct{A}_1 \rightarrow \struct{C}$. 
We denote \red{the unique relation of $\struct{C}$} (of arity $k$) by $R$.
Since $\struct{B}$ is preserved by $\mi$, by Proposition~\ref{newpoly}, it is also preserved by $\mi_k$.  
\blue{For every $i\in [k]$, we define $\boldf{t}_{i}\coloneqq (h(0),\dots,h(0),h(1),h(0),\dots, h(0))$ where $h(1)$ appears in the $i$-th entry.
%
%For $i\in [k+1]$, the tuples $\boldf{t}_{k+1,i}$ are defined in the same way with the only difference that they have arity $k+1$.
%
Since $\struct{C}$ is a pp-power of $\struct{B}$ and $\boldf{t}_{i} \in R$ for every $i\in [k]$, it follows from Proposition~\ref{InvAutPol} that $ \mi_{k}(\boldf{t}_{1},\dots,  \boldf{t}_{k}) \in R$, where $\mi_k$ acts doubly component-wise:
\[
\mi_{k}(\boldf{t}_{1},\dots,  \boldf{t}_{k})  = \left( \begin{array}{c} \mi_{k}(h(1),\dots, h(0)) \\ \vdots \\  \mi_{k}(h(0), \dots, h(1))
\end{array}
\right) = \left( \begin{array}{c} \left( \begin{array}{c} \mi_{k}(h(1)\of{1},\dots, h(0)\of{1}) \\ \vdots \\  \mi_{k}(h(1)\of{d}, \dots, h(0)\of{d})
\end{array}
\right) \\ \vdots \\  \left( \begin{array}{c} \mi_{k}(h(0)\of{1},\dots, h(1)\of{1}) \\ \vdots \\  \mi_{k}(h(0)\of{d}, \dots, h(1)\of{d})
\end{array}
\right) 
\end{array}
\right).
\]
Let $\boldf{t}\coloneqq (h(0),\dots, h(0))$.
We claim that $ \mi_{k+1}(\boldf{t}_{1},\dots, \boldf{t}_{i-1},\boldf{t},\boldf{t}_{i},\dots,  \boldf{t}_{k}) \in R$ for every $i\in [k+1]$.
Note that, for all $x,y,x',y' \in \mathbb{Q}$ and all $i,j\in [k]$,  we have 
\[  \mi_{k}(x,\dots,x,\underset{i}{y},x,\dots,x)<\mi_{k}(x',\dots,x',\underset{j}{y'},x',\dots,x') \label{eq:order} \tag{$*$}\]
iff one of the following holds:
\begin{itemize}
\item $\Min(x,y)<\Min(x',y')$, or
\item $\Min(x,y)=\Min(x',y')=x<x'$, or
\item $\Min(x,y)=\Min(x',y')$, $x<y$, $x'<y'$, and $j<i$, or
\item $\Min(x,y)=\Min(x',y')$, $y<x$, $y'<x'$, and $i<j$.
\end{itemize} 
In all four cases, the order in \eqref{eq:order} remains invariant if we replace $\mi_{k}$ with $\mi_{k+1}$ and expand the inputs on the left- and the right-hand side by the variables $x$ and $x'$, respectively, inserted into the same argument with an index from $[k+1]$.
This means that, when viewed as $(k\cdot d)$-dimensional tuples over $\mathbb{Q}$, 
$ \mi_{k}(\boldf{t}_{1},\dots,  \boldf{t}_{k}) $ and  
$ \mi_{k+1}(\boldf{t}_{1},\dots, \boldf{t}_{i-1},\boldf{t},\boldf{t}_{i},\dots,  \boldf{t}_{k})$ have the same order type for every $i\in [k+1]$.
Then, by the homogeneity of $(\mathbb{Q};<)$, for every $i\in [k+1]$, there exists $\alpha_i \in \Aut(\mathbb{Q};<)$ such that $\mi_{k+1}(\boldf{t}_{1},\dots, \boldf{t}_{i-1},\boldf{t},\boldf{t}_{i},\dots,  \boldf{t}_{k}) = \alpha_i \circ \mi_{k}(\boldf{t}_{1},\dots,  \boldf{t}_{k}) $.
Now the claim follows from Proposition~\ref{InvAutPol}.
Consider the map $g\colon[k+1]\rightarrow \mathbb{Q}^d$ given by $g(i)\coloneqq \mi_{k+1}(h(0),\dots, h(1),\dots, h(0))$ where $h(1)$ appears in the $i$-th entry.}
%
%
%        \red{Note that $h(1,\dots, i-1,i+1,\dots, k+1) = \mi_{k+1}(\boldf{t}_{1},\dots, \boldf{t}_{i-1},\boldf{t},\boldf{t}_{i},\dots,  \boldf{t}_{k}) \in R$ for every $i\in [k+1]$. Hence, $h$ is a homomorphism from $\struct{A}_2$ to $\struct{C}$.}
%
\red{Note that $g(1,\dots, i-1,i+1,\dots, k+1) = \mi_{k+1}(\boldf{t}_{1},\dots, \boldf{t}_{i-1},\boldf{t},\boldf{t}_{i},\dots,  \boldf{t}_{k})$ for every $i\in [k+1]$. Hence, $g$ is a homomorphism from $\struct{A}_2$ to $\struct{C}$.}
Since $\struct{C}$ was chosen arbitrarily, it follows from Theorem~\ref{testing} that $\Pol(\struct{B}) \models \mathcal{E}_{k,k+1}$ for all $k\geq 3$.

\textit{Case 5: $\struct{B}$ has $\elel$ as a polymorphism.} We repeat the strategy above using $\elel_{k}$ instead of $\mi_k$.
For all $x,y,x',y' \in \mathbb{Q}$ and all $i,j\in [k]$,  we have 
\[  \elel_{k}(x,\dots,x,\underset{i}{y},x,\dots,x)<\elel_{k}(x',\dots,x',\underset{j}{y'},x',\dots,x')  \]
iff one of the following holds:
\begin{itemize}
\item $\Min(x,y)<\Min(x',y')$, or
\item $\Min(x,y)=\Min(x',y')$ and $x<x'$, or
\item $\Min(x,y)=\Min(x',y')$, $x=x'<y'$, and $j<i$, or
\item $\Min(x,y)=\Min(x',y')$, $x=x'$, $i=j$, and $y<y'$. 
\end{itemize}

The cases~2-5 can be dualized in order to obtain witnesses for $\mathcal{E}_{k,k+1}$ for $k\geq 3$ in the cases where $\struct{B}$ is preserved by $\Max$, $\dual\,\mi, $ $\dual \,\elel$, and show that $\Pol(\struct{B})$ does not satisfy $\mathcal{E}_{k,k+1}$ for odd $k>1$ if it admits a pp-definition of $-\mathrm{X}$.  

If $\struct{B}$ is a finite structure, then 
$\CSP(\struct{B})$ is in FP / FPC if
and only if $\struct{B}$ does not pp-construct $\struct{E}_{\mathbb{Z}_n,3}$ for every $n\geq 2$ by  Theorem~\ref{finitedomainsituation}.
Then the claim follows from Lemma~\ref{noaffinecombinations}. 
\end{proof}

We can confirm the condition for expressibility in FP from Theorem~\ref{alternativefp} also for the structures $\mathrm{CSS}(\mathcal{F})$ from Theorem~\ref{thm:css}. 

\begin{theorem} \label{alternativefpGWNUS}  
Let $\mathcal{F}$ be a finite set of finite connected structures with a fixed finite signature,
and let $\struct{B} \coloneqq \mathrm{CSS}(\mathcal{F})$. 
Then 
\begin{enumerate} 
\item $\CSP(\struct{B})$ is expressible in $\FP$ / $\FPC$, and 
\item $\Pol(\struct{B})$ satisfies $\mathcal{E}_{k,k+1}$ for all but finitely many $k\in \mathbb{N}$.  
\end{enumerate} 
\end{theorem}   
\begin{proof}
$\CSP(\mathrm{CSS}(\mathcal{F}))$ is expressible in $\FP$ because it is even expressible in existential positive first-order logic.  
$\Pol(\mathrm{CSS}(\mathcal{F}))$ satisfies $\mathcal{E}_{k,k+1}$ for all but finitely many arities, because it contains WNU operations for all but finitely many arities by Lemma~5.4 in \cite{bodirsky2019topology}. 
\end{proof}

\subsection{\blue{Failures of known pseudo minor conditions}}
In the context of infinite-domain $\omega$-categorical CSPs, most classification results are formulated using  \emph{pseudo minor conditions} \cite{barto2019equations} which extend minor conditions by outer unary operations, i.e., they are of the form
$$e_{1}\circ f_{1}(x^{1}_{1},\dots,x^{1}_{n_{1}})  \approx \cdots \approx e_{k}\circ f_{k}(x^{1}_{k},\dots,x^{k}_{n_{k}}).$$
For instance, the following generalization of a WNU operation was used in \cite{bodirsky2021complexity} to give an alternative classification of the computational complexity of TCSPs.
An at least binary operation $f\in \Pol(\struct{B})$ is called \emph{pseudo weak near-unanimity} (pseudo-WNU) if there exist $e_{1},\dots,e_{n}\in \End(\struct{B})$ such that $$e_{1} \circ f(x,\dots,x,y) \approx e_{2} \circ f(x,\dots,x,y,x) \approx\cdots \approx e_{n} \circ f(y,x,\dots,x).$$

\begin{theorem}[\cite{bodirsky2021complexity}] \label{alternative}   
Let $\struct{B}$ be a temporal structure. Then either $\struct{B}$ has a pseudo-WNU polymorphism and $\CSP(\struct{B})$ is in P, or  $\struct{B}$ pp-constructs all finite structures and $\CSP(\struct{B})$ is NP-complete.  
\end{theorem}
It is natural to ask whether pseudo minor conditions can be used to formulate a generalization of the  3-4 WNU condition from item 7 of Theorem~\ref{finitedomainsituation} that would capture the expressibility in FP for the CSPs of reducts of finitely bounded homogeneous structures. 
One such generalization was considered in \cite{bodirsky2016reducts}. Proposition~\ref{criterion} shows that the criterion provided by Theorem~8 in \cite{bodirsky2016reducts} is insufficient in general. 

\begin{proposition}
\label{criterion} There exist pseudo-WNU polymorphisms $f,g$ of  $(\mathbb{Q};\mathrm{X})$ that satisfy $$f(x, x, y) \approx g(x, x, x, y).$$ 
\end{proposition}

%\begin{lemma}[see Lemma 3 in \cite{pinsker2021canonical}]\label{lem:construct}
%	Let $\struct{B}$ be an $\omega$-categorical structure and $f_1,g_1,\dots,f_n,g_n \in \Pol(\struct{B})$ where $f_i$ and $g_i$ have the same arity $k_i$. If for every $i \in \{1,\dots,n\}$ and all finite $F \subseteq B^{k_i}$ we have $f_i(\boldf{t}) = g_i(\boldf{t})$ for all $\boldf{t} \in F$, then there are $e,e_1,\dots,e_n \in \End(\struct{B})$ such that 
%	$\Pol(\struct{B})$ satisfies
%	$$e(f_i(x_1,\dots,x_{k_i})) \approx e_i(g_i(x_1,\dots,x_{k_i})).$$
%\end{lemma}
%
\begin{proof}[Proof of Proposition~\ref{criterion}] Consider the terms 
\begin{align*} f(x_{1},x_{2},x_{3})\coloneqq  &\ \mx(\mx(x_{1},x_{2}),\mx(x_{2},x_{3})), \\
g(x_{1},x_{2},x_{3},x_{4})\coloneqq  &\ \mx(\mx(x_{1},x_{2}),\mx(x_{3},x_{4})).
\end{align*}
\blue{It is easy to see that, for all distinct $x,y\in \mathbb{Q}$, we have}
\begin{align*} f(x,x,y)=f(y,x,x)=\ & \alpha^{2}(\Min(x,y)), \\  f(x,y,x)=\ & \beta(\alpha(\Min(x,y))), \\ 
g(x,x,x,y)= \cdots =g(y,x,x,x)=\ &   \alpha^{2}(\Min(x,y)), 
\end{align*}
where $\alpha,\beta$ are as in the definition of $\mx$.  
We also have $f(x,x,x)= \beta^{2}(x)= g(x,x,x,x)$ for all $x\in \mathbb{Q}$.
\blue{Clearly, $g$ is a WNU, and $f(x,x,y) = g(x,x,x,y)$ holds for all $x,y\in \mathbb{Q}$.
It remains to show that $f$ is a pseudo-WNU.
Our argumentation here is similar \red{to the one in} the proof of Proposition~\ref{olsakterms}, and in fact even simpler because we do not need any of the witnessing unary operations to be equal.
Let $S$ be a finite subset of $\mathbb{Q}$. We define $\struct{B}_1$ and $\struct{B}_2$ as the substructures of $(\mathbb{Q};<)$ on $\{f(y,x,x) \mid x,y\in S\}$ and $\{f(x,y,x) \mid x,y\in S\}$, respectively.
We claim that $h(f(y,x,x))\coloneqq f(x,y,x)$ is an isomorphism from $\struct{B}_1$ to $\struct{B}_2$.

To show that $h$ is well-defined, we must to show that $f(y,x,x)= f(y',x',x')$ implies $y=y'$ and $x=x'$ for all $x,y,x',y' \in S$.
If $x=y$ and $x'=y'$ or $x\neq y$ and $x'\neq y'$, then this follows directly from the fact that $\alpha$ and $\beta$ preserve $<$.
If $x=y$ and $x'\neq y'$, then $f(y,x,x) = f(y',x',x')$ implies $\beta^2(x)=\alpha^2(\min(x',y'))$.  By Lemma~\ref{claim:comparisson}, this is impossible.
Thus, in this case, the claim that $f(y,x,x)= f(y',x',x')$ implies $y=y'$ and $x=x'$  holds trivially.
The remaining case  $x\neq y$ and $x' =y'$ is analogous to the previous one.

Next we show that $h$ is a homomorphism, i.e., that $f(y,x,x)< f(y',x',x')$ implies that $f(x,y,x)<f(x',y',x')$ for all $x,y,x',y' \in S$.
Again, if $x=y$ and $x'=y'$ or $x\neq y$ and $x'\neq y'$, then this follows directly from the fact that $\alpha$ and $\beta$ preserve $<$.
In the remaining two cases we need to additionally use Lemma~\ref{claim:comparisson}.

\emph{Case 1:}  $x=y$ and $x'\neq y'$. 
Suppose that $f(y,x,x) < f(y',x',x')$. Then $\beta^2(x)<\alpha^2(\min(x',y'))$, which implies $x<\min(x',y')$ by Lemma~\ref{claim:comparisson}.
Then $\beta^2(x)<\beta \circ \alpha(\min(x',y'))$ by Lemma~\ref{claim:comparisson} and because $\beta$ preserves $<$. 
Thus $f(x,y,x) < f(x',y',x')$.

\emph{Case 2:}  $x\neq y$ and $x' =y'$.  
Suppose that $f(y,x,x) < f(y',x',x')$. Then $ \alpha^2(\min(x,y))< \beta^2(x')$, which implies $x\leq \min(x',y')$ by Lemma~\ref{claim:comparisson}.
Then $\beta \circ \alpha (\min(x,y))< \beta^2(x')$ by Lemma~\ref{claim:comparisson} and because $\beta$ preserves $<$. 
Thus $f(x,y,x) < f(x',y',x')$.

Hence, $h$ is an isomorphism. Since $(\mathbb{Q};<)$ is homogeneous, there exists $\eta\in \Aut(\mathbb{Q};<)$ extending $h$.
By Lemma~\ref{lem:construct}, there exist $e'$ and $e$ such that $e' \circ f (x,x,y) = e \circ f (x,y,x)$ holds for all $x,y\in \mathbb{Q}$.
Note that then also $e' \circ f (y,x,x) = e \circ f (x,y,x)$ holds for all $x,y\in \mathbb{Q}$. This completes the proof.}
%	
%	
%	
%	% 
%	%For $f$, note that $\alpha$ and $\beta$ are both self-embeddings of $(\mathbb{Q};<)$ by the definition of $\mx$. '
%	By the computation above, for all $x,y,x',y'\in \mathbb{Q}$, we have 
%	%
%	\[f(x,x,y)=f(y,x,x)<f(x',x',y')=f(y',x',x') \text{ if and only if } f(x,y,x)<f(x',y',x').\] 
%	%	
%	This means that for every finite $S\subseteq \mathbb{Q}$, the finite substructures of $(\mathbb{Q};<)$ on the images of $f(x,x,y)$, $f(x,y,x)$ and $f(y,x,x)$ with inputs restricted to $S^{3}$ are isomorphic.
%	%	
%	Since 	$(\mathbb{Q};<)$ is homogeneous, there exist  $\alpha',\beta',\gamma'\in \Aut(\mathbb{Q};<)$ such that for all $x,y\in S$
%	\[ \alpha' \circ f(x,x,y)= \beta' \circ f(x,y,x) = \gamma' \circ f(y,x,x). \]  
%	% 
%	Lemma~\ref{lem:construct} then implies that $f$  is a pseudo-WNU.
%		%	
%		Clearly, $g$ is a WNU, and $f(x,x,y) = g(x,x,x,y)$, which completes the proof.
\end{proof}
Another characterisation of finite-domain CSPs in FP that   fails for temporal CSPs is the existence of pseudo-WNU polymorphisms for all but finitely many arities (Proposition~\ref{PWNUs}).
\blue{\begin{definition} For $k\in \mathbb{N}_{\geq 2}$, the $k$-ary $\mx$ operation on $\mathbb{Q}$ is defined by  
\[
\mx_{k}(\boldf{t}) \coloneqq  \begin{cases}\mx(\boldf{t}) & \text{if  } k=2, \\
\mx\!\big(\!\mx_{k-1}({\boldf{t}}\of{1},\dots,{\boldf{t}}\of{k-1}),   \mx_{k-1}({\boldf{t}}\of{2},\dots,{\boldf{t}}\of{k})  \big) & \text{if } k>2. 
\end{cases} 
\]
\end{definition}}
By definition, every structure preserved by $\mx$ is also preserved by $\mx_k$ for $k\geq 3$.
Recall the operations  $\min_k, \mi_k$, and $\elel_k$ from Definition~\ref{definition:PWNUs}.
\begin{proposition}	\label{PWNUs} For every $k\geq 3$,  $\Min_{k}$, $\mx_{k}$, $\mi_{k}$, and $\elel_{k}$ are pseudo-WNU operations.
\end{proposition}  
\blue{\begin{proof} \label{proof_PWNUs}
The statement trivially holds for $\Min_{k}$.  
To show the statement for the operation   $\mx_{k}$, we first prove the following claim. Let \red{$\alpha,\beta \in \End(\mathbb{Q};<)$} from the definition of $\mx$.
For every $k\geq 2$ and $i\in \mathbb{Z}$, we define $f_{k,i}(x,y)\coloneq \mx_k(x_1,\dots, x_k)$ where, for every $j\in [k]$, $x_j$ equals $y$ if $j=i$ and $x$ otherwise. Clearly, if $i\notin [k]$, then $f_{k,i}(x,y) = \mx_k(x,\dots, x)$.
\begin{claim} \label{comparing}    
For every $k\geq 2$ and $i\in \mathbb{Z}$, there exist $h_{k,i,1},\dots, h_{k,i,k-2}\in \{\alpha,\beta\}$ such that, for all distinct $x,y\in \mathbb{Q}$, \[f_{k,i}(x,y) = 
\begin{cases}
h_{k,i,k-2}\circ \cdots \circ h_{k,i,1}\circ\alpha (\min(x,y)) & \text{if } i\in [k], \\
\beta^{k-1}(x) & \text{otherwise}.
\end{cases} \]
\end{claim} 
\begin{proof}[Proof of Claim~\ref{comparing}] We prove the statement by induction on $k$.\footnote{\blue{It is possible to prove stronger statements, e.g., that $h_{k,i,j}=\alpha$ for some $j\in [k-2]$ implies $h_{k,i,j'}=\alpha$ for every $j'\in [j]$, but these are irrelevant for the proof of Proposition~\ref{PWNUs}.}} 
In the base case $k=2$, the statement is trivially true by the definition of $\mx$.
In the induction step, suppose that the statement holds for $k-1$.
By the definition of $\mx_k$, $f_{k,i}(x,y) = \mx(f_{k-1,i}(x,y), f_{k-1,i-1}(x,y))$ for all distinct $x,y\in \mathbb{Q}$.
We have the following four cases:

\emph{Case 1:} $i\in [k-1]$ and $i-1\in [k-1]$. Then there exist  $h_{k-1,i,1},\dots, h_{k-1,i,k-3} \in \{\alpha,\beta\}$ and $h_{k-1,i-1,1},\dots, h_{k-1,i-1,k-3}\in \{\alpha,\beta\}$  such that  $f_{k-1,i}(x,y)= h_{k-1,i,k-3}\circ \cdots \circ h_{k-1,i,1} \circ \alpha (\min(x,y)) $ and $f_{k-1,i-1}(x,y)= h_{k-1,i-1,k-3}\circ \cdots \circ h_{k-1,i-1,1} \circ \alpha (\min(x,y)) $. 
By a repeated application of Lemma~\ref{claim:comparisson} and the fact that \red{$\alpha$ and $\beta$ preserve $<$}, we get that either $f_{k-1,i}(x,y) < f_{k-1,i-1}(x,y)$, $f_{k-1,i}(x,y) = f_{k-1,i-1}(x,y)$, or $f_{k-1,i}(x,y) > f_{k-1,i-1}(x,y)$ holds uniformly for all distinct $x,y\in \mathbb{Q}$.
If $f_{k-1,i}(x,y) < f_{k-1,i-1}(x,y)$, then $f_{k,i}(x,y) = \alpha(f_{k-1,i}(x,y))$. We can set $h_{k,i,k-2}\coloneqq \alpha$ and $h_{k,i,j}\coloneqq h_{k-1,i,j}$ for every $j\in [k-3]$.
If $f_{k-1,i}(x,y) = f_{k-1,i-1}(x,y)$, then $f_{k,i}(x,y) = \beta(f_{k-1,i}(x,y))$. We can set $h_{k,i,k-2}\coloneqq \beta$ and $h_{k,i,j}\coloneqq h_{k-1,i,j}$ for every $j\in [k-3]$.
If $f_{k-1,i}(x,y) > f_{k-1,i-1}(x,y)$,  then $f_{k,i}(x,y) = \alpha(f_{k-1,i-1}(x,y))$. We can set $h_{k,i,k-2}\coloneqq \alpha$ and $h_{k,i,j}\coloneqq h_{k-1,i-1,j}$ for every $j\in [k-3]$.

\emph{Case 2:} $i\in [k-1]$ and $i-1\notin [k-1]$. Then there exist $h_{k-1,i,1},\dots, h_{k-1,i,k-3} \in \{\alpha,\beta\}$ such that $f_{k-1,i}(x,y)= h_{k-1,i,k-3}\circ \cdots \circ h_{k-1,i,1}\circ \alpha (\min(x,y)) $, and $f_{k-1,i-1}(x,y)=  \beta^{k-2}(x) $.
By a repeated application of Lemma~\ref{claim:comparisson} and the fact that \red{$\alpha$ and $\beta$ preserve $<$}, we get that  $f_{k-1,i}(x,y) < f_{k-1,i-1}(x,y)$ holds uniformly for all distinct $x,y\in \mathbb{Q}$.
Then $f_{k,i}(x,y) = \alpha(f_{k-1,i}(x,y))$. We can set $h_{k,i,k-2}\coloneqq \alpha$ and $h_{k,i,j}\coloneqq h_{k-1,i,j}$ for every $j\in [k-3]$.

\emph{Case 3:} $i\notin [k-1]$ and $i-1\in [k-1]$. This case is analogous to the one above. We have $f_{k,i}(x,y) = \alpha(f_{k-1,i-1}(x,y))$. Hence we can set $h_{k,i,k-2}\coloneqq \alpha$ and $h_{k,i,j}\coloneqq h_{k-1,i-1,j}$ for every $j\in [k-3]$.

\emph{Case 4:} $i\notin [k-1]$ and $i-1\notin [k-1]$. Then $f_{k-1,i}(x,y)=  \beta^{k-2}(x) $ and $f_{k-1,i-1}(x,y)=  \beta^{k-2}(x) $.
Then $f_{k,i}(x,y)=  \beta^{k-1}(x) $.
\end{proof}  

Let $k\geq 3$, and let $S$ be an arbitrary finite subset of $\mathbb{Q}$.
For a fixed $i\in [k-1]$, let $\struct{B}_1$ be the substructure of $(\mathbb{Q};<)$ on $\{\mx_k(x,\dots,x,y,x,\dots, x) \mid x,y\in S\}$ where $y$ appears in the $i$-th entry, and let $\struct{B}_2$ be the substructure of $(\mathbb{Q};<)$ on $\{\mx_k(x,\dots,x,y) \mid x,y\in S\}$.
Consider the map $h(\mx_k(x,\dots,x,y,x,\dots, x))\coloneqq \mx_k(x,\dots,x,y)$. 
It follows from Claim~\ref{comparing}, Lemma~\ref{claim:comparisson}, and the fact that \red{$\alpha$ and $\beta$ preserve $<$} that $h$ is a well-defined isomorphism from $\struct{B}_1$ to $\struct{B}_2$.
Since $(\mathbb{Q};<)$ is homogeneous, there exists $\eta\in \Aut(\mathbb{Q};<)$ extending $h$.
Now the statement that $\mx_k$ is a pseudo-WNU operation follows directly from Lemma~\ref{lem:construct}.

\begin{figure}[t]
\begin{center}                       
\begin{tikzpicture} 
\node (k) at (0,0) [] {};

\node (k1l) at (-0.5,0.25) [] {};
\node (k1r) at (0.5,0.25)  [] {};
\node  (k2l) at (-1,0.5)  [] {};
\node (k2r)  at (1,0.5) [] {};

\node  (k2m) at (0,0.75)  [] {};

\node (x1) at (-3,1.5)  [] {$x_1$};
\node (x2) at (-2,1.5)  [] {$x_2$};
\node (xk1) at (2,1.5)  [] {$x_{k}$};
\node  (xk) at (3,1.5) [] {$x_{k+1}$};

\node (xk2) at (1,1.5)  [] {$x_{k-1}$};
\node  (x3) at (-1,1.5) [] {$x_3$};
\node  (xm) at (0,1.5) [] {$\cdots$};

\path (k) edge[draw,-, shorten <= -0.14cm] node[above]  {} (x1); 
\path (k) edge[draw,-  , shorten <= -0.14cm] node[above]  {} (xk);

\path (k1l) edge[draw,- , shorten <= -0.14cm] node[above]  {} (xk1); 
\path (k1r) edge[draw,-  , shorten <= -0.14cm] node[above]  {} (x2);

\path (k2l) edge[draw,- , shorten <= -0.14cm] node[above]  {} (xk2); 
\path (k2r) edge[draw,- , shorten <= -0.14cm] node[above]  {} (x3); 

\node  (k) at (0,-0.25)  [] {$f$};
\node (1l)  at (-0.6,0) [] {$f'$}; 
\node (1r)  at (0.6,0) [] {$f''$};

\end{tikzpicture}  
\caption{\label{decomposition}			 An illustration of the term $f=\mx_{k+1}(x_1,\dots,x_{k+1})$ and its subterms.} 
\end{center} 
\end{figure} 

Next, we consider the operations $\mi_k$ and $\elel_k$. 
We proceed exactly as with the operation $\mx_k$, using homogeneity of $(\mathbb{Q};<)$ and Lemma~\ref{lem:construct}. 
The argument boils down to showing that, for both $f\in \{\mi_k,\elel_k\}$, all $x,y,x',y'\in \mathbb{Q}$, and every $i\in [k]$,
we have  
\[f(x,\dots,x,\underset{i}{y},x,\dots,x)<f(x',\dots,x',\underset{i}{y'},x',\dots,x') \quad\text{iff}\quad f(x,\dots,x,y)<f(x',\dots,x',y').\]
This is the case because both the left- and the right-hand side are true if and only if one of the following cases applies.
For  $f=\mi_k$: 
\begin{itemize}
\item $\Min(x,y)<\Min(x',y')$, or 
\item $\Min(x,y)=\Min(x',y')=x<x'$;
\end{itemize} 
for  $f=\elel_k$:   
\begin{itemize}
\item $\Min(x,y)<\Min(x',y')$, or
\item $\Min(x,y)=\Min(x',y')$ and $x<x'$, or
\item $\Min(x,y)=\Min(x',y')$, $x=x'$, and $y<y'$.
\end{itemize}  
This finishes the proof.\end{proof} }

\subsection{\blue{New pseudo minor conditions}}
We present a new candidate for an algebraic condition given by pseudo minor identities that could capture the expressibility in FP for CSPs of reducts of finitely bounded homogeneous structures.
Let $\mathcal{E}'_{k,k+1}$ be the pseudo minor condition obtained from $\mathcal{E}_{k,k+1}$ by replacing each $g_{\boldf{t}}$ in $\mathcal{E}_{k,k+1}$ with  $e_{\boldf{t}}\circ g$ where $e_{\boldf{t}}$ is unary and $g$ has arity $k$. 
For instance, up to further renaming the function symbols, $\mathcal{E}'_{3,4}$ is the following condition:  
\begin{align*} 
b &\circ g(y,x,x)\approx c \circ g(y,x,x)\approx d \circ g(y,x,x), \\
a & \circ g(y,x,x)\approx c \circ g(x,y,x)\approx d \circ g(x,y,x), \\
a & \circ g(x,y,x)\approx b \circ g(x,y,x)\approx d \circ g(x,x,y), \\
a & \circ g(x,x,y)\approx b \circ g(x,x,y)\approx c \circ g(x,x,y). 
\end{align*}   
Note that $\mathcal{E}'_{k,k+1}$ implies the non-trivial minor condition $\mathcal{E}_{k,k+1}$.
%
%\footnote{\blue{\red{The fact that $\mathcal{E}'_{k,k+1}$ implies $\mathcal{E}_{k,k+1}$ is the reason} why we have not mentioned \emph{model-complete cores} when proposing \red{$\mathcal{E}'_{k,k+1}$}, as a candidate, as opposed to the conjectures in \cite{barto2019equations} concerning  the pseudo-Siggers identity.
%
%	In general, it is not the case that a non-trivial pseudo minor condition implies a non-trivial minor condition, not even w.r.t.\ satisfiability in polymorphism clones of reducts of finitely bounded structures.
%
%    For example, Proposition~10.4.1 in \cite{bodirsky2021complexity} provides a concrete reduct of a finitely bounded structure which has a pseudo-Siggers polymorphism and simultaneously pp-constructs $(\{0,1\}; 1\textup{IN}3)$.
%
%  It typically only makes sense to consider pseudo minor conditions under additional assumptions.}} 
%  
Also note that the existence of a $k$-ary WNU operation implies $\mathcal{E}'_{k,k+1}$.
However, $\mathcal{E}'_{k,k+1}$ is in general not implied by the existence of a $k$-ary pseudo-WNU operation: 
$\Pol(\mathbb{Q};\mathrm{X})$ contains a $k$-ary pseudo-WNU operation for every $k\geq 2$ (Proposition~\ref{PWNUs}) but does not satisfy 
$\mathcal{E}_{k,k+1}$ for every odd $k\geq 3$.
The latter statement follows from Theorem~\ref{alternativefp}, because 
$\CSP({\mathbb Q};\mathrm{X})$ is not in FP (Theorem~\ref{FPxorsat}).
\blue{The proof of Theorem~\ref{alternativefpGWNUS} shows that the statement of the theorem remains true if we replace $\mathcal{E}_{k,k+1}$ with $\mathcal{E}'_{k,k+1}$.
Theorem~\ref{alternativefp} also remains true under such replacement; in its proof, we can simply use Lemma~\ref{lem:construct} instead of Theorem~\ref{testing} for the cases where $\struct{B}$ is a temporal structure preserved by $\mi$, $\elel$, or their duals.}
\begin{corollary}  \label{pseudoalternativefpGWNUS}  Let $\struct{B}$ be as in \blue{Theorem~\ref{alternativefp} or Theorem~\ref{alternativefpGWNUS}}. The following are equivalent.
\begin{enumerate} 
\item $\CSP(\struct{B})$ is expressible in $\FP$ / $\FPC$. 
\item $\Pol(\struct{B})$ satisfies $\mathcal{E}'_{k,k+1}$ for all but finitely many $k\in \mathbb{N}$.  
\end{enumerate}
\end{corollary} 

\red{We would also like to point out that 
under fairly general assumptions on $\struct{B}$ it is possible to algorithmically test wether $\Pol(\struct{B})$ satisfies the pseudo-minor condition $\mathcal{E}'_{k,k+1}$; 
more specifically, if $\struct{B}$ is a homogeneous finitely bounded Ramsey structure (see~\cite{bodirsky2021complexity} for the definition of Ramsey structures and how to appropriately represent such structures on a computer; all first-order expansions of $(\mathbb{Q};<)$ satisfy the given conditions) then this can be shown as in the proof of Theorem 11.6.7 in~\cite{bodirsky2021complexity}. Such decidability results are not known for minor conditions such as $\mathcal{E}_{k,k+1}$.}
%, since $\mathcal{E}'_{k,k+1}$ can be viewed as a \emph{stable clone formula} 

%
%If the expressibility in FP for CSPs of reducts of finitely bounded homogeneous structures is captured by such pseudo h1 conditions, then this would mean that the situation there differs from the finite (Theorem~\ref{finitedomainsituation})  only in the way how the endomorphisms of a given structure interact with each other on top of the  WNU identities.
%  

\section{Open questions} 
We have completely classified expressibility of temporal CSPs in the logics FPC, FP, and Datalog. Our results show that all of the characterisations known for finite-domain CSPs fail for temporal CSPs. However, we have also seen new universal-algebraic conditions that characterise expressibility in FP simultaneously for finite-domain CSPs and for temporal CSPs. 
\blue{It is an open problem to find such conditions in the more general setting of the infinite-domain tractability conjecture:
\begin{conjecture}[\cite{barto2019equations}] \label{conjecture}
Let $\struct{B}$ be a reduct of a finitely bounded
homogeneous structure. Then one of the following holds.
\begin{enumerate} 
\item $\struct{B}$ pp-constructs $(\{0,1\};1\mathrm{IN}3)$ (and consequently, $\CSP(\struct{B})$ is NP-complete);
\item $\struct{B}$ is solvable in polynomial time.
\end{enumerate} 
\end{conjecture} 
\noindent For $\struct{B}$ as in Conjecture~\ref{conjecture}, we ask the following questions:
\begin{enumerate} 
\item Is  $\CSP(\struct{B})$ inexpressible in FP whenever  $\Pol(\struct{B})$ does not satisfy the minor condition $\mathcal{E}_{k,k+1}$   for all but finitely many  $k \geq 2$? 
\item We ask the previous question for  
the pseudo-minor condition $\mathcal{E}_{k,k+1}'$ instead of $\mathcal{E}_{k,k+1}$. 
\item  If $\CSP(\struct{B})$ is in FPC, is it also in FP? 
To the best of our knowledge, this could hold for CSPs in general, even without the additional assumptions on $\struct{B}$.
\end{enumerate}

It is also an open question whether FP extended with rank operators modulo all prime numbers (FPR) captures Ptime for finite-domain CSPs.
Another important candidate is  choiceless polynomial time (CPT)~\cite{blass2002polynomial}. 
We propose to extend both candidates to the setting of Conjecture~\ref{conjecture}:
\begin{enumerate}[resume]
\item Does FPR/CPT capture Ptime for CSPs of reducts of finitely bounded homogeneous structures?    
\end{enumerate}  
In the case of CPT, it is not even clear how to show inexpressibility for $\CSP(\{0,1\};1\mathrm{IN}3)$. In the case of FPR, the inexpressibility of $\CSP(\{0,1\};1\mathrm{IN}3)$ follows from Theorem~\ref{REDUCTION} and the results in \cite{gradel2019rank} because $(\{0,1\};1\mathrm{IN}3)$ pp-constructs all finite structures.} 

\appendix

\let\oldaddcontentsline\addcontentsline% Store \addcontentsline
\renewcommand{\addcontentsline}[3]{}% Make \addcontentsline a no-op

\begin{acks}
	We thank Wied Pakusa
	%, who was a co-author of the conference version of the present article, 
	for the idea how to use the multipedes construction in our context \red{and the referees of the conference version and the journal version for the many very useful suggestions.}  
\end{acks}

\bibliographystyle{ACM-Reference-Format}
\bibliography{final}

\end{document}